\documentclass[a4paper,11pt]{amsart}
\usepackage{amsfonts,amsmath,amsthm,amssymb,stmaryrd}
\newtheorem{theorem}{Theorem}[section]
\newtheorem{lemma}[theorem]{Lemma}
\newtheorem{Proposition}[theorem]{Proposition}
\newtheorem{Remark}[theorem]{Remark}
\numberwithin{equation}{section}

\usepackage[left=1 in, right=1 in,top=1 in, bottom=1 in]{geometry}

\def\Lbrack{\left \llbracket}
\def\Rbrack{\right \rrbracket}

\DeclareMathOperator{\diverge}{div}

\providecommand{\norm}[1]{\left\Vert#1\right\Vert}
\def\ls{\lesssim}

\def\dt{\partial_t}
\def\H{{}_0H^1}
\def\Hd{(\H)^\ast}
\def\na{\nabla_\ast}
\def\bna{\bar{\nabla}_\ast}
\def\rj{\Lbrack \rho \Rbrack}

\def\RRvert2{\right \vert\! \right\vert}
\def\Lvert3{\left \vert\!\left\vert\!\left\vert}
\def\Rvert3{\right \vert\!\right\vert\!\right\vert}

\title[The viscous surface-internal wave problem]{The viscous surface-internal wave problem:\\ nonlinear Rayleigh-Taylor instability}

\author{Yanjin Wang}
\address{
School of Mathematical Sciences\\
Xiamen University\\
Fujian 361005, China} \email[Y. J. Wang]{yanjin$\_$wang@xmu.edu.cn}
\thanks{Y. J. Wang was partially supported by National Natural Science Foundation of China-NSAF (No. 10976026)}

\author{Ian Tice}
\address{
Division of Applied Mathematics\\
Brown University \\
Providence, RI 02912, USA } \email[I. Tice]{tice@dam.brown.edu}
\thanks{I. Tice was supported by an NSF Postdoctoral Research Fellowship.}

\begin{document}

\begin{abstract}
We consider the free boundary problem for two layers of immiscible, viscous, incompressible fluid in a uniform gravitational field, lying above a rigid bottom in a three-dimensional horizontally periodic setting.  The effect of surface tension is either taken into account at both free boundaries or neglected at both.  We are concerned with the Rayleigh-Taylor instability, so we assume that the upper fluid is heavier than the lower fluid.  When the surface tension at the free internal interface is below a critical value, which we identify, we establish that the problem under consideration is nonlinearly unstable.
\end{abstract}

\maketitle


\section{Introduction}
\subsection{Formulation of the problem in Eulerian coordinates}

In this paper we study the viscous surface-internal wave problem, which concerns the dynamics of two layers of  distinct, immiscible, viscous, incompressible fluid lying above a flat rigid bottom and below an atmosphere of constant pressure.  We assume that a uniform gravitational field points in the direction of the rigid bottom.  This is a free boundary problem since both the upper surface, where the upper fluid meets the atmosphere, and the internal interface, where the upper and lower fluids meet, are free to evolve in time with the motion of the fluids.  We assume that the upper fluid is heavier than the lower fluid, which makes the problem susceptible to the so-called Rayleigh-Taylor instability.  We will consider the nonlinear Rayleigh-Taylor instability both with and without taking into account the effect of surface tension on the free surfaces.

We will assume that the two fluids occupy the moving domain $\Omega(t)$, which we take to be three-dimensional and horizontally periodic. One fluid $(+)$ fills the upper domain
\begin{equation}\label{omega_plus}
\Omega_+(t)=\{y\in   \mathrm{T}^2\times \mathbb{R}\mid \eta_-(t,y_1,y_2)<y_3< 1 +\eta_+(t,y_1,y_2)\},
\end{equation}
and the other fluid $(-)$ fills the lower domain
\begin{equation}
\Omega_-(t)=\{y\in  \mathrm{T}^2\times \mathbb{R}\mid  -b <y_3<\eta_-(t,y_1,y_2)\}.
\end{equation}
Here we have written the periodic horizontal cross-section as $\mathrm{T}^2=(2\pi L_1\mathbb{T}) \times (2\pi L_2\mathbb{T})$, where $\mathbb{T} = \mathbb{R}/\mathbb{Z}$ is the usual 1--torus and $L_1,L_2>0$ are fixed constants.  We assume that $b >0$ is the constant depth of the rigid bottom but that the surface functions $\eta_\pm$ are unknowns in the problem.  The surface $\Gamma_+(t) = \{y_3= 1  + \eta_+(t,y_1,y_2)\}$ is the moving upper boundary of $\Omega_+(t)$, $\Gamma_-(t) = \{y_3=\eta_-(t,y_1,y_2)\}$ is the moving internal interface between the two fluids, and $\Sigma_b = \{y_3=-b \}$ is the fixed lower boundary of $\Omega_-(t)$.

The fluids are described by their velocity and pressure functions, which are given for each $t\ge0$ by $v_\pm (t,\cdot):\Omega_\pm (t)\rightarrow \mathbb{R}^3$ and $\bar{p}_\pm (t,\cdot):\Omega_\pm (t)\rightarrow \mathbb{R}$, respectively. The fluid motion is governed by the incompressible Navier-Stokes equations:
\begin{equation}\label{e1}
\left\{\begin{array}{ll}\rho_\pm \partial_t v_\pm  + \rho_\pm  v_\pm
\cdot\nabla v_\pm +\diverge S(\bar{p}_\pm,v_\pm)= -g\rho_\pm e_3&\hbox{in
}\Omega_\pm (t)
\\ \diverge v_\pm =0&\hbox{in }\Omega_\pm (t)
\\\partial_t\eta_\pm  =v_{3,\pm}-v_{1,\pm}\partial_{y_1}\eta_\pm
-v_{2,\pm}\partial_{y_2}\eta_\pm &\hbox{on }
\Gamma_\pm(t)
\\S(\bar{p}_+,v_+)n_+=p_{e}n_+-\sigma_+ H_+n_+
\quad &\hbox{on }\Gamma_+(t)\\
v_+=v_-,\quad S(\bar{p}_+,v_+)n_-=S(\bar{p}_-,v_-)n_-+ \sigma_-
H_-n_-
 &\hbox{on }\Gamma_-(t) \\v_-=0 &\hbox{on
}\Sigma_b.\end{array}\right.
\end{equation}
Here the positive constants $\rho_\pm $ denote the densities of the respective fluids, $g>0$ is the acceleration of gravity, $e_3=(0,0,1)^T$,  and $-g \rho_\pm e_3$ is the gravitational force.  Throughout the paper we will write $\rj := \rho_+ - \rho_-$. Since we are interested in the Rayleigh--Taylor instability,  we assume that the upper fluid is heavier than the lower fluid, i.e.,
\begin{equation}
\rho_+>\rho_-\Longleftrightarrow\rj >0.
\end{equation}
The quantity $S(\bar{p}_\pm,v_\pm):= \bar{p}_\pm I-\mu_\pm \mathbb{D}(v_\pm)$ is known as the viscous stress tensor, where $\mu_\pm $ are the viscosities of the respective fluids and we have written $I$ for the $3\times3$ identity matrix and $\mathbb{D}(v_\pm)_{ij}=\partial_jv_{i,\pm} +\partial_iv_{j,\pm}$ for twice the velocity deformation tensor.  We have written  $\diverge S(\bar{p}_\pm,v_\pm)$ for the vector with $i^{th}$ component $\partial_j  S(\bar{p}_\pm,v_\pm)_{ij}$; an easy computation shows that if $\diverge{u_\pm}=0$, then $\diverge S(\bar{p}_\pm,v_\pm) = \nabla \bar{p}_\pm - \mu_\pm \Delta u_\pm$.   The constant $p_{e}$ is the atmospheric pressure, and we take $\sigma_\pm\ge 0$ to be the constant coefficients of surface tension.  We will assume that either $\sigma_\pm =0$ or $\sigma_\pm >0$, i.e. we do not allow only one free surface to experience the effect of surface tension.   In this paper, we let $\nabla_\ast$ denote the horizontal gradient, $\diverge_\ast$ denote the horizontal divergence and $\Delta_\ast$ denote the horizontal Laplace operator. Then the unit normal of $\Gamma_\pm(t)$ (pointing up), $n_\pm$,  is given by
\begin{equation}
n_\pm=\frac{(-\nabla_\ast\eta_\pm,1)}
{\sqrt{1+|\nabla_\ast\eta_\pm|^2}},
\end{equation}
and  $H_\pm$, twice the mean curvature of the surface $\Gamma_\pm(t)$, is given by the formula
\begin{equation}
H_\pm=\diverge_\ast\left(\frac{\nabla_\ast\eta_\pm}
{\sqrt{1+|\nabla_\ast\eta_\pm|^2}}\right).
\end{equation}
The third equation in \eqref{e1} is called the kinematic boundary condition since it implies that the free surfaces are advected with the fluids. Notice that on $\Gamma_-(t)$, the continuity of velocity implies that $v_+ = v_-$, which means that it is the common value of $v_\pm$ that advects the interface. For a more physical description of the equations \eqref{e1} and the boundary conditions in \eqref{e1}, we refer to \cite{3WL}.

Note that the constant $1$ appearing above in the definition of $\Omega_+(t)$ and $\Gamma_+(t)$ is the equilibrium height of the upper fluid, i.e. the height of a solution with $v_\pm =0$, $\eta_\pm =0$, etc.  It is not a loss of generality for us to assume this value is unity.  Indeed, if we were to replace the $1$ in $\Omega_+(t)$ and $\Gamma_+(t)$ by an arbitrary equilibrium height $L_3 >0$, then a standard scaling argument would allow us to rescale the coordinates and unknowns in such a way that $L_3 >0$ would be replaced by $1$.  Of course, in doing so we would have to multiply $L_1,$ $L_2,$ $\mu_\pm$, $\sigma_\pm$, $g$ and $b$ by positive constants. In our above formulation of the problem we assume that this procedure has already been done.

To complete the statement of the problem, we must specify initial conditions. We suppose that the initial surfaces $\Gamma_\pm(0)$ are given by the graphs of the functions $\eta_\pm(0)=\eta_{0,\pm}$, which yield the open sets $\Omega_\pm(0)$ on which we specify the initial data for the velocity, $v_\pm(0)=v_{0,\pm}:\ \Omega_\pm(0) \rightarrow  \mathbb{R}^3$. We will assume that $1+\eta_{0,+}>\eta_{0,-}>-b $ on $\mathrm{T}^2$ and that $\eta_{0,\pm}, v_{0,\pm}$ satisfy certain compatibility conditions (see the discussion before Theorem \ref{LWP}) required for the local well-posedness of the problem.  We assume that the initial surface functions satisfy the ``zero-average'' condition
\begin{equation}
\label{zero0}\int_{\mathrm{T}^2}\eta_{0,\pm}=0.
\end{equation}
It is not a  loss of generality to assume this.  Indeed, if \eqref{zero0} is not satisfied, then we could use the Galilean invariance of the equations  and the flatness of the rigid lower boundary to restore the zero-average condition (see the introduction of \cite{WTK} for details).  For sufficiently regular solutions to the problem, the condition \eqref{zero0} persists in time, i.e.
\begin{equation}
\label{zerot}\int_{\mathrm{T}^2}\eta_{\pm}(t)=0 \text{ for } t\ge 0.
\end{equation}
Indeed, from the equations $\eqref{e1}_2$, $\eqref{e1}_3$, and $\eqref{e1}_6$   we obtain
\begin{equation}
\frac{d}{dt} \int_{\mathrm{T}^2} \eta_-=\int_{\mathrm{T}^2}\partial_t\eta_-=\int_{\Gamma_-(t)}v_-\cdot
n_-=\int_{\Omega_-(t)} \diverge{v_-}=0,
\end{equation}
and
\begin{equation}
\frac{d}{dt} \int_{\mathrm{T}^2} \eta_+ = \int_{\mathrm{T}^2}\partial_t\eta_+ = \int_{\Gamma_+(t)}v_+ \cdot
n_+  = \int_{\Omega_+(t)} \diverge{v_+} + \int_{\Gamma_-(t)} v_- \cdot n_- = 0.
\end{equation}

To simplify the equations we introduce the indicator function $\chi$ and denote
\begin{eqnarray}\label{uni-notation}
\left\{\begin{array}{ll}
v=v_+{\chi_{\Omega_+(t)}}+v_-{\chi_{\Omega_-(t)}},
\quad
\bar{p}=\bar{p}_+{\chi_{\Omega_+}(t)}+\bar{p}_-{\chi_{\Omega_-(t)}},
\\\rho=\rho_+{\chi_{\Omega_+(t)}}+\rho_-{\chi_{\Omega_-(t)}},\quad
\mu=\mu_+{\chi_{\Omega_+(t)}}+\mu_-{\chi_{\Omega_-(t)}}
,\end{array}\right.
\end{eqnarray}
on $\Omega(t) = \Omega_+(t) \cup \Omega_-(t)$.  We similarly denote quantities on $\Gamma(t): = \Gamma_+(t) \cup \Gamma_-(t)$, etc.  We define the modified pressure through $\tilde{p}=\bar{p}+\rho gy_3-p_e-\rho_+g$. Then the modified equations are
\begin{equation}\label{form1}
\left\{\begin{array}{ll}
\rho\partial_t v + \rho v\cdot\nabla v+\nabla \tilde{p}=\mu\Delta v\quad&\text{in }\Omega(t)
\\ \diverge{v}=0\quad&\text{in }\Omega(t)
\\\partial_t\eta =v_3 -v_1\partial_{y_1}\eta -v_2\partial_{y_2}\eta \quad &\text{on }\Gamma (t)
\\(\tilde{p}_+I-\mu_+\mathbb{D}(v_+))n_+  =\rho_+g\eta_+ n_+-\sigma_+ H_+n_+
\quad &\text{on }\Gamma_+(t)\\
v_+=v_-\quad&\text{on }\Gamma_-(t)
\\
(\tilde{p}_+I-\mu_+\mathbb{D}(v_+))n_--(\tilde{p}_-I-\mu_-\mathbb{D}(v_-))n_-=\rj g\eta_-
n_-+ \sigma_- H_-n_- \quad &\text{on }\Gamma_-(t)
\\(v,\eta)\mid_{t=0}=(v_0,\eta_0).
\end{array}\right.
\end{equation}
In this paper we shall always unify the notations by means of \eqref{uni-notation} to suppress the subscript ``$\pm$'' unless clarification is needed.

\subsection{Reformulation via a flattening coordinate transformation}

The movement of the free boundaries $\Gamma_\pm(t)$ and the subsequent change of the domains $\Omega_\pm(t)$  create numerous mathematical difficulties. To circumvent these, as usual, we will transform the free boundary problem under consideration to a problem with a fixed domain and fixed boundary. Since we are interested in the stability and instability of the equilibrium state, we will use the equilibrium domain. We will not use the usual Lagrangian coordinate transformation as in \cite{So, B1}, but rather utilize a special flattening coordinate transformation that we introduced in \cite{WTK}.

To this end, we define the fixed domain
\begin{equation}
\Omega = \Omega_+\cup\Omega_-\text{ with }\Omega_+:=\{0<x_3<1\},\ \Omega_-:=\{-b<x_3<0\},
\end{equation}
for which we have written the coordinates as $x\in \Omega$. We shall write $\Sigma_+:=\{x_3= 1\}$ for the upper boundary, $\Sigma_-:=\{x_3=0\}$ for the internal interface and $\Sigma_b:=\{x_3=-b\}$ for the lower boundary.  Throughout the paper we will write $\Sigma = \Sigma_+ \cup \Sigma_-$.   We think of $\eta_\pm$ as a function on $\Sigma_\pm$ according to $\eta_+: (\mathrm{T}^2\times\{1\}) \times \mathbb{R}^{+} \rightarrow\mathbb{R}$ and $\eta_-:(\mathrm{T}^2\times\{0\}) \times \mathbb{R}^{+} \rightarrow \mathbb{R}$, respectively, where $\mathbb{R}^{+} = (0,\infty)$. We will transform the free boundary problem in $\Omega(t)$ to one in the fixed domain $\Omega $ by using the unknown free surface functions $\eta_\pm$. For this we define
\begin{equation}
\bar{\eta}_+:=\mathcal{P}_+\eta_+=\text{ Poisson extension of }\eta_+ \text{ into }\mathrm{T}^2 \times \{x_3\le 1\}
\end{equation}
and
\begin{equation}
\bar{\eta}_-:=\mathcal{P}_-\eta_-=\text{ specialized Poisson extension of }\eta_-\text{ into }\mathrm{T}^2 \times \mathbb{R},
\end{equation}
where $\mathcal{P}_\pm$ are defined by \eqref{P+def} and \eqref{P-def}.  We now encounter the first key problem of how to define an appropriate coordinate transformation to flatten the two free surfaces together.  Since we only need to transform the third component of the spatial coordinate and keep the other two fixed, we can flatten the domain by a linear combination of the three boundary functions, as introduced by Beale in \cite{B2}. However, this would result in the discontinuity of the Jacobian matrix of the coordinate transformation, which would then lead to severe technical difficulties in the proof of the local well-posedness of the problem. In \cite{WTK}, we overcame this difficulty  by flattening  the coordinate domain via the following special coordinate transformation, writing $\tilde{b} = 1 + x_3/b$:
\begin{equation}\label{cotr}
\left\{\begin{array}{lll}
\Omega_+\ni x\mapsto(x_1,x_2, x_3 +
x_3^2(\bar{\eta}_+ - (1+1/b)\bar{\eta}_-)+ \tilde{b} \bar{\eta}_-)=\Theta_+(t,x)=(y_1,y_2,y_3)\in\Omega_+(t),
\\\Omega_-\ni x\mapsto(x_1,x_2,x_3 + \tilde{b} \bar{\eta}_-)=\Theta_-(t,x)=(y_1,y_2,y_3)\in\Omega_-(t).
\end{array}\right.
\end{equation}
Note that $\Theta(\Sigma_+,t)=\Gamma_+(t),\ \Theta (\Sigma_-,t)=\Gamma_-(t)$ and $\Theta_-(\cdot,t) \mid_{\Sigma_b} = Id \mid_{\Sigma_b}$. In order to write down the equations in the new coordinate system,  we compute
\begin{equation}
\begin{array}{ll} \nabla\Theta =\left(\begin{array}{ccc}1&0&0\\0&1&0\\A &B &J \end{array}\right)
\text{ and }\mathcal{A} := \left(\nabla\Theta
^{-1}\right)^T=\left(\begin{array}{ccc}1&0&-A  K \\0&1&-B  K \\0&0&K
\end{array}\right)\end{array}.
\end{equation}
Here the components in the matrix are
\begin{equation}\label{ABJ_def}
\left\{\begin{array}{lll}
A_+ = x_3^2 \left(\partial_1\bar{\eta}_+ - (1+1/b) \partial_1\bar{\eta}_- \right) + \partial_1\bar{\eta}_- \tilde{b},
\,
B_+ = x_3^2 \left(\partial_2\bar{\eta}_+ - (1+1/b) \partial_2\bar{\eta}_- \right) + \partial_2\bar{\eta}_- \tilde{b},
\\
J_+ = 1 + 2x_3 (\bar{\eta}_+ - (1+1/b) \bar{\eta}_-)+x_3^2(\partial_3\bar{\eta}_+
-(1+1/b) \partial_3\bar{\eta}_-) + \bar{\eta}_-/b +\partial_3\bar{\eta}_- \tilde{b},
\\ A_-=\partial_1 \bar{\eta}_- \tilde{b}, \,
 B_-=\partial_2\bar{\eta}_- \tilde{b} ,
\\  J_-=1+\bar{\eta}_-/b + \partial_3\bar{\eta}_- \tilde{b}, \,  K =J ^{-1}.
\end{array}\right.
\end{equation}
Notice that $J={\rm det}\, \nabla\Theta $ is the Jacobian of the coordinate transformation. We also denote
\begin{equation}\label{WNT_def}
 \begin{split}
    W & = \partial_t\Theta_3 K \\
    \mathcal{N} &= (-\partial_1\eta ,-\partial_2\eta ,1) \\
    \mathcal{T}^i &= e_i + \partial_i\eta e_3 \text{ for }i=1,2\\
    \theta_{ij}&=\partial_j\Theta_i\text{ for }i,j=1,2,3.
 \end{split}
\end{equation}
It is straightforward to check that, because of how we have defined $\bar{\eta}_\pm$, the matrix $\mathcal{A}$ and the other terms defined above are all continuous across the internal interface $\Sigma_-$. This is crucial for the whole analysis in our proof of local well-posedness of the problem in \cite{WTK}. Also, we may remark that one can replace $x_3^2$ in the coordinate transformation \eqref{cotr} by some smooth function $h(x_3)$ with $h(1)=1$ and $h(0)=h'(0)=\cdots=0$, to let $\Theta$ be more regular across the interface $\Sigma_-$. Note that if $\eta $ is sufficiently small (in an appropriate Sobolev space), then the mapping $\Theta $ is a  $C^1$ diffeomorphism.  This allows us to transform the problem to one in the fixed spatial domain $\Omega$ for each $t\ge 0$. We will transform the problem in different ways according to whether we take into account the effect of surface tension.

In the case with surface tension, i.e. $\sigma_\pm>0$, $\eta$ will be in a higher regularity class than $u$. This allows us to use the clever idea introduced by Beale in \cite{B2} to transform the velocity field in a manner that preserves the divergence-free condition. We define the pressure $p$ on $\Omega$ by the composition $p(t,x)=\tilde{p} \circ \Theta(t,x)$, but we define the velocity $u$ on $\Omega$ according to $v_i\circ \Theta(t,x)=K \theta_{ij} u_j(t,x)$ (equivalently, $u_i(t,x)=J \mathcal{A}_{ji} v_j \circ \Theta(t,x)$). The advantages of this transform are twofold.  First, $u$ has divergence zero in $\Omega$ if and only if $v$ has the same property in $\Omega(t)$.  Second, the right-hand side of $\eqref{form1}_3$ is replaced simply by $u_3$. In the new coordinates, the system \eqref{form1} with surface tension becomes the equations for the new unknowns $(u,p,\eta)$ in the following perturbation form, denoting the interfacial jump as $\Lbrack f\Rbrack =f_+\mid_{x_3=0}-f_-\mid_{x_3=0}$:
 \begin{equation}\label{surface}
 \left\{\begin{array}{lll}\rho\partial_t u
 -\mu\Delta u+\nabla p=f\quad&\hbox{in }\Omega
\\ \diverge u=0&\hbox{in }\Omega
\\ \partial_t\eta=u_3 &\hbox{on }\Sigma
\\ ( p_+I-\mu_+\mathbb{D}(u_+)) e_3= (\rho_+g\eta_+  -\sigma_+\Delta_\ast \eta_+)e_3+g_+&\hbox{on }\Sigma_+
\\\Lbrack u\Rbrack=0,\quad \Lbrack pI-\mu\mathbb{D}(u)\Rbrack e_3
=(\rj g\eta_- +\sigma_- \Delta_\ast \eta_-)e_3-g_-&\hbox{on }\Sigma_-
\\ u_-=0 &\hbox{on }\Sigma_b
\\(u,\eta )\mid_{t=0}=(u_0,\eta_0),
\end{array}\right.
\end{equation}
where the nonlinear terms $f=(f^1,f^2,f^3)$ and $g_\pm=(g_\pm^1,g_\pm^2,g_\pm^3)$ are given by
\begin{equation}\label{f}
\begin{split}
f^i&=-\rho J \mathcal{A}_{ji}[ \partial_t(K \theta_{jk})u_k
-W\partial_3(K \theta_{jk})u_k+K \partial_k(K\theta_{jl})u_ku_l]
  \\&\quad+\rho [W \partial_3 u_i -K u_k \partial_k u_i ]
 +\mu
J\mathcal{A}_{ni}\mathcal{A}_{jk}\partial_k(\mathcal{A}_{jl})\partial_l(K\theta_{nm})u_m
  \\&\quad+\mu   \mathcal{A}_{jk}\partial_k(\mathcal{A}_{jl})   \partial_lu_i
 +\mu J\mathcal{A}_{ni}\mathcal{A}_{jk} \mathcal{A}_{jl}\partial_k\partial_l(K\theta_{nm})u_m
 \\&\quad+\mu J\mathcal{A}_{ni}\mathcal{A}_{jk}
\mathcal{A}_{jl}\partial_l(K\theta_{nm})\partial_k u_m +\mu
J\mathcal{A}_{ni}\mathcal{A}_{jk}
\mathcal{A}_{jl}\partial_k(K\theta_{nm})\partial_l u_m
 \\&\quad+\mu   (\mathcal{A}_{jk}
\mathcal{A}_{jl}-\delta_{kl})    \partial_k \partial_l u_i
 +(\delta_{ki}-J\mathcal{A}_{ji}\mathcal{A}_{jk} )\partial_k p,\quad
 i=1,2,3;
 \end{split}
 \end{equation}
 \begin{equation}\label{g_+^i}
\begin{split}
g_+^i&=\mu J\mathcal{A}_{mk}\partial_k(K\theta_{jl})u_l\mathcal{
N}_j\mathcal{T}_m^i
+\mu(\mathcal{A}_{mk}\theta_{jl}-\delta_{ik}\delta_{jl})\partial_k
u_l\mathcal{ N}_j\mathcal{T}_m^i
   \\&\quad + \mu J\mathcal{A}_{jk}\partial_k(K\theta_{ml})u_l\mathcal{ N}_j\mathcal{T}_m^i
+\mu(\mathcal{A}_{jk}\theta_{ml}-\delta_{jk}\delta_{il})\partial_k
u_l\mathcal{ N}_j\mathcal{T}_m^i
 \\&\quad -\mu\partial_i u_1
\partial_1\eta-\mu\partial_i u_2
\partial_2\eta+\partial_3u_l\mathcal{ N}_l\partial_i\eta
 \\& \quad-\mu \partial_1 u_i \partial_1\eta -\mu
\partial_2 u_i \partial_2\eta
 + \mu\partial_k u_3 \mathcal{ N}_k\partial_i\eta,\quad
 i=1,2,
  \end{split}
 \end{equation}
 \begin{equation} \label{g_+^3}
\begin{split}
g_+^3& = 2\mu  (\mathcal{A}_{3k} K\theta_{3l}-\delta_{3k}\delta_{3l})\partial_ku_l
+2\mu  \mathcal{A}_{3k}\partial_k(K\theta_{3l})u_l
\\&\quad  +2\mu
(\mathcal{A}_{ik}\partial_k(K\theta_{jl})u_l
+\mathcal{A}_{ik}K\theta_{jl}\partial_k u_l  )
 (\mathcal{ N}_j\mathcal{ N}_i-\delta_{j3}\delta_{i3})|\mathcal{N}|^{-2}
\\&\quad+2\mu  (\mathcal{A}_{3k}\partial_k(K\theta_{3l}
)u_l+\mathcal{A}_{3k}K\theta_{3l}\partial_k u_l
)(|\mathcal{N}|^{-2}-1)
 \\&\quad- \sigma_+((1+|\nabla_\ast\eta|^2)^{-1/2} -1) \Delta_\ast\eta
 \\&\quad+ \sigma_+(1+|\nabla_\ast\eta|^2)^{-3/2} ((\partial_1\eta)^2\partial_1^2\eta
 +2\partial_1\eta\partial_2\eta\partial_1\partial_2\eta+(\partial_2\eta)^2\partial_2^2\eta);
 \end{split}
 \end{equation}
 \begin{equation}\label{g_-^i}
\begin{split} -g_-^i&=\Lbrack \mu\Rbrack
J\mathcal{A}_{mk}\partial_k(K\theta_{jl})u_l\mathcal{
N}_j\mathcal{T}_m^i
+(\mathcal{A}_{mk}\theta_{jl}-\delta_{ik}\delta_{jl})\Lbrack\mu
\partial_k u_l\Rbrack\mathcal{ N}_j\mathcal{T}_m^i
   \\&\quad  + \Lbrack \mu\Rbrack J\mathcal{A}_{jk}\partial_k(K\theta_{ml})u_l\mathcal{ N}_j\mathcal{T}_m^i
+(\mathcal{A}_{jk}\theta_{ml}-\delta_{jk}\delta_{il})\Lbrack\mu\partial_k
u_l\Rbrack\mathcal{ N}_j\mathcal{T}_m^i
   \\&\quad-\Lbrack\mu\Rbrack\partial_i
u_1
\partial_1\eta-\Lbrack\mu\Rbrack\partial_i u_2
\partial_2\eta+\Lbrack\mu\partial_3u_l\Rbrack\mathcal{
N}_l\partial_i\eta  \\& \quad-\Lbrack\mu\Rbrack \partial_1
u_i \partial_1\eta -\Lbrack\mu \Rbrack\partial_2 u_i \partial_2\eta
 + \Lbrack\mu\partial_k u_3 \Rbrack \mathcal{ N}_k\partial_i\eta,\quad i=1,2,
  \end{split}
 \end{equation}
 \begin{equation}\label{g_-^3}
\begin{split}-g^3& = 2  (\mathcal{A}_{3k} K\theta_{3l}-\delta_{3k}\delta_{3l})\Lbrack\mu\partial_ku_l\Rbrack
+2\Lbrack\mu  \Rbrack \mathcal{A}_{3k}\partial_k(K\theta_{3l})u_l
 \\&\quad  +2
(\Lbrack\mu\Rbrack\mathcal{A}_{ik}\partial_k(K\theta_{jl})u_l
+\mathcal{A}_{ik}K\theta_{jl}\Lbrack\mu\partial_k u_l\Rbrack  )
 (\mathcal{ N}_j\mathcal{ N}_i-\delta_{j3}\delta_{i3})|\mathcal{N}|^{-2}
 \\&\quad+2
(\Lbrack\mu\Rbrack\mathcal{A}_{3k}\partial_k(K\theta_{3l}
)u_l+\mathcal{A}_{3k}K\theta_{3l}\Lbrack\mu\partial_k u_l\Rbrack
)(|\mathcal{N}|^{-2}-1)
  \\&\quad+\sigma_- ((1+|\nabla_\ast\eta|^2)^{-1/2} -1) \Delta_\ast\eta
 \\&\quad-\sigma_- (1+|\nabla_\ast\eta|^2)^{-3/2} ((\partial_1\eta)^2\partial_1^2\eta+2\partial_1\eta\partial_2\eta\partial_1\partial_2\eta+(\partial_2\eta)^2\partial_2^2\eta).
  \end{split}
 \end{equation}
Note that  the coordinate transformation \eqref{cotr} guarantees that $\Lbrack u\Rbrack=0$ on $\Sigma_-$.

In the case without  surface tension, i.e., $\sigma_\pm=0$, $\eta$ will be in the same regularity class as $v$, so we can not transform the velocity field  as above to  preserve  the divergence-free condition. Rather, we define the velocity $u$ and the pressure $p$ on $\Omega$ directly by the composition with $\Theta$: $u(t,x)=v\circ\Theta(t,x)$ and $p=\tilde{p}\circ\Theta(t,x)$. Hence in the new coordinates, the system \eqref{form1} without surface tension becomes the following equations for the new unknowns $(u,p,\eta)$ in the perturbation form:
\begin{equation}\label{nosurface2}
\left\{\begin{array}{lll}\rho\partial_t u-\mu\Delta u+\nabla p=G^1\quad&\hbox{in }\Omega
\\ \diverge u=G^2&\hbox{in }\Omega
\\ \partial_t\eta-u_3=G^4&\hbox{on }\Sigma
\\ ( p_+I-\mu_+\mathbb{D}(u_+)) e_3= \rho_+g\eta_+ e_3+G^3_+&\hbox{on }\Sigma_+
\\ \Lbrack u\Rbrack=0,\quad \Lbrack pI-\mu\mathbb{D}(u)\Rbrack e_3=\rj g\eta_- e_3-G^3_- &\hbox{on }\Sigma_-
\\ u_-=0 &\hbox{on }\Sigma_b
\\(u,\eta)\mid_{t=0}=(u_0,\eta_0),
\end{array}\right.
\end{equation}
where the nonlinear terms are given by
\begin{eqnarray}
&&G^1=\rho W\partial_3u-\rho u\cdot\nabla_\mathcal{A}u+\mu(\Delta_\mathcal{A}-\Delta)u-(\nabla_\mathcal{A}-\nabla)p,\label{G1}
\\&&G^2=(\diverge -\diverge_\mathcal{A})u,\label{G2}
\\&&G^3_+=(p-\rho_+ g \eta_+)(e_3-\mathcal{N}_+)+\mu_+(\mathbb{D}u e_3-\mathbb{D}_{\mathcal{A}_+}u_+ \mathcal{N}_+),\label{G3+}
\\&&G^3_-=-(\Lbrack p \Rbrack-\rj  g \eta_-)(e_3-\mathcal{N}_-)-\Lbrack \mu(\mathbb{D}u e_3-\mathbb{D}_\mathcal{A}u \mathcal{N})\Rbrack,\label{G3-}
\\&&G^4=-u_1\partial_1\eta-u_2\partial_2\eta. \label{G4}
\end{eqnarray}
We may also refer to \cite{GT_per} for more explicit expressions of $G^i$.

\subsection{Main results}

The Rayleigh-Taylor instability arises when steady states of two fluid layers with different densities are accelerated in the direction toward the denser fluid.  The instability is well-known since the classical work of Rayleigh \cite{3R} and of Taylor \cite{3T}.  For a general discussion of the physics related to this topic, we refer to \cite{3K} and the references therein.

The Rayleigh-Taylor instability has attracted much attention in the mathematics community due both to its physical importance and to the mathematical challenges it offers.  Presently, the linear Rayleigh-Taylor instability is well understood, but the nonlinear instability is much less so.  The standard procedure (see, for instance, Chandrasekhar's book \cite{3C}) for analyzing the linear instability is to try to construct ``normal mode'' (horizontal Fourier modes of a fixed spatial frequency that grow exponentially in time) solutions to the linearization of problem \eqref{form1}.  For instance, the inviscid ($\mu_\pm =0$) version of our present problem can be analyzed with normal modes as in \cite{3C}.  It can be shown that without surface tension all the spatial frequencies are unstable.  With surface tension, there is a critical frequency $|\xi|_c=\sqrt{g\rj /\sigma_-}$ so that the spatial frequencies $\xi$ with $0<|\xi|< |\xi|_c$ are unstable, while the spatial frequencies $\xi$ with $|\xi|>|\xi|_c$ are stable.   This in particular implies that  the inviscid version of problem \eqref{form1}, defined in a horizontally infinite setting (i.e., horizontal cross-section $\mathbb{R}^2$ in place of $\mathrm{T}^2$), is linearly unstable.  However, in the horizontally periodic setting, if the domain is sufficiently small so that
\begin{equation}\label{crit_st_def}
\min\{L_1^{-1},L_2^{-1}\} \ge \sqrt{\frac{g\rj }{\sigma_- }}
\Longleftrightarrow\ \sigma_- \ge \sigma_c:= g\rj \max\{L_1^2,L_2^2\},
\end{equation}
then there will be no frequency belonging to the interval $0<|\xi|<|\xi|_c$.  This implies that the inviscid version of problem  \eqref{form1} should be linearly stable.  The same result for (viscous and inviscid) compressible fluids was shown by Guo and Tice \cite{3GT2, 3GT1}.

As a part of our previous paper \cite{WTK}, we proved that the viscous surface-internal wave problem \eqref{form1} is nonlinearly stable for $\sigma_->\sigma_c$ and that the solutions decay to the equilibrium state at an exponential rate. In this paper, we will continue our study of \eqref{form1} by showing that the  problem  is nonlinearly unstable for $\sigma_-<\sigma_c$.

Our main results can be stated as follows. The first theorem establishes  the nonlinear instability for the case with surface tension.

\begin{theorem}\label{maintheorem}
Assume that $\rj >0$, $\sigma_+ >0$, and $0<\sigma_-<\sigma_c$, where $\sigma_c$ is defined by \eqref{crit_st_def}. Let the norm $ \Lvert3   \cdot   \Rvert3_{0}$ be defined by \eqref{norm1} and another (stronger) norm $ \Lvert3   \cdot   \Rvert3_{00}$ be defined by \eqref{norm2}. There exist constants $\theta_0>0$  and $C>0$ such that for any sufficiently small $0< \iota < \theta_0$, there  exist solutions $\begin{pmatrix} u^\iota(t)  \\ \eta^\iota (t) \end{pmatrix}$ to  \eqref{surface} so that
 \begin{equation}
\Lvert3   \begin{pmatrix} u^\iota(0) \\ \eta^\iota(0) \end{pmatrix}
\Rvert3_{00}\le C\iota,\hbox{ but }  \Lvert3     \begin{pmatrix} u^\iota(T^\iota) \\ \eta^\iota(T^\iota) \end{pmatrix}
\Rvert3_{0}\ge \frac{\theta_0}{2}.
 \end{equation}
Here the escape time $T^\iota>0$ is
\begin{equation}
T^\iota:=\frac{1}{\Lambda}\log\frac{\theta_0}{\iota},
\end{equation}
where $\Lambda>0$, defined by \eqref{Lambda}, is the sharp  linear growth rate obtained in Section \ref{linear instability}.
\end{theorem}

In \cite{WTK} we also proved the nonlinear stability for the problem \eqref{nosurface2} without surface tension ($\sigma_\pm =0$), provided that $\rj <0$, i.e. the lighter fluid is on top.  Our next theorem is the analog of Theorem \ref{maintheorem} for the case without surface tension in the case $\rj >0$.

\begin{theorem}\label{maintheorem2}
Assume that $\rj >0$ and $\sigma_\pm =0$. Let the norm $ \Lvert3   \cdot   \Rvert3_{0}$ be defined by \eqref{norm1} and another (stronger) norm $\Lvert3   \cdot   \Rvert3_{00}$ defined by \eqref{norm3} (with $N\ge 3$ an integer). There exist constants $\theta_0>0$  and $C>0$ such that for any sufficiently small $0< \iota<\theta_0$, there  exist solutions $\begin{pmatrix} u^\iota(t)  \\ \eta^\iota (t) \end{pmatrix}$ to  \eqref{nosurface2} so that
 \begin{equation}
  \Lvert3   \begin{pmatrix} u^\iota(0) \\ \eta^\iota(0) \end{pmatrix}
  \Rvert3_{00}\le C\iota,\hbox{ but }  \Lvert3     \begin{pmatrix} u^\iota(T^\iota) \\ \eta^\iota(T^\iota) \end{pmatrix}
    \Rvert3_{0}\ge \frac{\theta_0}{2}.
 \end{equation}
Here the escape time $T^\iota>0$ is
 \begin{equation}
T^\iota:=\frac{1}{\Lambda_\ast}\log\frac{\theta_0}{\iota},
\end{equation}
 where $\Lambda_\ast>0$, defined by \eqref{littlelambda}, is close to the sharp  linear growth rate $\Lambda$ obtained in Section \ref{linear instability}.
\end{theorem}

\begin{Remark}
The solution is actually a triple $(u,p,\eta)$, and the quantity $\Lvert3   \cdot   \Rvert3_{00}$ employed in Theorems \ref{maintheorem} and \ref{maintheorem2} is actually a norm for $(u,p,\eta)$.  See Remark \ref{norm_remark} for more details on our slight abuse of notion and notation.  Also, although we have suppressed it in the notation, the quantity $\Lvert3   \cdot   \Rvert3_{00}$ appearing in Theorem \ref{maintheorem2} depends on an integer $N\ge 3$ that measures the regularity of solutions ($u \in \ddot{H}^{4N}$, etc).  The choice $N=3$ is the minimal regularity framework in which we can close our nonlinear instability estimates. However, by letting $N$ be arbitrarily large we find that no high-regularity smallness criterion can prevent instability.
\end{Remark}

\begin{Remark}
In both Theorems \ref{maintheorem} and \ref{maintheorem2} the instability occurs in the $\Lvert3 \cdot \Rvert3_{0}$ norm, which is just the $L^2$ norms of $u$ and $\eta$.  This means that although our instability analysis works in a framework with some degree of regularity, the onset of instability occurs at the lowest level of regularity.  Note that an easy refinement of our analysis would allow us to replace $\Lvert3 \cdot \Rvert3_{0}$ with the $L^2$ norm of $\eta_-$.  This highlights the fact that the instability occurs at the internal interface.  We have chosen not to pursue this refinement in order to make $\Lvert3 \cdot \Rvert3_{0}$ consistent with a quantity that appears in our linear analysis.
\end{Remark}

\begin{Remark}
Theorems \ref{maintheorem} and \ref{maintheorem2}, together with our results in \cite{WTK}, establish sharp nonlinear stability criteria for the equilibrium state in the viscous surface-internal wave problem.  We summarize these and the rates of decay to equilibrium in the table below.
\begin{displaymath}
\begin{array}{| c | c | c |  c |}
\hline
 & \rj < 0 & \rj =0 & \rj >0  \\ \hline
 \sigma_\pm =0 &
\begin{array}{c}
\textnormal{nonlinearly stable} \\
\textnormal{almost exponential decay}
\end{array} &
\textnormal{locally well-posed} & 
\textnormal{nonlinearly unstable}     \\ \hline
\begin{array}{c}
0 < \sigma_+  \\ 0 < \sigma_- < \sigma_c
\end{array} &
\begin{array}{c}
\textnormal{nonlinearly stable} \\
\textnormal{exponential decay}
\end{array} &
\begin{array}{c}
\textnormal{nonlinearly stable} \\
\textnormal{exponential decay}
\end{array} &
\textnormal{nonlinearly unstable}      \\ \hline
\begin{array}{c}
0 < \sigma_+  \\  \sigma_c = \sigma_-
\end{array} &
\begin{array}{c}
\textnormal{nonlinearly stable} \\
\textnormal{exponential decay}
\end{array} &
\begin{array}{c}
\textnormal{nonlinearly stable} \\
\textnormal{exponential decay}
\end{array} &
\textnormal{locally well-posed}
\\ \hline
\begin{array}{c}
0 < \sigma_+  \\  \sigma_c < \sigma_-
\end{array} &
\begin{array}{c}
\textnormal{nonlinearly stable} \\
\textnormal{exponential decay}
\end{array} &
\begin{array}{c}
\textnormal{nonlinearly stable} \\
\textnormal{exponential decay}
\end{array} &
\begin{array}{c}
\textnormal{nonlinearly stable} \\
\textnormal{exponential decay}
\end{array}
\\ \hline
\end{array}
\end{displaymath}

Note that our table identifies two critical regimes:  $\sigma_- = \sigma_c$, $\sigma_+ >0$, $\rj >0$; and $\sigma_\pm =0$, $\rj=0$.  In these critical regimes we know from \cite{WTK} that the problem is locally well-posed, but it is  not clear to us what the stability of the system should be.

\end{Remark}

In general, the passage from linear instability to nonlinear instability is delicate for  partial differential equations due to, for instance, severe nonlinearities with possible unbounded derivatives.  In 1995, Guo and Strauss \cite{GS} developed a bootstrap instability framework to treat such an issue. The main idea of the Guo-Strauss approach is to show that on the time scale of the instability, the stronger Sobolev norm  of the solution is actually bounded by the growth of its weaker norm (see also Lemma 1.1 of \cite{GHS} for the abstract framework of this method).  For our Rayleigh-Taylor problem the term $\rj g\eta_-$ is what leads to the instability; since it is of  low order, we are led to use the Guo-Strauss bootstrap framework here.  However, the bootstrap lemma in \cite{GHS} can not be applied directly, so in the proof of Theorems \ref{maintheorem} and \ref{maintheorem2} we use a slight variant.   For the sake of completeness, we record an abstract version of this method in the following theorem.

\begin{theorem}\label{bootstrap2}
Assume the following.
\begin{enumerate}
 \item $L$ is a linear operator on a Banach space $X$ with norm $\Lvert3\cdot\Rvert3_{0}$, and there exists another norm  $\norm{\cdot}$ such that
\begin{equation}\label{l21}
\Lvert3 e^{tL}y\Rvert3_0 \le C_Le^{t\Lambda}\|y\|
\end{equation}
for some $C_L$ and $\Lambda> 0$.

\item $N = N(y)$ is some nonlinear operator on $X$, and there exist another norm $\Lvert3\cdot\Rvert3_{00}$ and a constant $C_N$ such that
\begin{equation}\label{l12}
\|N(y)\|\le C_N\Lvert3y\Rvert3^2_{00}
\end{equation}
for all $y\in X$ with  $\Lvert3y\Rvert3_{00}<\infty$.

\item There is a growing mode $y_0\in X$ with $\Lvert3 y_0\Rvert3_{0}= 1$ and $\Lvert3y_0\Rvert3_{00}<\infty$, and
\begin{equation}\label{l27}
 e^{Lt} y_0=e^{\lambda t}y_0
 \end{equation}
for some $\lambda>0$ with $\frac{\Lambda}{2}<\lambda\le \Lambda$.

\item There exists a small constant $\delta$ such that for any solution $y(t)$ to the equation
\begin{equation}\label{equation}
\frac{d}{dt}y+Ly=N(y)
\end{equation}
with $\Lvert3 y(t)\Rvert3_{00}\le\delta$ for all $t\in [0,T]$, there exists $C_\delta> 0$ so that the following  energy estimate holds:
\begin{equation}
\Lvert3y(t)\Rvert3_{00}^2\le  C_\delta\Lvert3y(0)\Rvert3_{00}^2
   +  \frac{\lambda}{2} \int_0^t\Lvert3y(s)\Rvert3_{00}^2\,ds +   C_\delta\int_0^t\Lvert3y(s)\Rvert3_{00}^3\,ds
   + C_\delta\int_0^t \Lvert3y(s)\Rvert3_{0}^2\,ds
\end{equation}
for all $t \in [0,T]$.

\end{enumerate}
Let $y^\iota$ be a solution to \eqref{equation} with  initial data $y^\iota (0) = \iota  y_0$.  Then there exists a sufficiently small (fixed) number $\theta_0>0$, which depends explicitly on $C_L$, $C_N$, $C_\delta$, $\lambda$, $\Lambda$, $y_0$, $\delta$ but is independent of $\iota $, so that for any sufficiently small $\iota>0$,
\begin{equation}
 \Lvert3 y^\iota(T^\iota )\Rvert3_0\ge  \frac{\theta_0}{2}>0,
\end{equation}
where the escape time $T^\iota>0$ is defined by
\begin{equation}
 T^\iota := \frac{1}{\lambda}\log \frac{\theta_0}{\iota }.
\end{equation}
\end{theorem}

Theorem \ref{bootstrap2} can be proved by a slight modifications of the  proof of Lemma 1.1 of \cite{GHS}, which also follows along the lines of our proof of Theorems \ref{maintheorem} and \ref{maintheorem2}, and hence we omit its proof.  In fact, one may read Theorem \ref{bootstrap2} as a description of our strategy for proving Theorems \ref{maintheorem} and \ref{maintheorem2}.  Comparing with Lemma 1.1 of \cite{GHS}, we see that \eqref{l21} and \eqref{l27} are two relaxations of the corresponding conditions in that lemma.  This is precisely the case we encounter when we try to use the Guo-Strauss bootstrap framework to prove nonlinear Rayleigh-Taylor instability.  Another point we want to mention here is that in Theorem \ref{bootstrap2} we implicitly assume that the linear growing mode $y_0$ can be used as initial data for the nonlinear problem \eqref{equation}. However, in our nonlinear problems \eqref{surface} and \eqref{nosurface2} defined on a domain with boundary, to guarantee the local well-posedness the initial data must satisfy certain compatibility conditions that the growing modes to the linearized problem would not satisfy.  Hence, we need to employ an argument from \cite{JT} to use the linear growing mode to construct initial data for the nonlinear problem in Section \ref{data}.

We will end this introduction by reviewing some previous mathematical results closely related to our current paper. The stability and instability analysis for the linearized problem without viscosity is well understood \cite{3C}. However, there is no general theory that guarantees the passage from linear instability to nonlinear instability for  partial differential equations, so the question of nonlinear instability was not immediately resolved by the  linear analysis.  For the inviscid Rayleigh-Taylor problem without surface tension, Ebin \cite{3E} proved the nonlinear ill-posedness of the problem for incompressible fluids,  Guo and Tice \cite{3GT1} proved an analogous result for compressible fluids, and Hwang and Guo \cite{hw_guo} proved the nonlinear instability of the incompressible problem with a continuous density distribution.  For the viscous Rayleigh-Taylor problem with or without surface tension, Guo and Tice \cite{3GT2} proved the linear instability for  compressible fluids, and with surface tension  Pr\"{u}ss and Simonett \cite{PS} proved nonlinear instability for incompressible fluids in an $L^p$-setting by using Henry's instability theorem.  The proof of \cite{PS} is quite abstract, hence in this paper for the case with surface tension we  provide an alternative proof to show the nonlinear instability in the natural energy space.  However, to our best knowledge, the result is completely new for the case without surface tension.  In forthcoming papers, we expect to apply the framework here to show the nonlinear Rayleigh-Taylor instability for viscous incompressible fluids with magnetic field, the linear instability of which was obtained by Wang \cite{W}.

\medskip

\hspace{-13pt}{\bf Notation.} In this paper, for any given domain, we write $H^k$ for the usual $L^2$-based Sobolev space of order $k\ge0$. We define the piecewise Sobolev spaces $\ddot{H}^k(\Omega), k\ge 0$ by
\begin{equation}
\ddot{H}^k(\Omega) = \{u\chi_{\Omega_\pm}\in H^k(\Omega_\pm)\}\text{ with norm } \|u\|_{k}^2 :=\|u\|_{\ddot{H}^k(\Omega)}^2 := \|u_+\|_{H^k(\Omega_+)}^2 + \|u_-\|_{H^k(\Omega_-)}^2.
\end{equation}
When $k=0$, $\ddot{H}^0(\Omega)={H}^0(\Omega)=L^2(\Omega)$. We let the Sobolev spaces $H^s( \Sigma_\pm)$ for $s\in \mathbb{R}$ be equivalent to $H^s( \mathrm{T}^2)$, with norm $\|\cdot\|_{H^s(\Sigma_\pm)}$; we  write
\begin{equation}
\|\eta\|_{s}^2 := \|\eta\|_{{H}^s(\Sigma)}^2 :=\|\eta_+\|_{H^s(\Sigma_+)}^2 + \|\eta_-\|_{H^s(\Sigma_-)}^2.
\end{equation}
We do not distinguish the notation of norms on the domain or on the boundary. Basically, when we write  $\|\partial_t^ju\|_{k}$ and $\|\partial_t^jp\|_{k}$ we always mean that the space is $\ddot{H}^k(\Omega)$, and when we write $\|\partial_t^j\eta\|_{k}$ we always mean that the space is $ {H}^k(\Sigma)$. When there is potential for confusion, typically when trace estimates are employed, we will clarify as needed. 
Let us define
\begin{equation}\label{0H}
\begin{split}
 {}_0H^1(\Omega)  := \{ u \in H^1(\Omega)\mid u\mid_{\Sigma_b}=0\}  \hbox{ and
 }
 {}_0H_\sigma^1(\Omega) := \{ u \in  {}_0H^1(\Omega) \mid \diverge u=0 \},
\end{split}
\end{equation}
with the obvious restriction that the last space is for vector-valued functions only. It is easy to see that $u\in {}_0H^1(\Omega)$ if and only if $u\in \ddot{H}^1(\Omega)$ with $\Lbrack u\Rbrack=0$ on $\Sigma_-$ and $u=0$ on $\Sigma_b$.  We will not need negative index spaces on $\Omega$ except for $({}_0H^1(\Omega))^\ast$.  In our analysis, we will occasionally abuse notation by writing  $\|\cdot\|_{-1}$ for the norm in $({}_0H^1(\Omega))^\ast$.  Here it is not the case that $({}_0H^1(\Omega))^\ast = H^{-1}$ because of boundary conditions;  we employ this abuse of notation in order to have indexed sums of norms include terms like $\|\cdot\|_{4N-2j+1}$ for $j=2N+1$.  Sometimes we use $\|\cdot\|_{L^pX}$ to denote the norm of the space $L^p(0,T;X)$. We will sometimes extend the above abuse of notation by writing $\|\cdot\|_{L^p H^{k}}$ for $k=-1$ in a sum of norms; when we do this we actually mean $\|\cdot\|_{L^p ({}_0H^1(\Omega))^\ast}$.

We shall also employ the notational convention for derivatives from \cite{GT_per}.  When using space-time differential multi-indices, we will write $\mathbb{N}^{1+m} = \{ \alpha = (\alpha_0,\alpha_1,\dotsc,\alpha_m) \}$ to emphasize that the $0-$index term is related to temporal derivatives.  For just spatial derivatives we write $\mathbb{N}^m$.  For $\alpha \in \mathbb{N}^{1+m}$ we write $\partial^\alpha = \partial_t^{\alpha_0} \partial_1^{\alpha_1}\cdots \partial_m^{\alpha_m}.$ We define the parabolic counting of such multi-indices by writing $|\alpha| = 2 \alpha_0 + \alpha_1 + \cdots + \alpha_m.$  We will also write $\na f$ for the horizontal gradient of $f$, i.e. $\na f = \partial_1 f e_1 + \partial_2 f e_2$, while $\nabla f$ will denote the usual full gradient.

For a given norm $\|\cdot\|$ and an integer $k\ge 0$, we introduce the following notation for sums of spatial derivatives:
\begin{equation}
 \|\na^k f\|^2 := \sum_{\substack{\alpha \in \mathbb{N}^2 \\  |\alpha|\le k} } \|\partial^\alpha  f\|^2 \text{ and } \|\nabla^k f\|^2 := \sum_{\substack{\alpha \in \mathbb{N}^{3} \\
 |\alpha|\le k} } \|\partial^\alpha  f\|^2.
\end{equation}
The convention we adopt in this notation is that $\na$ refers to only ``horizontal'' spatial derivatives, while $\nabla$ refers to full spatial derivatives.   For space-time derivatives we add bars to our notation:
\begin{equation}
 \| \bna^k f\|^2 := \sum_{\substack{\alpha \in \mathbb{N}^{1+2} \\  |\alpha|\le k} } \|\partial^\alpha  f\|^2 \text{ and }
\| \bar{\nabla}^k f\|^2 := \sum_{\substack{\alpha \in
\mathbb{N}^{1+3}
\\   |\alpha|\le k} } \|\partial^\alpha  f\|^2.
\end{equation}
We allow for composition of derivatives in this counting scheme in a natural way; for example, $\norm{\na \na^k f} = \norm{\na^k \na f} = \norm{\na^{k+1} f}$.

We will employ the Einstein convention of summing over  repeated indices.  Throughout the paper $C>0$ will denote a generic constant that can depend on the parameters of the problem, integers $N$ used in energy definitions, and $\Omega$, but does not depend on the data, etc.  We refer to such constants as ``universal.''  Such constants are allowed to change from line to line. When a constant depends on a quantity $z$ we will write $C = C(z)= C_z$ to indicate this.  We will employ the notation $a \lesssim b$ to mean that $a \le C b$ for a universal constant $C>0$. To indicate some constants in some places so that they can be referred to later, we will denote them in particular by $C_1,C_2$, etc.

\section{Linear instability}\label{linear instability}

In this section, we consider the linearized equations  of  \eqref{surface} and \eqref{nosurface2} in the  unified  form
\begin{equation}\label{linear}
\left\{\begin{array}{lll}\rho\partial_t u-\mu\Delta u+\nabla p=0\quad&\hbox{ in }\Omega
\\ \diverge u=0&\hbox{ in }\Omega
\\ \partial_t\eta=u_3 &\hbox{ on }\Sigma
\\ (p_+I-\mu_+\mathbb{D}(u_+)) e_3=(\rho_+g\eta_+ -\sigma_+\Delta_\ast\eta_+)e_3&\hbox{ on
}\Sigma_+
\\\Lbrack u\Rbrack=0,\quad \Lbrack pI-\mu\mathbb{D}(u)\Rbrack e_3=(\rj g\eta_- +\sigma_-\Delta_\ast\eta_-)e_3&\hbox{ on
}\Sigma_-
\\ u_-=0 &\hbox{ on }\Sigma_b
\\(u,\eta)|_{t=0}=(u_0,\eta_0),
\end{array}\right.
\end{equation}
where $\sigma_\pm\ge0$.

\subsection{Growing mode solution}

We want to construct a growing mode solution to \eqref{linear}, so we assume an ansatz
\begin{equation}\label{ansatz}
u(t,x)=w(x){\rm e}^{\lambda t},\ p(t,x)=\tilde{\pi}(x){\rm e}^{\lambda t},
\ \eta(t,x')= {\zeta}(x'){\rm e}^{\lambda t},
\end{equation}
for some $\lambda>0$ (the same in both $\Omega_\pm$), where $x'=(x_1,x_2)$. Substituting this ansatz into \eqref{linear}, we find that $ {\zeta}=\lambda^{-1}w_3|_{\Sigma}$. By this fact we can eliminate $ {\zeta}$ from \eqref{linear} and then arrive at the following time-invariant system for $w$ and $\tilde{\pi}$:
\begin{equation}\label{linear2}
\left\{\begin{array}{ll} \lambda\rho w-\mu\Delta w +\nabla \tilde{\pi}
=0&\hbox{in }\Omega
\\ \diverge w=0&\hbox{in }\Omega
\\ (\tilde{\pi}_+ I-\mu_+ \mathbb{D}(w_+)) e_3
=\lambda^{-1}(g\rho_+w_{3,+}-\sigma_+\Delta_\ast
w_{3,+})e_3&\hbox{on }\Sigma_+
\\\llbracket w\rrbracket=0,\quad
\llbracket \tilde{\pi} I-\mu \mathbb{D}(w)\rrbracket e_3
=\lambda^{-1}(g\rj w_3+\sigma_-\Delta_\ast w_{3})e_3&\hbox{on
}\Sigma_-
\\w_-=0&\hbox{on
}\Sigma_b.
\end{array}\right.
\end{equation}
Since the domain is horizontally flat, we are free to make the further structural assumption that the $x'$ dependence of $w, \tilde{\pi}$ is given as a Fourier mode ${\rm e}^{ix'\cdot\xi}$ for $\xi=(\xi_1,\xi_2)\in L_1^{-1}\mathbb{Z}\times L_2^{-1}\mathbb{Z}$. Together with the growing mode ansatz \eqref{ansatz}, this constitutes a ``normal mode'' ansatz (see \cite{3C}). We define the new unknowns $\varphi,\theta,\psi,\pi: (-b,1)\rightarrow \mathbb{R}$ according to
\begin{equation}\label{ansatz2}
w_1(x)=-i\varphi( x_3){\rm e}^{ix'\cdot\xi},\ w_2(x)=-i\theta( x_3){\rm e}^{ix'\cdot\xi},
\ w_3(x)= \psi( x_3){\rm e}^{ix'\cdot\xi},\
\tilde{\pi}(x)=\pi(x_3){\rm e}^{ix'\cdot\xi}.
\end{equation}
For each fixed nonzero spatial frequency $\xi$, we deduce from the equations \eqref{linear2} a system of ODEs for $\varphi,\theta,\psi,\pi$ and $\lambda$, denoting $'=d/dx_3$:
\begin{equation}\label{linear3}
\left\{\begin{array}{ll}
\lambda^2\rho\varphi+\mu\lambda(|\xi|^2\varphi-\varphi'') -\lambda\xi_1\pi=0&\hbox{ in }(-b,1)
\\\lambda^2\rho\theta+\mu\lambda(|\xi|^2\theta-\theta'')-\lambda\xi_2\pi =0&\hbox{ in }(-b,1)
\\\lambda^2\rho\psi+\mu\lambda(|\xi|^2\psi-\psi'')+\lambda\pi' =0&\hbox{ in }(-b,1)
\\ \xi_1\varphi+\xi_2\theta+\psi'=0&\hbox{ in }(-b,1)
\\ \mu_+\lambda(\xi_1\psi_+-\varphi_+') = \mu_+\lambda(\xi_2\psi_+-\theta_+')=0&\hbox{ at }
x_3=1
\\ -2\mu_+\lambda\psi_+'+\lambda\pi_+  = g\rho_+\psi_+ + \sigma_+|\xi|^2\psi_+&\hbox{ at }
x_3=1
\\
\llbracket\varphi\rrbracket=\llbracket\theta\rrbracket=\llbracket\psi\rrbracket=0&\hbox{
at } x_3=0
\\\llbracket\mu\lambda(\xi_1\psi-\varphi') \rrbracket=\llbracket\mu\lambda(\xi_2\psi-\theta') \rrbracket=0&\hbox{ at }
x_3=0
\\\llbracket-2\mu\lambda\psi'+\lambda\pi \rrbracket= g\rj \psi -\sigma_-|\xi|^2\psi&\hbox{ at }
x_3=0
\\\varphi_-=\theta_-= \psi_- =0&\hbox{ at }
x_3=-b.\end{array}\right.
\end{equation}
We may eliminate $\pi$ by multiplying the first and second equations by $\xi_1,$ $\xi_2$ respectively, summing, and then using the fourth equation to solve for
\begin{equation}\label{pi_solve}
 \lambda | \xi |^2 \pi = - (\lambda^2 \rho\psi' + \lambda \mu (|\xi|^2\psi' - \psi''')).
\end{equation}
Eliminating $\pi$ from the third equation in \eqref{linear3}, we then obtain the following fourth-order ODE for $\psi$:
\begin{equation}\label{linear4}
\left\{\begin{array}{ll}
-\lambda^2\rho(|\xi|^2\psi-\psi'')=\mu\lambda(|\xi|^4\psi-2|\xi|^2\psi''+\psi'''')
&\hbox{in }(-b,1)
\\ \mu_+\lambda(|\xi|^2\psi_++{\psi_+''})=0 &\hbox{at } x_3=1
\\  \mu_+\lambda(\psi_+'''-3|\xi|^2\psi_+')
 = \lambda^2\rho_+{\psi_+'}+g\rho_+|\xi|^2\psi_+ + \sigma_+|\xi|^4 \psi_+&\hbox{at } x_3=1
\\
\llbracket\psi\rrbracket=\llbracket\psi'\rrbracket =0  &\hbox{at }
x_3=0
\\ \llbracket\mu\lambda(|\xi|^2\psi+{\psi''})
\rrbracket=0  &\hbox{at } x_3=0
\\ \llbracket\mu\lambda(\psi'''-3|\xi|^2\psi')
\rrbracket=\llbracket\lambda^2\rho{\psi'}\rrbracket+g\rj |\xi|^2\psi - \sigma_-|\xi|^4
\psi &\hbox{at } x_3=0
\\\psi_-=\psi'_-=0&\hbox{at } x_3=-b.
\end{array}\right.
\end{equation}

Given a solution $\psi$ to \eqref{linear4}, we can recover the function $\pi$ via \eqref{pi_solve}.  To recover $\varphi$ and $\theta$, we first note that the problem \eqref{linear3} is invariant under simultaneous rotations of $(\varphi,\theta)$ and $(\xi_1,\xi_2)$.  Indeed, it is easily verified that if $R \in SO(2)$ is any rotation operator, then $R(\varphi,\theta)$, $R(\xi_1,\xi_2)$ is also a solution with the same $\psi$ and $\lambda$.  Then, given any $\xi$ we choose a rotation operator $R_\xi$ so that $R_\xi \xi = (|\xi|,0)$.  It is easy to check that when $\xi_2 =0$, we may simply take $\theta =0$ in \eqref{linear3}.  Then from a solution  $\psi_{|\xi|}$, $\lambda$, $R_\xi = (|\xi|,0)$ to \eqref{linear4} we can recover $\varphi_{|\xi|}$ through the equation $|\xi| \varphi_{|\xi|} + \psi_{|\xi|}' =0$.  Then we return to a full solution to \eqref{linear3} for $\xi$ by setting $(\varphi,\theta) = R^{-1}_{\xi}(\varphi_{|\xi|},0)$.

The problem \eqref{linear4} can be viewed as an eigenvalue problem with the eigenvalue $\lambda$. However, due to the fact that $\lambda$ appears both quadratically and linearly, this problem does not possess a standard variational structure. In \cite{3GT2}, Guo and Tice developed a robust method to overcome such a difficulty and constructed growing mode solutions for viscous compressible fluids. In \cite{W}, Wang also applied this method to construct growing mode solutions for viscous incompressible fluids with magnetic field in a slightly different setting.  The method was also employed by Jang and Tice \cite{JT} for the compressible Navier-Stokes-Poisson system.  The growing mode solution to \eqref{linear} can be constructed in the same way, though the actual procedure is somewhat simpler, so we will only sketch the procedure for  later use, such as for deriving some estimates.  For a more thorough presentation of the idea, we refer to \cite{3GT2}.

In order to restore the ability to use variational methods we artificially remove the linear dependence on $\lambda$ in \eqref{linear4} by defining the modified viscosities $\tilde{\mu}=s\mu$, where $s>0$ is an arbitrary parameter. This results in a family $(s>0)$ of modified problems:
\begin{equation}\label{linear5}
\left\{\begin{array}{ll}
-\lambda^2\rho(|\xi|^2\psi-\psi'')=s\mu(|\xi|^4\psi-2|\xi|^2\psi''+\psi'''')
&\hbox{in }(-b,1)
\\ s\mu_+(|\xi|^2\psi_++{\psi_+''})=0 &\hbox{at } x_3=1
\\  s\mu_+(\psi_+'''-3|\xi|^2\psi_+')
 = \lambda^2\rho_+{\psi_+'}+g\rho_+|\xi|^2\psi_+ + \sigma_+|\xi|^4 \psi_+&\hbox{at } x_3=1
\\
\llbracket\psi\rrbracket=\llbracket\psi'\rrbracket =0  &\hbox{at }
x_3=0
\\ \llbracket s\mu(|\xi|^2\psi+{\psi''})
\rrbracket=0  &\hbox{at } x_3=0
\\ \llbracket s\mu(\psi'''-3|\xi|^2\psi')
\rrbracket=\llbracket\lambda^2\rho{\psi'}\rrbracket+g\rj |\xi|^2\psi - \sigma_-|\xi|^4
\psi &\hbox{at } x_3=0
\\\psi_-=\psi'_-=0&\hbox{at } x_3=-b.
\end{array}\right.
\end{equation}
Note that for any fixed $s>0$ and $\xi$, \eqref{linear5} is a standard eigenvalue problem for $-\lambda^2$, which has a natural variational structure that allows us to use variational methods to construct solutions. For this, we define the energies
\begin{equation}\label{ener}
\begin{split}
E(\psi;|\xi|,s)
&=\frac{1}{2}\int_{-b}^1 s\mu(4|\xi|^2|\psi'|^2+||\xi|^2\psi+\psi''|^2)
\\&\quad+\frac{1}{2}|\xi|^2(\sigma_+|\xi|^2+g\rho_+)|\psi(1)|^2
+\frac{1}{2}|\xi|^2(\sigma_-|\xi|^2-g\rj )|\psi(0)|^2,
\end{split}
\end{equation}
\begin{equation}
J(\psi;|\xi|)
=\frac{1}{2}\int_{-b}^1
\rho(|\xi|^2|\psi|^2+|\psi'|^2)\,dx_3,
\end{equation}
which are both well-defined on the space ${}_0H^2((-b,1)) := \{ v\in H^2((-b,1)) \;\vert\; \psi(-b)=\psi'(-b)=0\}$. We define the set
\begin{equation}\label{constrain}
\mathcal{C}=\{ \psi \in {}_0H^2((-b,1)) \ |\ J(\psi;|\xi|)=1\}.
\end{equation}
Then we want to find the smallest $-\lambda^2$ by minimizing
\begin{equation}\label{min}
-\lambda^2(|\xi|;s)=\alpha(|\xi|;s):=\inf_{\psi\in
\mathcal{C}}E(\psi;|\xi|,s).
\end{equation}
It is standard to show that a minimizer of \eqref{min} exists and that the minimizer satisfies Euler-Lagrange equations equivalent to \eqref{linear5}.  A standard bootstrap argument also shows that the minimizer $\psi$ is in $H^k((-b,0))$ (resp. $H^k((0,1))$) for all $k\ge 0$ when restricted to $(-b,0)$ (resp. $(0,1)$), and hence is smooth when restricted to either interval.

To  construct the growing mode we then need to clarify the sign of the infimum \eqref{min}. For $\sigma_-\ge\sigma_c$, we always have that $E(\psi;|\xi|,s)\ge0$ for any nonzero frequency $\xi\in L_1^{-1}\mathbb{Z}\times L_2^{-1}\mathbb{Z}$ and any $\psi\in \mathcal{C}$, which then means that $\alpha(|\xi|;s)\ge0$. This suggests that  no growing mode solution to \eqref{linear} can be constructed when $\sigma_- \ge \sigma_c$, and then that the system should be linearly stable.  In \cite{WTK} we proved the nonlinear stability of the  problem \eqref{form1} in the case $\sigma_->\sigma_c$.  However, when $0\le\sigma_-<\sigma_c$  and $0<|\xi|< |\xi|_c:=\sqrt{\rj g/\sigma_-}$ (it is interpreted that when $\sigma_-=0$ this means $0<|\xi|< \infty$), then it is easy to see that if $s>0$ is small enough we can always find  $\psi\in \mathcal{C}$ such that $E(\psi;|\xi|,s)\le 0$.  Hence in this case there is $\psi_s(|\xi|)$ such that $E(\psi_s(|\xi|);|\xi|,s)=\alpha(|\xi|;s)<0$.

We want to show that there is a fixed point $s$ so that $\lambda(|\xi|,s)=s$, which will then allow us to construct a solution to the original problem \eqref{linear4}. To this end, we study the behavior of $\alpha(s)$ as a function of $s> 0$.

\begin{lemma}\label{alpha_s}
For any fixed $|\xi|\in(0,|\xi|_c)$, let $\alpha(s)=\alpha(|\xi|;s)$ be defined by \eqref{min}. Then the following hold.
\begin{enumerate}
 \item $\alpha\in C_{loc}^{0,1}((0,\infty))\cap C^0((0,\infty))$ and $\alpha(s)$ is strictly increasing in $s$.
 \item There exist constants  $C_0,C_1,C_2>0$ depending on  $\rho_\pm,$ $\mu_\pm,$ $\sigma_\pm,$ $b,$ $g,$ $|\xi|$ so that
\begin{equation}\label{s1}
\alpha(s)\le-C_0+sC_1,
\end{equation}
and
\begin{equation}\label{s2}
\alpha(s)\ge-\frac{2 g\rj }{\rho_-}|\xi|+sC_2.
\end{equation}
\end{enumerate}

\end{lemma}

\begin{proof}
We decompose the energy \eqref{ener} as
\begin{equation}\label{decomp}
E(\psi;|\xi|,s)=sE_1(\psi;|\xi|)+E_0(\psi;|\xi|),
\end{equation}
where
\begin{equation}
E_1(\psi;|\xi|):=\frac{1}{2}\int_{-b}^1 \mu(4|\xi|^2|\psi'|^2+||\xi|^2\psi+\psi''|^2),
\end{equation}
\begin{equation}
E_0(\psi;|\xi|):=\frac{1}{2}|\xi|^2(\sigma_+|\xi|^2+g\rho_+)|\psi(1)|^2
+\frac{1}{2}|\xi|^2(\sigma_-|\xi|^2-g\rj )|\psi(0)|^2.
\end{equation}
The decomposition \eqref{decomp} keeps the same form as in Proposition 3.6 of \cite{3GT2}, hence $(1)$ follows in the same way.

Now we prove $(2)$. First, for \eqref{s1}, note that for any fixed $0<|\xi|<|\xi|_{c}$ there exists $ \psi_{|\xi|}$ such that $C_0=-E_0( \psi_{|\xi|};|\xi|)>0$. Then we have that $E(\psi_{|\xi|};|\xi|,s)\le -C_0+sC_1$ for some $C_1>0$, and hence \eqref{s1} holds. Next,  for \eqref{s2}, observe that for any $\psi\in \mathcal{C}$, we have
\begin{equation}\label{kkk}
\begin{split}
-\frac{|\xi|^2g\rj |\psi(0)|^2}{2}
&=-|\xi|^2g\rj \int_{-b}^0\psi'\psi\,dx_3
\\&\ge
-\frac{|\xi|g\rj }{\rho_-}\left(\int_{-b}^0\rho_-|\xi|^2|\psi|^2\,dx_3\right)^\frac{1}{2}\left(\int_{-b}^0\rho_-|\psi'|^2\,dx_3\right)^\frac{1}{2}\ge
-\frac{2|\xi|g\rj }{\rho_-}.
\end{split}
\end{equation}
Since the other terms in the energy $E$ are nonnegative, we have
\begin{equation}
\alpha(s)\ge-\frac{2|\xi|g\rj }{\rho_-}+s\inf_{\psi\in \mathcal{C}}E_1(\psi).
\end{equation}
We denote by $C_2$ this latter positive infimum; then \eqref{s2} follows and we conclude our lemma.
\end{proof}

For any fixed $|\xi|\in (0,|\xi|_c)$, we then define the open set
\begin{equation}
\mathcal{S}=\alpha^{-1}((-\infty,0))\subset(0,\infty).
\end{equation}
Note that $\mathcal{S}$ is non-empty and allows us to define $\lambda(s)=\sqrt{-\alpha(s)}$ for $s\in\mathcal{S}$. We next show that there is a unique fix point $s\in \mathcal{S}$ so that $s=\lambda(|\xi|;s)$.

\begin{lemma}\label{fix}
There exists a unique $s\in \mathcal{S}$ so that $\lambda(|\xi|;s)=\sqrt{-\alpha(|\xi|; s)}>0$ and
\begin{equation}\label{fixedpoint}
s=\lambda(|\xi|;s).
\end{equation}
\end{lemma}

\begin{proof}
By Lemma \ref{alpha_s}, there exists $s_\ast>0$ such that
\begin{equation}
\mathcal{S}=\alpha^{-1}((-\infty,0))=(0,s_\ast).
\end{equation}
We define $\lambda=\sqrt{-\alpha}$ on $\mathcal{S}$ and define the function $\Phi:\ (0,s_\ast)\rightarrow(0,\infty)$ by
\begin{equation}
\Phi(s)=s/\lambda(|\xi|;s),
\end{equation}
which is continuous and strictly increasing in $s$. Moreover, $\lim_{s\rightarrow0}\Phi(s)=0$ and $\lim_{s\rightarrow s_\ast}\Phi(s)=+\infty$. Hence there is unique $s\in (0,s_\ast)$ so that $\Phi(s)=1$, which gives \eqref{fixedpoint}.
\end{proof}

Notice that in Lemma \ref{fix}, for any fixed $|\xi|\in(0,|\xi|_{c})$ the fixed point $s=s(|\xi|)\in\mathcal{S}$ is unique, we may write uniquely $\lambda(|\xi|)$ within $(0,|\xi|_{c})$, while the corresponding solution to \eqref{linear4} is written by $\psi_{|\xi|}$. We have the following behavior of $\lambda(|\xi|)$.

\begin{Proposition}\label{lambdabehavior}
The function $\lambda:(0,|\xi|_c)\to (0,\infty)$ satisfies the bound
\begin{equation}\label{bound}
\lambda(|\xi|)\le \frac{b g \rj }{4\mu_-}.
 \end{equation}
 Moreover,
 \begin{equation}
 \lim_{|\xi|\rightarrow0}\lambda(|\xi|)=0\hbox{ for }\sigma_\pm\ge0,\hbox{ and }
 \lim_{|\xi|\rightarrow|\xi|_c}\lambda(|\xi|)=0\hbox{ if }\sigma_->0.
 \end{equation}
\end{Proposition}

\begin{proof}
 We shall use the fact that $E(\psi_{|\xi|};|\xi|,\lambda(|\xi|))=-\lambda^2(|\xi|)<0$. First, we have
\begin{equation}\label{lkk}
 \lambda(|\xi|)E_1(\psi_{|\xi|})<  \frac{|\xi|^2g\rj |\psi_{|\xi|}(0)|^2}{2},
\end{equation}
where $E_1(\psi_{|\xi|}):=E_1(\psi_{|\xi|};|\xi|,\lambda(|\xi|))$. Since $\psi_{|\xi|}(-b)=0$, we have
\begin{equation}
\psi_{|\xi|}(0)=\int_{-b}^0 \psi_{|\xi|}' \,dx_3 \le \sqrt{b} \left(\int_{-b}^0 |\psi_{|\xi|}'|^2\,dx_3\right)^{1/2}
=\frac{\sqrt{b}}{\sqrt{2\mu_-|\xi|^2}}\left(\int_{-b}^0
2\mu_-|\xi|^2|\psi_{|\xi|}'|^2\,dx_3\right)^{1/2} ,
\end{equation}
which implies that
\begin{equation}
 \frac{|\xi|^2 g\rj |\psi_{|\xi|}(0)|^2}{2}\le  \frac{b g\rj }{4\mu_-}E_1(\psi_{|\xi|}).
\end{equation}
This together with \eqref{lkk} implies \eqref{bound}.

Next, we derive the limit behaviors of $\lambda(|\xi|)$. First, by \eqref{s2}, we have
\begin{equation}
0\le \lambda^2(|\xi|)\le \frac{2 g\rj }{\rho_-}|\xi|,
\end{equation}
which implies that $\lim_{|\xi|\rightarrow0}\lambda(|\xi|)=0$ for $\sigma_\pm\ge0$. On the other hand, by \eqref{kkk}, we have
\begin{equation}
  |\psi_{|\xi|}(0)|^2\le \frac{4  }{\rho_-|\xi| }.
 \end{equation}
Hence, we obtain
\begin{equation}
\lambda^2(|\xi|)\le \frac{|\xi|^2(g\rj -\sigma_-|\xi|^2)}{2} { |\psi_{|\xi|}(0)|^2}  \le \frac{2|\xi|(g\rj -\sigma_-|\xi|^2)}{\rho_-} ,
 \end{equation}
which implies that $\lim_{|\xi|\rightarrow|\xi|_c}\lambda(|\xi|)=0$ for $\sigma_->0$.
\end{proof}

By Proposition \ref{lambdabehavior},  we can then define
\begin{equation} \label{Lambda}
0<\Lambda:= \sup_{0<|\xi|< |\xi|_c} {\lambda(|\xi|)}<\infty .
\end{equation}
For $\sigma_->0$, there is only a finite number of spatial frequencies $\xi\in (L_1^{-1}\mathbb{Z})\times (L_2^{-1}\mathbb{Z})$ satisfying $|\xi|<|\xi|_c$, so the the largest growth rate $\Lambda$  must be achieved when $0<\sigma_- <\sigma_c$. For $\sigma_-=0$ it is not clear whether $\Lambda$ is achieved.  However, we can achieve a growth rate that is arbitrarily close to $\Lambda$, and so in particular $\Lambda_\ast$ is achieved, where
\begin{equation} \label{littlelambda}
0<\Lambda/2<\Lambda_\ast\le \Lambda.
\end{equation}

We may now construct a growing mode solution to the linearized problem \eqref{linear}.
\begin{theorem}\label{growingmode}
Let $\Lambda$ be defined by \eqref{Lambda} and $\Lambda_\ast$ be defined by \eqref{littlelambda}.  Then the following hold.
\begin{enumerate}
 \item Let $0< \sigma_- < \sigma_c$.  Then there is a growing mode solution to \eqref{linear} so that
\begin{equation}\label{gro1}
\|\eta(t)\|_{H^k}=e^{\Lambda t}\|\eta(0)\|_{H^k},\quad \|u(t)\|_{H^k}=e^{\Lambda t}\|u(0)\|_{H^k},\hbox{ for any }k\ge0.
\end{equation}

 \item  Let $\sigma_-=0$. Then there is a growing mode solution to \eqref{linear} so that
\begin{equation}\label{gro2}
\|\eta(t)\|_{k}=e^{\Lambda_\ast t}\|\eta(0)\|_{k},\quad \|u(t)\|_{k}=e^{\Lambda_\ast t}\|u(0)\|_{k},\hbox{ for any }k\ge0.
\end{equation}
\end{enumerate}
\end{theorem}

\begin{proof}
Let $|\xi|>0$ be so that $\lambda(|\xi|)=\Lambda$ for $\sigma_->0$ or $\lambda(|\xi|)=\Lambda_\ast$ for $\sigma_-=0$.  Let $\psi_{|\xi|}$ be the corresponding solution to \eqref{linear4} with $\lambda(|\xi|)$, which is constructed above by the minimization problem \eqref{min}. We then define a solution to \eqref{linear3} as described immediately after \eqref{linear4}, which then allows us to define $u$, $p$, and $\eta$ according to \eqref{ansatz}, \eqref{ansatz2}, and  ${\eta}=\lambda^{-1}u_3|_{\Sigma}$. Then we have that $u\in {}_0H_\sigma^1(\Omega)\cap \ddot{H}^k(\Omega)$, $p \in \ddot{H}^k(\Omega)$, and $\eta\in H^k(\Sigma)$  for any $k \ge 0$ and $(u,p,\eta)$ solve the linearized problem \eqref{linear}.  Moreover, $u,\eta$ satisfy \eqref{gro1} or \eqref{gro2}.
\end{proof}

\subsection{Sharp growth rate}

In this subsection, we will show that $\Lambda$ defined by \eqref{Lambda} is the sharp growth rate of arbitrary solutions to the linearized problem \eqref{linear}.  Since the spectrum of the linear operator is  complicated, it is hard to obtain the largest growth rate of the solution operator in ``$L^2$'' in the usual way. Instead, motivated by \cite{3GT2,JT}, we use careful energy estimates to show that $e^{\Lambda t}$ is the sharp growth rate in a slightly weaker sense (cf. \eqref{l21}).

First, we have the following energy identity.
\begin{lemma}\label{identity}
Let  $(u,p,\eta)$ solve \eqref{linear}, then
\begin{equation}\label{energyidentity}
\begin{split} \frac{1}{2} &\frac{d}{dt}
\left(\int_\Omega\rho |\partial_t u|^2+\int_{\Sigma_+}
\sigma_+|\nabla_\ast u_3|^2+\rho_+g|u_3|^2+\int_{\Sigma_-}
\sigma_-|\nabla_\ast
u_3|^2-\rj g|u_3|^2\right)
\\&+\frac{1}{2}\int_\Omega\mu
|\mathbb{D}\partial_t u|^2=0.
\end{split}
\end{equation}
\end{lemma}

\begin{proof}
We differentiate $\eqref{linear}_1$ in time, multiply the resulting equation by $\partial_{t} u$ and  then integrate by parts over $\Omega$. By using the other conditions in $\eqref{linear}$, we obtain \eqref{energyidentity}.
\end{proof}

The next result allows us to estimate the energy in terms of $\Lambda$.

\begin{lemma} \label{varineq}
Let  $u\in{}_0H_\sigma^1(\Omega)$ with $\sigma\nabla_\ast u_3\in L^2(\Sigma)$, then we have the inequality
\begin{equation}
\begin{split}&\frac{1}{2} \int_{\Sigma_+}
\sigma_+|\nabla_\ast
u_3|^2+\rho_+g|u_3|^2+\frac{1}{2}\int_{\Sigma_-}
\sigma_-|\nabla_\ast u_3|^2-\rj g|u_3|^2
\\ &\quad\ge-
\frac{\Lambda^2}{2}\int_\Omega\rho
|u|^2-\frac{\Lambda}{4}\int_\Omega\mu |\mathbb{D}u|^2.\label{i0}
\end{split}
\end{equation}
\end{lemma}

\begin{proof}
Let $\hat{f}$ be   the horizontal Fourier transform of $f$. By the Parseval theorem, we have
\begin{equation}\label{i1}
\begin{split}
& \frac{1}{2} \int_{\Sigma_+}
\sigma_+|\nabla_\ast
u_3|^2+\rho_+g|u_3|^2+\frac{1}{2}\int_{\Sigma_-}
\sigma_-|\nabla_\ast u_3|^2-\rj g|u_3|^2
\\&\quad=\frac{1}{4\pi^2}\times \sum_{\xi\in L_1^{-1}\mathbb{Z}\times
L_2^{-1}\mathbb{Z}}\left\{\frac{1}{2}  (\sigma_+|\xi|^2
+\rho_+g)|\hat{u}_3(1)|^2+ \frac{1}{2}(\sigma_-|\xi|^2
-\rj g)|\hat{u}_3(0)|^2 \right\}.
\end{split}
\end{equation}
Noticing that for $\xi=0$,
\begin{equation}
\hat{u}_3(0)=\int_{\Sigma_-}u_3=\int_{\Omega_-}{\rm div }u =0
\end{equation}
and then
\begin{equation}
\hat{u}_3(1)=\int_{\Sigma_+}u_3=\int_{\Omega_+}{\rm div }u =0,
\end{equation}
we may reduce \eqref{i1} to be
\begin{equation}\label{i2}
\begin{split} & \frac{1}{2} \int_{\Sigma_+}
\sigma_+|\nabla_\ast
u_3|^2+\rho_+g|u_3|^2+\frac{1}{2}\int_{\Sigma_-}
\sigma_-|\nabla_\ast u_3|^2-\rj g|u_3|^2
\\&\quad=\frac{1}{4\pi^2}\times \sum_{0\neq\xi\in L_1^{-1}\mathbb{Z}\times
L_2^{-1}\mathbb{Z}}\left\{\frac{1}{2}  (\sigma_+|\xi|^2
+\rho_+g)|\hat{u}_3(1)|^2+ \frac{1}{2}(\sigma_-|\xi|^2
-\rj g)|\hat{u}_3(0)|^2\right\}.
\end{split}
\end{equation}

Now for any fixed $\xi\neq 0$,  writing
\begin{equation}
\hat{u}_1(x_3)=-i\varphi(\xi,
x_3),\ \hat{u}_2(x_3)=-i\theta(\xi,x_3),\ \hat{u}_3(x_3)=\psi(\xi,
x_3)
\end{equation}
we  have $\xi_1 \varphi+\xi_2\theta+\psi'=0$. The right hand side of \eqref{i2} and the constraint $\xi_1 \varphi+\xi_2\theta+\psi'=0$ are  obviously invariant under simultaneous rotations of $\xi$ and $(\varphi,\theta)$, so without loss of generality we may assume that $\xi = (|\xi| , 0)$ with $|\xi| > 0$ and $\theta = 0$. Noting that for $|\xi|\in (0,|\xi|_c)$, by the definition of \eqref{ener}, we have
\begin{equation}\label{i3}
\begin{split}
\frac{1}{2}  (&\sigma_+|\xi|^2 +\rho_+g)|\hat{u}_3(1)|^2+
\frac{1}{2}(\sigma_-|\xi|^2 -\rj g)|\hat{u}_3(0)|^2
\\& =\frac{1}{2}  (\sigma_+|\xi|^2
+\rho_+g)|\psi(1)|^2+ \frac{1}{2}(\sigma_-|\xi|^2
-\rj g)|\psi(0)|^2
\\&  =\frac{1}{|\xi|^2}\left(
E(\psi;|\xi|,\Lambda)-\frac{1}{2}\int_{-b}^1
\Lambda\mu(4|\xi|^2|\psi'|^2+||\xi|^2\psi+\psi''|^2)\,dx_3\right)
\\& \ge \frac{1}{|\xi|^2}\left(-\Lambda^2
J(\psi;|\xi|)-\frac{\Lambda}{2}\int_{-b}^1
\mu(4|\xi|^2|\psi'|^2+||\xi|^2\psi+\psi''|^2)\,dx_3\right)
\\& = -\frac{\Lambda^2}{2}\int_{-b}^1
\rho(|\psi|^2+ | \varphi |^2)\,dx_3
-\frac{\Lambda}{2}\int_{-b}^1 \mu(4
|\psi'|^2+ ||\xi| \psi-
\varphi'|^2)\,dx_3
 .\end{split}
\end{equation}
Here in the inequality above we have used the  following variational characterization for $\Lambda$, which follows directly from the definitions \eqref{min} and \eqref{Lambda},
\begin{equation}
E(\psi;|\xi|, \Lambda)\ge -\Lambda^2
J(\psi;|\xi|),\hbox{ for any } 0<|\xi|< |\xi|_c \hbox{ and any
}\psi\in {}_0H^2((-b,1)).
\end{equation}
Note that for $|\xi|\ge |\xi|_c$ the left hand side of \eqref{i3} is nonnegative, so \eqref{i3} holds trivially, and so we deduce that it holds for all $0 \neq \xi \in L_1^{-1} \mathbb{Z} \times L_2^{-1} \mathbb{Z}$.

Plugging \eqref{i3} into \eqref{i2} and translating the resulting inequality back to the original notation for fixed $\xi$,  by the Fubini and Parseval theorems, we find that
\begin{equation}
\begin{split}  & \frac{1}{2} \int_{\Sigma_+}
\sigma_+|\nabla_\ast
u_3|^2+\rho_+g|u_3|^2+\frac{1}{2}\int_{\Sigma_-}
\sigma_-|\nabla_\ast u_3|^2-\rj g|u_3|^2
\\&\quad \ge -\frac{1}{4\pi^2}\times \sum_{\xi\in L_1^{-1}\mathbb{Z}\times
L_2^{-1}\mathbb{Z}}\frac{\Lambda^2}{2}\int_{-b}^1
\rho|\hat{u}(x_3)|^2\,dx_3 +\frac{\Lambda}{4}\int_{-b}^1 \mu |
\widehat{\mathbb{D}(u)} (x_3)|\,dx_3
  \\&\quad=-
\frac{\Lambda^2}{2}\int_\Omega\rho
|u|^2-\frac{\Lambda}{4}\int_\Omega\mu |\mathbb{D}u|^2.
\end{split}
\end{equation}
This is \eqref{i0} and we conclude our lemma.
\end{proof}

Now we can prove our main result of the subsection.
\begin{theorem} \label{lineargrownth}
Let  $(u,p,\eta)$ solve \eqref{linear}. Then we have the following estimates for $t \ge 0$:
\begin{equation}\label{result1}
\|u(t)\|_{1}^2+\|\partial_tu(t)\|_0^2+\int_0^t\| u(s)\|_{1}^2\,ds \le
C e^{2\Lambda t} \| u(0)\|_2^2,
\end{equation}
and
\begin{equation}\label{result2}
\| \eta(t)\|_{0}^2+\| \partial_t\eta(t)\|_{1/2}^2+\int_0^t\|
  \partial_s\eta(s)\|_{1/2}^2\,ds  \le
C   e^{\Lambda t}(\| \eta(0)\|_{0}^2+\|
u(0)\|_2^2).
\end{equation}
\end{theorem}
\begin{proof}
Integrating the result of Lemma \ref{identity} in time from $0$ to $t$, by Lemma \ref{varineq}, we get
\begin{equation}\label{j111}
\begin{split}
& \frac{1}{2}
\int_\Omega\rho |\partial_t u|^2+\frac{1}{2}\int_0^t\int_\Omega\mu
|\mathbb{D}\partial_t u|^2
\\&\quad=K_0-\frac{1}{2}
\left(\int_{\Sigma_+} \sigma_+|\nabla_\ast
u_3|^2+\rho_+g|u_3|^2+\int_{\Sigma_-} \sigma_-|\nabla_\ast
u_3|^2-\rj g|u_3|^2\right)
\\&\quad \le K_0+
\frac{\Lambda^2}{2}\int_\Omega\rho
|u|^2+\frac{\Lambda}{4}\int_\Omega\mu |\mathbb{D}u|^2
\end{split}
\end{equation}
with
\begin{equation}
K_0=\frac{1}{2} \left(\int_\Omega\rho|\partial_t u(0)|^2+\int_{\Sigma_+} \sigma_+|\nabla_\ast
u_3(0)|^2+\rho_+g|u_3(0)|^2+\int_{\Sigma_-} \sigma_-|\nabla_\ast
u_3(0)|^2-\rj g|u_3(0)|^2\right).
\end{equation}
To compactly rewrite the previous inequality, we denote the weighted norms by
\begin{equation}
\|u\|_\star^2:=\int_\Omega\rho |  u|^2\hbox{ and }\|u\|_{\star\star}^2:=\frac{1}{2}\int_\Omega\mu
|\mathbb{D}  u|^2
\end{equation}
which are equivalent to $\|\cdot\|_0^2$ and $\|\cdot\|_1^2$ respectively. We denote the corresponding inner-products by $\langle \cdot,\cdot\rangle_{\ast}$, etc.  Then \eqref{j111} reads as
\begin{equation}\label{j4}
\frac{1}{2}\|\partial_t  u(t) \|_\star^2 + \int_0^t\|\partial_t  u(s)\|_{\star\star}^2 ds \le K_0
+\frac{\Lambda^2}{2}\| u(t) \|_\star^2 +  \frac{\Lambda}{2}\| u(t) \|_{\star\star}^2.
\end{equation}

Integrating in time and using Cauchy's inequality, we may bound
\begin{equation}\label{j222}
\begin{split}
\Lambda\|u(t)\|_{\star\star}^2
&=\Lambda\|u(0)\|_{\star\star}^2+ \Lambda\int_0^t 2 \langle
u(s),\partial_s u(s) \rangle_{\star\star}\,ds
\\&\le
\Lambda\|u(0)\|_{\star\star}^2+\int_0^t\|\partial_su(s)\|_{\star\star}^2\,ds
+\Lambda^2\int_0^t\|u(s)\|_{\star\star}^2\,ds.
\end{split}
\end{equation}
On the other hand
\begin{equation}\label{j333}
\Lambda\partial_t\|u(t)\|_{\star}^2=2 \Lambda \langle
u(t),\partial_t u(t) \rangle_{\star}\le \|\partial_tu(t)\|_{\star}^2
+\Lambda^2\|u(t)\|_{\star}^2.
\end{equation}
Hence, combining \eqref{j222}--\eqref{j333} with \eqref{j4}, we derive the differential inequality
\begin{equation}\label{j5}
\partial_t\|u(t)\|_\star^2+\|u(t)\|_{\star\star}^2\le K_1+2\Lambda \left( \|u(t)\|_\star^2+\int_0^t\|u(s)\|_{\star\star}^2 \,ds \right)
\end{equation}
for $K_1=2K_0/\Lambda+2\|u(0)\|_{\star\star}^2$.  An application of Gronwall's inequality to \eqref{j5} yields
\begin{equation}\label{j6}
\|u(t)\|_\star^2+\int_0^t\|u(s)\|_{\star\star}^2
\le e^{2\Lambda t}\|u(0)\|_\star^2+\frac{K_1}{2\Lambda}( e^{2\Lambda
t}-1).
\end{equation}
Now plugging \eqref{j6} and \eqref{j222} into \eqref{j4}, we find that
\begin{equation}\label{j7}
\frac{1}{\Lambda}\|\partial_tu(t)\|_\star^2+\|u(t)\|_{\star\star}^2
\le K_1+\Lambda \|u(t)\|_\star^2+2\Lambda \int_0^t
\|u(s)\|_{\star\star}^2\,ds \le  e^{2\Lambda
t}(2\Lambda\|u(0)\|_\star^2+K_1).
\end{equation}
By the trace theorem and using the equations \eqref{linear}, we have
\begin{equation}
K_0,K_1\lesssim
\| u(0)\|_2^2+\|\partial_t u(0)\|_0^2\lesssim \|
u(0)\|_2^2.
\end{equation}
and so \eqref{j7} implies \eqref{result1}.

To prove \eqref{result2},  we use \eqref{result1} together with the kinematic boundary condition $\dt \eta = u_3$ and the  trace theorem to estimate
\begin{equation}\label{j8}
\begin{split}\| \partial_t\eta(t)\|_{1/2}^2+\int_0^t\|
  \partial_s\eta(s)\|_{1/2}^2\,ds
 &=   \|u_3(t)\|_{H^{1/2}(\Sigma)}^2+\int_0^t\| u_3(s)\|_{H^{1/2}(\Sigma)}^2\,ds
\\&\lesssim   \|u_3(t)\|_{1}^2+\int_0^t\| u_3(s)\|_{1}^2\,ds
 \lesssim
 e^{2\Lambda t}\| u(0)\|_2^2,
 \end{split}
\end{equation}
and then by \eqref{j8},
 \begin{equation}\label{j9}
 \| \eta(t)\|_{0}
 \le  \| \eta(0)\|_{0} +\int_0^t \|\partial_s\eta(s)\|_0 \,ds
\lesssim  e^{\Lambda t}(\| \eta(0)\|_{0}+\| u(0)\|_2).
\end{equation}
Hence, \eqref{j8} and \eqref{j9} imply \eqref{result2}.
\end{proof}

\section{Nonlinear energy estimates}\label{energy}

\subsection{Energy estimates with surface tension}

In this subsection, we will derive the nonlinear energy estimates for the system \eqref{surface}. For this, we define the energy and dissipation, the definitions of which rely on the linear energy identity of the homogeneous form of \eqref{surface}:
\begin{equation}
\begin{split}
&\frac{d}{dt}\left(\frac{1}{2}\int_\Omega \rho |u|^2 +\frac{1}{2}\int_{\Sigma_+}
\rho_+g|\eta_+|^2+\sigma_+|\nabla_\ast\eta_+|^2+\frac{1}{2}\int_{\Sigma_-}
-\rj g|\eta_-|^2+\sigma_-|\nabla_\ast\eta_-|^2\right)
\\
&\quad+\frac{1}{2}\int_{\Omega}\mu|\mathbb{D} u|^2=0.
\end{split}
\end{equation}
According to this energy identity and the structure of the equations \eqref{surface}, we define the instantaneous energy (a higher regularity version of what we used in \cite{WTK}) as
\begin{equation}\label{energy00}
  \mathcal{E}:=\sum_{\ell=0}^2\|\partial_t^\ell u\|_{4-2\ell}^2+\sum_{\ell=0}^1\|\partial_t^\ell p\|_{4-2\ell-1}^2+\|\eta\|_{  5 }^2
  +\sum_{\ell=1}^3
  \|\partial_t^\ell\eta\|_{ 4-2\ell+3/2  }^2
  \end{equation}
and the corresponding dissipation rate as
\begin{equation}\label{dissipation00}
\begin{split}
 \mathcal{D}:=
&\sum_{\ell=0}^2\|\partial_t^\ell
u\|_{4-2\ell+1}^2+\sum_{\ell=0}^1\|\partial_t^\ell p\|_{4-2\ell}^2
+\|\eta\|_{   11/2  }^2+\sum_{\ell=1}^3
  \|\partial_t^\ell\eta\|_{ 4-2\ell+5/2  }^2
\\&  +\|\rho \partial_t^3
u\|_{-1}^2+\|\nabla\partial_t^2u\|_{ H^{-1/2}(\Sigma_+)
}^2+\|\llbracket\mu\nabla\partial_t^2u\rrbracket\|_{
H^{-1/2}(\Sigma_-) }^2
\\&    +\|\partial_t^2p\|_{ 0
}^2+\|
\partial_t^2 p\|_{ H^{-1/2}(\Sigma_+) }^2+\|\llbracket \partial_t^2
p\rrbracket\|_{ H^{-1/2}(\Sigma_-) }^2 .
\end{split}
\end{equation}
We recall that in these definitions we have abused notation by writing $\|\cdot \|_{-1} := \|\cdot\|_{\Hd}$.

We will now derive our a priori estimates.  Throughout the rest of this subsection we will assume that $\mathcal{E}(t)\le \delta^2$ for some sufficiently small $\delta>0$ and for all $t$ in  the interval in which the solution is defined.  We will implicitly allow $\delta$ to be made smaller in each result, but we will reiterate the smallness of $\delta$ in our main result.

To begin with, in the following lemma we present the estimates of the nonlinear terms $f$ and $g$, defined by \eqref{f}--\eqref{g_-^3}, in terms of $\mathcal{E}$ and $\mathcal{D}$.

\begin{lemma}
Let $f$ and $g$ be defined by \eqref{f}--\eqref{g_-^3}, then we have
\begin{equation}\label{nonl1}
\sum_{\ell=0}^1\left(\|\partial_t^\ell
f\|_{4-2\ell-2}^2+\|\partial_t^\ell
g\|_{4-2\ell-3/2}^2\right)\le C\mathcal{E}^2,
\end{equation}
and
\begin{equation}\label{nonl2}
\sum_{\ell=0}^2\left(\|\partial_t^\ell
f\|_{4-2\ell-1}^2+\|\partial_t^\ell
g\|_{4-2\ell-1/2}^2\right)\le C\mathcal{E} \mathcal{D}.
\end{equation}
\end{lemma}

\begin{proof}
The full expressions that define $f$ and $g$   are rather complicated, and a full analysis of each term would be tedious.  As such, we will identify only the principal terms in the expressions and provide details for only the most delicate estimates.  The other terms (lower order terms) may be handled through straightforward modifications of the arguments presented here.

The principal terms appearing in $f$ and $g$ may be identified by first examining the regularity of the terms appearing in $\mathcal{E}$ and $\mathcal{D}$, as defined by \eqref{energy00}--\eqref{dissipation00}, and by appealing to Lemma \ref{Poi} to compare the regularity of $\bar{\eta}$ to $\eta$.  We find that, roughly speaking, the regularity of $u$ (and its time derivatives) is same as $\partial_t\bar{\eta}$, one order higher than $p$ and at least one order lower than $\bar{\eta}$. Based on this, we identify the principal terms in $f$ and $g$ defined by \eqref{f}--\eqref{g_-^3} as
\begin{equation}\label{app1}
f\sim \nabla^3\bar{\eta}    u+\nabla^2\bar{\eta} \nabla  u+\nabla\bar{\eta}\, \mu\nabla^2 u
\end{equation}
and
\begin{equation}\label{app2}
g_+\sim \nabla_\ast\eta_+\nabla u_++\nabla_\ast^2\eta_+u_+,\quad
g_- \sim \nabla_\ast\eta_-\Lbrack\mu\nabla u\Rbrack+\nabla_\ast^2\eta_-u.
\end{equation}

We will prove only the most involved estimate, namely,
\begin{equation} \label{non111}
\|\partial_t^2
f\|_{-1}^2+\|\partial_t^2
g\|_{-1/2}^2 \lesssim\mathcal{E} \mathcal{D}.
\end{equation}
The other  estimates may be derived through somewhat easier arguments.  By Lemma \ref{Poi}--\ref{-1norm}, we have
\begin{equation}
\begin{split}
 \|\partial_t^2
f\|_{-1} :=\|\partial_t^2 f\|_{\Hd}& \lesssim
\| \partial_t^2 \nabla^3\bar{\eta}\|_{0} \|   u\|_1+ \| \partial_t  \nabla^3\bar{\eta}\|_0\|\partial_t u\|_1
+\| \nabla^3\bar{\eta}\|_0  \|\partial_t^2  u\|_1
\\&\quad+\|\partial_t^2\nabla^2\bar{\eta} \|_0 \|\nabla u\|_1
+\|\partial_t\nabla^2\bar{\eta}\|_0\| \partial_t\nabla  u\|_1
+\|\nabla^2\bar{\eta}\|_1\| \partial_t^2\nabla  u\|_0
\\&\quad+\|\partial_t^2\nabla\bar{\eta} \|_0\|\nabla^2 u\|_1
  +\|\partial_t\nabla\bar{\eta} \|_1 \|\partial_t\nabla^2 u\|_0
  +\|\nabla\bar{\eta} \|_3\|\mu\partial_t^2\nabla^2 u\|_{\Hd}
\\&\lesssim
\|
\partial_t^2 \eta \|_{5/2} \|   u\|_1+\| \partial_t   {\eta}\|_{5/2}\|\partial_t u\|_1
+\|  {\eta}\|_{5/2} \|\partial_t^2  u\|_1
\\&\quad+\|\partial_t^2 {\eta} \|_{3/2} \| u\|_2
+\|\partial_t {\eta}\|_{3/2}\| \partial_t   u\|_2
+\| {\eta}\|_{5/2}\| \partial_t^2   u\|_1
\\&\quad+\|\partial_t^2 {\eta} \|_{1/2}\| u\|_2
  +\|\partial_t {\eta} \|_{3/2}\|\partial_t u\|_2
  +\| {\eta} \|_{7/2}\|\mu\partial_t^2\nabla^2 u\|_{\Hd}
\\&\lesssim \sqrt{\mathcal{E} \mathcal{D}}.
\end{split}
\end{equation}
Here, in the last inequality, we have used the  estimate
\begin{equation}\label{a_bound}
 \|\mu\partial_t^2 \nabla^2 u\|_{\Hd} \ls \|\partial_t^2 u\|_{1} + \|\partial_t^2 \nabla u_+\|_{H^{-1/2}(\Sigma_+)}
+ \|\partial_t^2 \Lbrack\mu \nabla u\Rbrack \|_{H^{-1/2}(\Sigma_-)}\ls \sqrt{\mathcal{D}}.
\end{equation}
Indeed,   by  H\"{o}lder's inequality and the trace theorem, we obtain that for any $\varphi\in \H(\Omega)$ and any $i,j=1,2,3$,
\begin{equation}
\begin{split}
&\langle\mu\partial_t^2 \partial_i \partial_j u,\varphi\rangle_\ast = -\int_\Omega \mu \partial_t^2\partial_i u \cdot \partial_j \varphi \,dx
+\mu_+ \int_{\Sigma_+} \partial_t^2 \partial_i u_+ \cdot \varphi  (e_3\cdot e_j) - \int_{\Sigma_-}\partial_t^2\Lbrack\mu\partial_i u \Rbrack \cdot \varphi (e_3\cdot e_j)
\\&\quad
\lesssim \norm{\partial_t^2u}_1\norm{\varphi}_1+\|\partial_t^2 \nabla u_+\|_{H^{-1/2}(\Sigma_+)}\|\varphi\|_{H^{1/2}(\Sigma_+)}
+\|\partial_t^2\Lbrack\mu\nabla u\Rbrack\|_{H^{-1/2}(\Sigma_-)}\|\varphi\|_{H^{1/2}(\Sigma_-)}
\\&\quad
\lesssim \left(\norm{\partial_tu}_1+\|\partial_t^2 \nabla u_+\|_{H^{-1/2}(\Sigma_+)}
+\|\partial_t^2\Lbrack\mu\nabla u\Rbrack\|_{H^{-1/2}(\Sigma_-)}\right)\norm{\varphi}_1.
\end{split}
\end{equation}
Taking the supremum over such $\varphi$ with $\|\varphi\|_1\le 1$, we get \eqref{a_bound}.

A similar application of these lemmas, together with trace estimates, implies that
\begin{equation}
\begin{split} \|\partial_t^2
g_-\|_{-1/2} &\lesssim
\|\partial_t^2\nabla_\ast\eta_-\|_0\|\Lbrack\mu\nabla
u\Rbrack\|_{L^2(\Sigma)}+\|\partial_t\nabla_\ast\eta_-\|_0\|\Lbrack\mu\partial_t\nabla
u\Rbrack\|_{L^2(\Sigma)}
\\&\quad+\|\nabla_\ast\eta_-\|_2\|\Lbrack\mu\partial_t^2\nabla
u\Rbrack\|_{H^{-1/2}(\Sigma)}
+
 \|\partial_t^2\nabla_\ast^2\eta_-\|_0\|u\|_{L^2(\Sigma)}
\\&\quad+\|\partial_t\nabla_\ast^2\eta_-\|_0\|\partial_tu\|_{L^2(\Sigma)}
 +\|\nabla_\ast^2\eta_-\|_0\|\partial_t^2u\|_{L^2(\Sigma)}
\\ &\lesssim
\|\partial_t^2 \eta_-\|_1\|
u \|_{2}+\|\partial_t \eta_-\|_1\| \partial_t
u \|_{2}+\| \eta_-\|_3\|\Lbrack\mu\partial_t^2\nabla
u\Rbrack\|_{H^{-1/2}(\Sigma)}
\\&\quad
+
 \|\partial_t^2 \eta_-\|_2\|u\|_{1}
+\|\partial_t \eta_-\|_2\|\partial_tu\|_{1}
 +\| \eta_-\|_2\|\partial_t^2u\|_{1}
\\&\lesssim \sqrt{\mathcal{E}\mathcal{D}},
\end{split}
\end{equation}
and similarly, $\|\partial_t^2 g_+\|_{-1/2}^2\lesssim \mathcal{E}\mathcal{D}$. Hence, \eqref{non111} follows.
\end{proof}

\subsubsection{Energy evolution}

We first derive the energy evolution of the pure temporal derivatives.
\begin{lemma} \label{nosur1}
Let $(u,p,\eta)$ solve \eqref{surface}.  Then we have the estimate
\begin{equation}\label{0es}
\begin{split}
&\sum_{\ell=0}^2\left(\|\partial_t^\ell u(t)\|_{0}^2+\|\partial_t^\ell \eta(t)\|_{0}^2\right)
+\int_0^t \sum_{\ell=0}^2 \|\partial_t^\ell u(s)\|_{1}^2\,ds
\\&\quad \le C \mathcal{E} (0)+ C\int_0^t\mathcal{E}(s)
\mathcal{D}(s)\,ds+C\int_0^t \|\eta(s)\|_0^2\,ds.
\end{split}
\end{equation}
\end{lemma}

\begin{proof}
We multiply $\eqref{surface}_1$ by $u$ and then integrate by parts respectively in $\Omega_+$ and $\Omega_-$; by using the other conditions in $\eqref{surface}$, together with  duality and  trace theory, we then obtain that for any $\varepsilon>0$,
\begin{equation}\label{0es0}
\begin{split}&\frac{1}{2}\frac{d}{dt}
\left(\int_\Omega\rho |u|^2+\int_{\Sigma_+} \sigma_+|\nabla_\ast
\eta_+|^2+\rho_+ g|\eta_+|^2+\int_{\Sigma_-} \sigma_-|\nabla_\ast
\eta_-|^2   \right) +\frac{1}{2}\int_\Omega\mu |\mathbb{D}
u|^2
\\&\quad=\int_\Omega
 f\cdot
u-\int_\Sigma   {g} \cdot u +\int_{\Sigma_-}\rj  g \eta_- u_3
\\&\quad\le C_\varepsilon (\|f\|_{-1}^2
+\|g\|_{-1/2}^2+\|\eta\|_{0}^2)+\varepsilon\|u\|_1^2
.\end{split}
\end{equation}
Now we integrate \eqref{0es0} in time from $0$ to $t$ and bound the resulting terms using the Korn inequality and the Poincar\'e inequality on $\Sigma_-$, as well as the nonlinear estimates \eqref{nonl2}; taking $\varepsilon$ sufficiently small in the resulting bound, we may absorb the $\|u\|_1^2$ term onto the left to find that
\begin{equation}\label{0es1}
\|u(t)\|_0^2+\|\eta(t)\|_1^2+\int_0^t\|u(s)\|_1^2\,ds\le C
\mathcal{E}(0)+ C\int_0^t\mathcal{E}(s)
\mathcal{D}(s)\,ds+C\int_0^t\|\eta(s)\|_{0}^2\,ds.
\end{equation}

Now we apply  $\partial_t^\ell$ with $\ell=1,2$ to $\eqref{surface}_1$ and integrate and estimate as above to find that
\begin{equation}\label{0es01}
\begin{split} &\frac{1}{2}\frac{d}{dt}
\left(\int_\Omega\rho |\partial_t^\ell u|^2+\int_{\Sigma_+}
\sigma_+|\nabla_\ast \partial_t^\ell \eta_+|^2+\rho_+
g|\partial_t^\ell \eta_+|^2+\int_{\Sigma_-} \sigma_-|\nabla_\ast
\partial_t^\ell \eta_-|^2   \right) +\frac{1}{2}\int_\Omega\mu |\mathbb{D}
\partial_t^\ell u|^2
\\&\quad=\int_\Omega
 \partial_t^\ell f\cdot
\partial_t^\ell u-\int_\Sigma   {\partial_t^\ell g} \cdot \partial_t^\ell u +\int_{\Sigma_-}\rj  g \partial_t^\ell \eta_- \partial_t^\ell u_3
\\&\quad \le C (\|\partial_t^\ell f\|_{-1}
+\|\partial_t^\ell g\|_{-1/2}+\|\partial_t^{\ell}
\eta\|_{-1/2})\|\partial_t^\ell u\|_1
\\&\quad \le
C_\varepsilon  (\|\partial_t^\ell f\|_{-1}^2 +\|\partial_t^\ell
g\|_{-1/2}^2+\|\partial_t^{\ell-1}
u\|_{1}^2)+\varepsilon\|\partial_t^\ell u\|_1^2
.\end{split}
\end{equation}
Here in the last inequality we have used the fact that
\begin{equation}
\|\partial_t^{\ell}
\eta\|_{-1/2}=\|\partial_t^{\ell-1} u_3\|_{H^{-1/2}(\Sigma)}\le
C\|\partial_t^{\ell-1} u\|_{1}.
\end{equation}
Then we integrate  \eqref{0es01} in time from $0$ to $t$ and argue as above to deduce that  for $\ell=1,2$,
\begin{equation}\label{0es2}
\|\partial_t^\ell
u(t)\|_0^2+\|\partial_t^\ell\eta(t)\|_1^2+\int_0^t\|\partial_t^\ell
u(s)\|_1^2\,ds\le C \mathcal{E}(0)+ C\int_0^t\mathcal{E}(s)
\mathcal{D}(s)\,ds+C\int_0^t\|\partial_t^{\ell-1}
u(s)\|_{0}^2\,ds.
\end{equation}

Consequently, chaining \eqref{0es1} and \eqref{0es2} together, we get \eqref{0es}.
\end{proof}

Next, we derive the energy evolution of the mixed horizontal space-time derivatives.

\begin{lemma}\label{horizontal}
Let $(u,p,\eta)$ solve \eqref{surface}, then for any $\varepsilon>0$, there exists constant $C_\varepsilon>0$ such that
\begin{equation}\label{1es}
\begin{split}
&\sum_{ |\alpha|\le 4 } \left(\|\partial^\alpha u(t)\|_{0}^2+\|\partial^\alpha\eta(t)\|_{1}^2\right)
+\int_0^t \sum_{|\alpha|\le 4 }\|\partial^\alpha u(s)\|_{1}^2\,ds
\\&\quad \le  C_\varepsilon\mathcal{E} (0)+C_\varepsilon\int_0^t
 {\mathcal{E}(s)} \mathcal{D} (s) \,ds
+\varepsilon\int_0^t \mathcal{D}(s)\,ds
 +C_\varepsilon
\int_0^t\|\eta(s)\|_{0}^2\,ds  .
\end{split}
\end{equation}
\end{lemma}

\begin{proof}
Since the boundaries $\Sigma_\pm$ and $\Sigma_b$ are flat,  we are free to take temporal and horizontal derivatives of the equations \eqref{surface} without disrupting the structure of the boundary conditions.  Applying $\partial^\alpha$ for $\alpha\in \mathbb{N}^{1+2}$ with $|\alpha|\le 4, \alpha_0<2$ to $\eqref{surface}_1$, multiplying the  resulting equations by $\partial^\alpha u$, and then integrating by parts over $\Omega$, we find that (again using trace estimates)
\begin{equation}\label{1es0}
\begin{split}
&\frac{1}{2}\frac{d}{dt}
\left(\int_\Omega\rho |\partial^\alpha u|^2+\int_{\Sigma}
\sigma|\nabla_\ast
\partial^\alpha \eta|^2 \right) +\frac{1}{2}\int_\Omega\mu
|\mathbb{D}\partial^\alpha u|^2
\\&\quad=\int_\Omega
\partial^\alpha f\cdot
\partial^\alpha u-\int_\Sigma \partial^\alpha  {g} \cdot
\partial^\alpha u+\int_{\Sigma_-}\rj    g \partial^\alpha \eta_- \partial^\alpha  u_3
\\&\quad\le   C (\|\partial^\alpha f\|_{-1}+\|
\partial^\alpha  {g}\|_{-1/2}+\|\partial^\alpha  \eta\|_{-1/2})\|\partial^\alpha u\|_1.
\end{split}
\end{equation}

Notice that for $|\alpha|\le 4, \alpha_0<2$,  the kinematic boundary condition, the trace theorem, and the usual Sobolev interpolation allow us to estimate
\begin{equation}\label{1es1}
\begin{split}
\|\partial^\alpha  \eta\|_{-1/2}^2&\le \|\eta\|_{7/2}^2+\|\partial_t\eta\|_{3/2}^2
=\|\eta\|_{7/2}^2+\|u_3\|_{H^{3/2}(\Sigma)}^2
\\&\le
C_\varepsilon (\|\eta\|_{0}^2+\|u\|_{0}^2)+\varepsilon(\|\eta\|_{4}^2+\|u\|_{3}^2)
\le C_\varepsilon (\|\eta\|_{0}^2+\|u\|_{0}^2)+\varepsilon
\mathcal{D} .
\end{split}
\end{equation}
We then integrate \eqref{1es0} in time from $0$ to $t$,  plug \eqref{1es1} into the resulting inequality, and use Korn's inequality, Poincar\'{e}'s inequality, Cauchy's inequality, and the estimate \eqref{nonl2} to find that
\begin{equation}
\begin{split}
 &  \sum_{ \substack{|\alpha|\le 4 \\ \alpha_0<2}}(\|\partial^\alpha
u(t)\|_0^2+\|
\partial^\alpha \eta(t)\|_1^2)   +\int_0^t  \sum_{ \substack{|\alpha|\le 4 \\ \alpha_0<2}} \|\partial^\alpha
u(s)\|_1^2\,ds
\\&\quad\le  C\mathcal{E} (0)+C\int_0^t
 {\mathcal{E}(s)} \mathcal{D} (s) \,ds
+\varepsilon\int_0^t \mathcal{D}(s) \,ds
 +C_\varepsilon
\int_0^t(\|\eta(s)\|_{0}^2+\|u(s)\|_{0}^2)\,ds .
\end{split}
\end{equation}
We then obtain the estimate \eqref{1es} from the above inequality and the estimate \eqref{0es} of Lemma \ref{nosur1}.
\end{proof}

\subsubsection{Full energy estimates}

Notice that the estimates \eqref{1es} of Lemma \ref{horizontal} only contain the ``horizontal'' part energy estimates. In this subsection we will improve the estimates to be the full energy estimates.

\begin{theorem}\label{energywith}
Let $(u,p,\eta)$ solve \eqref{surface}. If $\mathcal{E}(t) \le \delta^2$ for sufficiently small $\delta$ for all $t \in [0,T]$, then we have that
\begin{equation}\label{2es}
\mathcal{E}(t)+\int_0^t\mathcal{D}(s)\,ds
 \le  C\mathcal{E} (0)+C
\int_0^t\|\eta(s)\|_{0}^2 \,ds
\end{equation}
for $t \in [0,T]$.
\end{theorem}
\begin{proof}
As in \cite{WTK}, we will use the elliptic regularity theory for the stationary Stokes problems and the structure of the equations \eqref{surface} to improve the estimates.  We divide our proof into several steps.

\emph{Step 1 -- $\mathcal{E}$ estimates}

First,  we may apply  $\partial_t^\ell$ with $\ell=0,1$ to the equations \eqref{surface} to find
\begin{equation}\label{jellip1}
\left\{\begin{array}{lll}
-\mu\Delta \partial_t^\ell u+\nabla\partial_t^\ell p
=-\rho\partial_t^{\ell+1} u+\partial_t^\ell f
\quad&\hbox{in }\Omega
\\ \diverge \partial_t^\ell u=0&\hbox{in }\Omega
\\ \llbracket \partial_t^\ell p_+I-\mu_+\mathbb{D}(\partial_t^\ell u_+)\rrbracket e_3
=\rho_+g\partial_t^\ell\eta_+
e_3-\sigma_+\Delta_\ast\partial_t^\ell\eta_++\partial_t^\ell
g^3_+&\hbox{on }\Sigma_+
\\ \llbracket \partial_t^\ell u\rrbracket=0,
\ \llbracket\partial_t^\ell pI-\mu\mathbb{D}(\partial_t^\ell u)\rrbracket
e_3=\rj g\partial_t^\ell\eta_- e_3+\sigma_-\Delta_\ast\partial_t^\ell\eta_--\partial_t^\ell g_-^3&\hbox{on
}\Sigma_-
\\ \partial_t^j u_-=0 &\hbox{on }\Sigma_b.
\end{array}\right.
\end{equation}
Applying the two-phase Stokes regularity theory in Lemma \ref{cStheorem} with $r=4-2\ell\ge 2$ to the problem \eqref{jellip1} and using \eqref{1es} and \eqref{nonl1}, we obtain
\begin{equation}\label{l0}
\begin{split}  &\|\partial_t^\ell u(t)\|_{4-2\ell}^2+ \|\partial_t^{\ell} p(t)\|_{4-2\ell-1}^2
  \\&\quad\lesssim \|\partial_t^{\ell+1} u\|_{4-2\ell-2 }^2+\| \partial_t^\ell f(t)\|_{4-2\ell-2
}^2+\|\partial_t^\ell \eta\|_{4-2\ell-3/2+2  }^2+\|
 \partial_t^\ell{g}\|_{4-2\ell-3/2 }^2
   \\&\quad  \lesssim \mathcal{E}^2+\|\partial_t^{\ell+1} u\|_{4-2\ell-2 }^2+\|\partial_t^\ell \eta\|_{4-2\ell-3/2+2  }^2.
   \end{split}
\end{equation}
Chaining \eqref{l0} and \eqref{1es}, we obtain
\begin{equation}\label{l1}
\sum_{\ell=0}^2\left(\|\partial_t^\ell u\|_{4-2\ell}^2+\|\partial_t^\ell \eta\|_{4-2\ell+1}^2\right)+\sum_{\ell=0}^1\|\partial_t^\ell p\|_{4-2\ell-1}^2\lesssim \mathcal{E}^2+\mathcal{Z},
\end{equation}
where we have written compactly $\mathcal{Z}$ to denote the right hand side of \eqref{1es}.

Next, employing the kinematic boundary conditions and the trace theorem, we deduce from \eqref{l1} that
\begin{equation}\label{l3}
\begin{split}
\sum_{\ell=1}^3\|\partial_t^\ell\eta(t)\|_{   4-2\ell+3/2  }^2
 &=\sum_{\ell=1}^3\|\partial_t^{\ell-1} u_3(t)\|_{  H^{ 4-2\ell+3/2}(\Sigma)  }^2
 \\&\lesssim \sum_{\ell=1}^3\|\partial_t^{\ell-1} u(t)\|_{    4-2(\ell-1)  }^2\lesssim\mathcal{E}^2+\mathcal{Z}.
    \end{split}
\end{equation}
Here, when $\ell=3$ we have used the fact $\|\partial_t^2 u_3\|_{  H^{ -1/2}(\Sigma)  }\lesssim \|\partial_t^2 u\|_{  0 }$ since $\diverge \partial_t^2 u=0$ in $\Omega$.

We complete the energy part of the desired estimate by summing  \eqref{l1}--\eqref{l3} and then absorbing the $\mathcal{E}^2$ term (taking $\delta$ small enough) onto the left side of the resulting inequality.  This yields the estimate
\begin{equation}\label{hh1}
\mathcal{E}(t)  \lesssim \mathcal{Z}.
\end{equation}

\emph{Step 2 -- $\mathcal{D}$ estimates for $u$ and $\nabla p$}

Now we turn to the dissipation part estimates. Notice that we have not yet derived at estimate of $\eta$ in terms of the dissipation, so we can not apply the two-phase elliptic estimates of Lemma \ref{cStheorem} as above. It is crucial to observe that from \eqref{1es} we can get higher regularity estimates of u on the boundaries $\Sigma=\Sigma_+\cup\Sigma_-$. Indeed, since $\Sigma_\pm$ are flat, by the definition of Sobolev norm on $\mathrm{T}^2$  and the trace theorem, we may deduce from \eqref{1es} that for $\ell=0,1,2$,
\begin{equation}\label{boundary}
\begin{split}
\int_0^t  \|\partial_t^\ell u(s)\|_{H^{4-2\ell+1/2}(\Sigma)}^2\,ds
&\lesssim \int_0^t  \left(\| \partial_t^\ell u(s)\|_{H^{
1/2}(\Sigma)}^2+\|\nabla_\ast^{(4-2\ell)}\partial_t^\ell
u(s)\|_{H^{1/2}(\Sigma)}^2\right)\,ds
\\ & \lesssim   \int_0^t  \left(\| \partial_t^\ell u(s)\|_{1}^2+\|\nabla_\ast^{(4-2\ell)}\partial_t^\ell
u(s)\|_{1}^2\right)\,ds\lesssim
\mathcal{Z}.
    \end{split}
\end{equation}
This motivates us to use the one-phase  Stokes regularity theory of Lemma \ref{cS1phaselemma2}.

The key point is to observe that $(\partial_t^\ell u_+,\partial_t^\ell p_+) $ with $\ell=0,1$ solve the  problem
 \begin{equation} \label{pro+11}
 \left\{\begin{array}{lll}-\mu_+\Delta \partial_t^\ell u_++\nabla\partial_t^\ell p_+
 =-\rho_+\partial_t^{\ell+1} u_++\partial_t^\ell f_+\quad&\hbox{in }&\Omega_+
\\ \diverge \partial_t^\ell u_+=0&\hbox{in }&\Omega_+
\\ \partial_t^\ell u_+=\partial_t^\ell u_+&\hbox{on }&\Sigma,
\end{array}\right.
\end{equation}
and  $(\partial_t^\ell u_-,\partial_t^\ell p_-) $ solve the  problem
 \begin{equation} \label{pro-11}
 \left\{\begin{array}{lll}-\mu_-\Delta \partial_t^\ell u_-+\nabla\partial_t^\ell p_-
 =-\rho_-\partial_t^{\ell+1} u_-+\partial_t^\ell f_-\quad&\hbox{in }&\Omega_-
\\ \diverge \partial_t^\ell u_-=0&\hbox{in }&\Omega_-
\\ \partial_t^\ell u_-=\partial_t^\ell u_-&\hbox{on }&\Sigma_-,
\\ \partial_t^\ell u_-=0 &\hbox{on }&\Sigma_b.
\end{array}\right.
\end{equation}
We apply Lemma \ref{cS1phaselemma2} with $r=4-2\ell+1$ to the problems \eqref{pro+11} and \eqref{pro-11} respectively, using the estimates \eqref{nonl2}, \eqref{boundary}, and then summing up, to have
\begin{equation}\label{ll0}
\begin{split}
&\int_0^t
 \|\partial_t^\ell u(s)\|_{4-2\ell+1}^2+\|\nabla\partial_t^\ell
p(s)\|_{4-2\ell-1}^2 \,ds
\\&\quad\lesssim
\int_0^t\|\partial_t^{\ell+1} u(s)\|_{4-2\ell-1 }^2+\|
\partial_t^\ell f(s)\|_{4-2\ell-1 }^2+\|\partial_t^\ell u(s)\|_{H^{4-2\ell+1/2}(\Sigma)}^2\,ds
\\&\quad\lesssim \int_0^t\|\partial_t^{\ell+1} u(s)\|_{4-2\ell-1
}^2\,ds +\int_0^t \mathcal{E}(s)\mathcal{D}(s)\,ds+ \mathcal{Z}
\\&\quad\lesssim \int_0^t\|\partial_t^{\ell+1} u(s)\|_{4-2\ell-1
}^2\,ds +  \mathcal{Z}.
\end{split}
\end{equation}
Chaining \eqref{ll0} and \eqref{1es}, we obtain
\begin{equation}\label{lll2}
   \int_0^t\sum_{\ell=0}^2
 \|\partial_t^\ell u(s)\|_{4-2\ell+1}^2+\sum_{\ell=0}^1 \|\nabla\partial_t^\ell
p(s)\|_{4-2\ell-1}^2  \,ds \lesssim \mathcal{Z}.
\end{equation}

\emph{Step 3 -- $\mathcal{D}$ estimates for $\eta$ and $p$}

Now that we have obtained \eqref{lll2}, we estimate the remaining parts in  $\mathcal{D}$. We will turn to the boundary conditions in \eqref{surface}. First we derive estimates for $\eta$. For the time derivatives of $\eta$, we use the kinematic boundary conditions, the trace theorem and \eqref{lll2} to obtain
\begin{equation}\label{toesd52}
\begin{split}
\int_0^t\sum_{\ell=1}^3
  \|\partial_t^\ell\eta(s)\|_{ 4-2\ell+5/2  }^2\,ds
  &=\int_0^t\sum_{\ell=1}^3
  \|\partial_t^{\ell-1}u_3(s)\|_{ H^{4-2\ell+5/2}(\Sigma)  }^2\,ds
\\&\lesssim
\int_0^t\sum_{\ell=1}^3\|\partial_t^{\ell-1}u_3(s)\|_{ 4-2(\ell-1)+1
}^2\,ds\lesssim\mathcal{Z}.
\end{split}
\end{equation}
For the term $\eta$ without temporal derivatives we use the dynamic boundary conditions
\begin{equation} \label{ellip1}
 -\sigma_+\Delta_\ast\eta_+  = p_+-2\mu_+\partial_3u_{3,+}- {g}_+-\rho_+g\eta_+\hbox{ on }\Sigma_+
 \end{equation}
and
\begin{equation}\label{ellip2}
 -\sigma_-\Delta_\ast\eta_-  = \llbracket -p +2\mu\partial_3u_3\rrbracket
- {g}_-+\rj g\eta_-\hbox{ on }\Sigma_-.
\end{equation}
Notice that at this point we do not have any bound on $p$ on the
boundary $\Sigma$ yet, but we have the estimate of $\nabla p$ in
$\Omega$. Applying $\nabla_\ast$ to  \eqref{ellip1}--\eqref{ellip2} and employing the standard  elliptic theory on $\mathrm{T}^2$ as well as  the trace theorem and \eqref{lll2}, we find that
\begin{equation}\label{toesd62}
\begin{split}\int_0^t
\|\nabla_\ast\eta(s)\|_{   9/2  }^2\,ds
       & \lesssim  \int_0^t\| \nabla_\ast p(s)\|_{H^{5/2}(\Sigma)}^2+\| \nabla_\ast\partial_3u_3(s)\|_{H^{5/2}(\Sigma)}^2
       +\| \nabla_\ast\eta(s)\|_{ 5/2}^2\,ds
      \\& \lesssim  \int_0^t\| \nabla_\ast p(s)\|_{3}^2+\|  u (s)\|_{5}^2
       +\| \eta(s)\|_{ 7/2}^2\,ds
       \\& \lesssim\mathcal{Z} + \int_0^t \| \eta(s)\|_{ 7/2}^2\,ds .
       \end{split}
\end{equation}
Using Poincar\'e's inequality on $\Sigma_\pm$ (since $\int_{\mathrm{T}^2}\eta=0$) and the same type of Sobolev interpolation trick we used in \eqref{1es1}, we deduce from \eqref{toesd62} that
\begin{equation}\label{toesd72}
\int_0^t
\| \eta(s)\|_{ 11/2  }^2\,ds
\lesssim\mathcal{Z} + \int_0^t \|
\eta(s)\|_{ 0}^2\,ds\lesssim\mathcal{Z}  .
\end{equation}

We now estimate  $\|\partial_t^\ell p\|_0$ for $\ell=0,1$. By the equations \eqref{ellip1}--\eqref{ellip2},  the estimates \eqref{lll2}, \eqref{toesd72} and \eqref{nonl2}, and the trace theorem, we find that for $\ell=0,1$,
\begin{equation}\label{pp11}
\begin{split}
    &\int_0^t
 \|\partial_t^\ell p_+(s)\|_{L^2(\Sigma_+)}^2+  \| \llbracket  \partial_t^\ell p\rrbracket(s)\|_{L^2(\Sigma_+)}^2  \,ds
 \\&\quad\lesssim \int_0^t
 \|\partial_t^\ell \partial_3 u_3(s)\|_{L^2(\Sigma)}^2+\|\partial_t^\ell g(s)\|_{L^2(\Sigma)}^2+\|\partial_t^\ell\eta(s)\|_{2}^2  \,ds
 \\&\quad\lesssim \mathcal{Z}+\int_0^t
\mathcal{E}(s)\mathcal{D}(s) \,ds\lesssim
\mathcal{Z}.
  \end{split}
\end{equation}
 By Poincar\'e's inequality on $\Omega_+$ (Lemma \ref{poincare}), \eqref{lll2} and \eqref{pp11}, we have
\begin{equation}\label{pp21}
\begin{split}
 \int_0^t   \| \partial_t^\ell
p_+(s)\|_{H^1(\Omega_+)}^2  \,ds &=\int_0^t   \| \partial_t^\ell
p_+(s)\|_{L^2(\Omega_+)}^2+\|\nabla
\partial_t^\ell p_+(s)\|_{L^2(\Omega_+)}^2 \,ds
\\&\lesssim\int_0^t   \| \partial_t^\ell
p_+(s)\|_{L^2(\Sigma_+)}^2+\|\nabla
\partial_t^\ell p_+(s)\|_{L^2(\Omega_+)}^2 \,ds \lesssim\mathcal{Z}.
  \end{split}
\end{equation}
On the other hand, by the trace theorem and \eqref{pp11}--\eqref{pp21}, we have
\begin{equation}\label{pp31}
\begin{split}
 \int_0^t   \| \partial_t^\ell
p_-(s)\|_{L^2(\Sigma_-)}^2  \,ds&\le \int_0^t   \| \partial_t^\ell
p_+(s)\|_{L^2(\Sigma_-)}^2+\| \llbracket \partial_t^\ell
p(s)\rrbracket\|_{L^2(\Sigma_-)}^2  \,ds
\\& \lesssim
\int_0^t   \| \partial_t^\ell p_+(s)\|_{H^1(\Omega_+)}^2+\|
\llbracket
\partial_t^\ell p(s)\rrbracket\|_{L^2(\Sigma_-)}^2  \,ds
 \lesssim\mathcal{Z},
 \end{split}
\end{equation}
and so again by Poincar\'e's inequality on $\Omega_-$ as well as \eqref{lll2} and \eqref{pp31}, we have
\begin{equation}\label{pp41}
\begin{split}
\int_0^t   \| \partial_t^\ell
p_-(s)\|_{H^1(\Omega_-)}^2  \,ds &= \int_0^t   \| \partial_t^\ell
p_-(s)\|_{L^2(\Omega_-)}^2+ \| \nabla\partial_t^\ell
p_-(s)\|_{L^2(\Omega_-)}^2 \,ds
\\&\lesssim\int_0^t   \|
\partial_t^\ell p_-(s)\|_{L^2(\Sigma_-)}^2+ \| \nabla\partial_t^\ell
p_-(s)\|_{L^2(\Omega_-)}^2 \,ds \lesssim\mathcal{Z}.
 \end{split}
\end{equation}
 In light of \eqref{pp21} and \eqref{pp41}, we may improve the estimate \eqref{lll2} to be
\begin{equation}\label{lll3}
\begin{split}
    \int_0^t\sum_{\ell=0}^2
 \|\partial_t^\ell u(s)\|_{4-2\ell+1}^2+\sum_{\ell=0}^1 \| \partial_t^\ell
p(s)\|_{4-2\ell}^2  \,ds \lesssim \mathcal{Z}.
 \end{split}
\end{equation}

\emph{Step 4 -- Remaining $\mathcal{D}$ estimates for $u$ and $p$}

Next, we want to estimate $\|\rho\partial_t^3 u(s)\|_{\Hd}^2$ and $\|\partial_t^2p\|_0^2$. To do this, we shall use a weak formulation of the problem \eqref{surface}. It follows that
\begin{equation}\label{weak}
\begin{split}
\langle\rho\partial_t^3 u,\varphi \rangle_{\ast}+(\frac{\mu}{2}\mathbb{D}\partial_t^2u, \mathbb{D}\varphi)
&=\langle\partial_t^2 f,\varphi \rangle_{\ast}-\langle
\partial_t^2g,\varphi \rangle_{-1/2}+\sigma_\pm(\Delta_\ast \partial_t^2\eta_\pm,  \varphi_{3,+})_{\pm}
\\&\quad -\rho_+g(\partial_t^2\eta_{+},\varphi_{3,+} )_{+}
+\rj g(\partial_t^2\eta_{-},\varphi_3)_{-} \hbox{ for all
}\varphi\in {}_0H_\sigma^1(\Omega).
 \end{split}
\end{equation}
Here $\langle\cdot,\cdot\rangle_\ast$ denotes the dual pairing between $({}_0H_\sigma^1(\Omega))^\ast$ and ${}_0H_\sigma^1(\Omega)$, $\langle \cdot,\cdot \rangle_{-1/2}$ denotes the dual pairing between $H^{1/2}(\Sigma)$ and $H^{-1/2}(\Sigma)$, $(\cdot,\cdot)_{\pm}$ is the $L^2$ inner product on $\Sigma_\pm$. By the estimates \eqref{nonl2}, \eqref{toesd52} and \eqref{lll3}, we see from the formulation \eqref{weak} that
\begin{equation}\label{000}
\int_0^t\|\rho\partial_t^3u(s)\|_{({}_0H_\sigma^1(\Omega))^\ast}^2\,ds\le
\mathcal{Z}.
\end{equation}
Since ${}_0H_\sigma^1(\Omega)\subset  {}_0H^1(\Omega)$, the usual theory of Hilbert spaces provides a unique operator $E: ({}_0H_\sigma^1(\Omega))^\ast\to ({}_0H^1(\Omega))^\ast$ satisfying the property that $Ef\vert_{ {}_0H_\sigma^1(\Omega)} = f$ and that $\|Ef\|_{({}_0H^1(\Omega))^\ast}= \|f\|_{({}_0H_\sigma^1(\Omega))^\ast}$ for all $f \in ({}_0H_\sigma^1(\Omega))^\ast$. Using this $E$, we regard the element of $  ({}_0H_\sigma^1(\Omega))^\ast$ as an element of $({}_0H^1(\Omega))^\ast$ in a natural way. Hence, we may also rewrite \eqref{000} as
\begin{equation}\label{111}
\int_0^t\|\rho\partial_t^3u(s)\|_{({}_0H^1(\Omega))^\ast}^2\,ds\le
\mathcal{Z}.
\end{equation}
With the estimate \eqref{111} in mind, since the pressure can be viewed as a Lagrange multiplier to \eqref{weak}, applying Lemma \ref{Pressure}, we obtain
\begin{equation}\label{112}
\int_0^t\|\partial_t^2p(s)\|_{0}^2\,ds\le
\mathcal{Z}.
\end{equation}

Now we derive the boundary estimates for $\partial_t^2 u$ and $\partial_t^2p$. First, by \eqref{lll3}, we have
 \begin{equation}\label{011}
         \int_0^t\|\nabla_\ast\partial_t^2 u(s)\|_{ H^{-1/2}(\Sigma) }^2\,ds
         \lesssim \int_0^t\| \partial_t^2 u(s)\|_{ H^{1/2}(\Sigma) }^2\,ds\lesssim\int_0^t\| \partial_t^2 u(s)\|_{  1 }^2\,ds
        \lesssim\mathcal{Z},
\end{equation}
and then by the incompressibility condition and \eqref{011}, we get
 \begin{equation}\label{012}
        \int_0^t\|\partial_3\partial_t^2 u_3(s)\|_{ H^{-1/2}(\Sigma) }^2\,ds
        =\int_0^t\|\partial_1\partial_t^2 u_1(s)+\partial_2\partial_t^2 u_2(s)\|_{ H^{-1/2}(\Sigma)
        }^2\,ds
        \lesssim\mathcal{Z}.
\end{equation}
For the term $\llbracket \mu \partial_3 \partial_t^2  u_i\rrbracket,\ i=1,2$, we use the first two identities of the dynamic boundary conditions on  $\Sigma_-$, along with the estimates \eqref{nonl2} and \eqref{011}, to see that
\begin{equation}\label{013}
\begin{split}
\int_0^t\|\llbracket \mu \partial_3 \partial_t^2
 u_i(s)\rrbracket\|_{ H^{-1/2}(\Sigma_-) }^2\,ds&=
 \int_0^t\|\partial_t^2g_-^i(s)\|_{ H^{-1/2}(\Sigma_-) }^2+\|\llbracket \mu \partial_i \partial_t^2
 u_3(s)\rrbracket(s)\|_{ H^{-1/2}(\Sigma_-) }^2\,ds
        \\&\lesssim \mathcal{Z}+\int_0^t
\mathcal{E}(s)\mathcal{D}(s) \,ds \lesssim\mathcal{Z},\
i=1,2.
 \end{split}
\end{equation}
Similarly, we have
 \begin{equation}\label{014}
\int_0^t\| \partial_3 \partial_t^2 u_{i,+}(s) \|_{ H^{-1/2}(\Sigma_+) }^2\,ds \lesssim \mathcal{Z}.
 \end{equation}
Next, using the third identity of the dynamic boundary condition on $\Sigma$, along with the estimates \eqref{nonl2} and \eqref{012}, we obtain the bound
\begin{equation}\label{015}
\begin{split}
        & \int_0^t\|\partial_t^2 p_+(s)\|_{ H^{-1/2}(\Sigma_+) }^2+\|\llbracket  \partial_t^2 p(s)\rrbracket \|_{ H^{-1/2}(\Sigma_-)
        }^2\,ds
  \\&\quad\le\int_0^t\|\partial_t^2 g^3(s)\|_{ {-1/2}  }^2+\| \partial_3\partial_t^2 u_3(s)\|_{ H^{-1/2}(\Sigma)
 }^2+\|  \partial_t^2 \eta(s)\|_{ 3/2}^2\,ds
        \\&\quad\lesssim \mathcal{Z}+\int_0^t
\mathcal{E}(s)\mathcal{D}(s) \,ds
\lesssim\mathcal{Z}.
 \end{split}
\end{equation}

We now combine \eqref{toesd52}, \eqref{toesd72}, \eqref{lll3}, \eqref{111}--\eqref{112} and \eqref{011}--\eqref{015}, and then use the smallness of $\delta$ to absorb the $\int \mathcal{E} \mathcal{D}$ term onto the left.  This   yields the inequality
\begin{equation}\label{hh2}
\int_0^t\mathcal{D}(s)\,ds
 \lesssim \mathcal{Z},
\end{equation}
which completes the dissipation part of the estimates.

\emph{Step 5 -- Conclusion}

Consequently, by \eqref{hh1}, \eqref{hh2} and the definition of  $\mathcal{Z}$, we have
\begin{equation}\label{hh3}
\mathcal{E}(t)+\int_0^t\mathcal{D}(s)\,ds
 \le  C_\varepsilon\mathcal{E} (0)+ C_\varepsilon\int_0^t
 {\mathcal{E}(s)} \mathcal{D} (s) \,ds
 +\varepsilon C \int_0^t \mathcal{D}(s) \,ds+C_\varepsilon
\int_0^t\|\eta(s)\|_{0}^2 \,ds.
\end{equation}
By taking $\varepsilon$ and $\delta$ sufficiently small, the previous inequality implies that
\begin{equation}\label{hh333333}
\begin{split}\mathcal{E}(t)+\int_0^t\mathcal{D}(s)\,ds
 \le  C\mathcal{E} (0)+C
\int_0^t\|\eta(s)\|_{0}^2 \,ds .
 \end{split}
\end{equation}
This is \eqref{2es} and we conclude our proof.
\end{proof}

 \subsection{Energy estimates without surface tension}

In this subsection, we will derive the nonlinear energy estimates for the system \eqref{nosurface2}. For this, the definitions of the energy and dissipation rely on the linear energy identity of the homogeneous form of \eqref{nosurface2}:
\begin{equation}\label{lin_en_ev}
\frac{d}{dt}\left(\frac{1}{2}\int_\Omega \rho |u|^2 +\frac{1}{2}\int_{\Sigma_+}
\rho_+g|\eta_+|^2+\frac{1}{2}\int_{\Sigma_-} -\rj g|\eta_-|^2\right)+\frac{1}{2}\int_{\Omega}\mu|\mathbb{D}
u|^2=0.
\end{equation}
According to this energy identity and the structure of the equations, as in \cite{GT_per,WTK}, we define the energies and dissipations as follows. For any integer $N \ge 3$ we define the  energy as
\begin{equation}\label{E_2N}
 \mathcal{E}_{2N}: =
 \sum_{j=0}^{2N} \left( \|\partial_t^j u\|_{4N-2j}^2 + \|\partial_t^j \eta\|_{4N-2j}^2 \right)
 + \sum_{j=0}^{2N-1} \|\partial_t^j p\|_{4N-2j-1}^2
\end{equation}
and the corresponding dissipation as
\begin{equation}
\begin{split}
  \mathcal{D}_{2N} &:= \sum_{j=0}^{2N} \|\partial_t^j u\|_{4N-2j+1}^2 + \sum_{j=0}^{2N-1} \|\partial_t^j p\|_{4N-2j}^2
  \\
&\quad+  \| \eta\|_{4N-1/2}^2 + \|\partial_t \eta\|_{4N-1/2}^2 +
\sum_{j=2}^{2N+1} \|\partial_t^j \eta\|_{4N-2j+5/2}^2.
 \end{split}
\end{equation}
 We also  define
\begin{equation}\label{F_2N}
\mathcal{F}_{2N}:=   \|\eta\|_{4N+1/2}^2.
\end{equation}

We will now derive our a priori estimates.  We assume throughout this subsection that $\mathcal{E}_{2N}(t) + \mathcal{F}_{2N}(t) \le \delta^2$ for some sufficiently small $\delta>0$ and for all $t$ in  the interval in which the solution is defined.  We will implicitly allow $\delta$ to be made smaller in each result, but we will reiterate the smallness of $\delta$ in our main result.

 \subsubsection{Energy evolution in geometric form}

We can not use the equations of perturbed form \eqref{nosurface2} to derive the energy evolution of the pure temporal derivatives because of lack of control of the highest temporal derivative of $p$.  Instead, to control the temporal derivatives we employ  the following equivalent geometric formulation. The geometric form is the equations directly transformed by composition from the original problem \eqref{form1}:
\begin{equation}\label{nosurface}
\left\{\begin{array}{lll}
\rho\partial_t
u-\rho W\partial_3u + \rho u\cdot
\nabla_{\mathcal{A}}u-\mu\Delta_{\mathcal{A}}u+\nabla_{\mathcal{A}}p=0 &\hbox{in
} \Omega
\\ \diverge_{\mathcal{A}}u=0&\hbox{in } \Omega
\\ \partial_t\eta=u\cdot \mathcal{N}&\hbox{on } \Sigma
\\  S_{\mathcal{A}}(p_+,u_+) \mathcal{ N}_+=\rho_+g\eta_+ \mathcal{N}_+&\hbox{on } \Sigma_+
\\ \Lbrack u\Rbrack=0,\quad\Lbrack S_{\mathcal{A}}(p,u)\Rbrack \mathcal{N}_-=\rj g\eta_- \mathcal{N}_-&\hbox{on } \Sigma_-
\\ u=0 &\hbox{on } \Sigma_b
\\(u,\eta )\mid_{t=0}=(u_0,\eta_0).
\end{array}\right.
\end{equation}
Here we have written the differential operators $\nabla_\mathcal{A}, \diverge_\mathcal{A}$  and $\Delta_\mathcal{A}$ with their actions given by $(\nabla_\mathcal{A}f)_i:= \mathcal{A}_{ij}\partial_jf$, $\diverge_\mathcal{A}X := \mathcal{A}_{ij}\partial_jX_i$, and $\Delta_\mathcal{A}f= \diverge_\mathcal{A}\nabla_\mathcal{A}f$ for appropriate $f$ and $X$.   We write $S_\mathcal{A}(p, u):=pI-\mu \mathbb{D}_\mathcal{A}u$ for the stress tensor, where $(\mathbb{D}_\mathcal{A}u)_{ij} = \mathcal{A}_{ik}\partial_ku^j+\mathcal{A}_{jk}\partial_ku^i$ is the symmetric $\mathcal{A}$--gradient. Note that if we extend $\diverge_\mathcal{A}$ to act on symmetric tensors in the natural way, then $\diverge_\mathcal{A}S_\mathcal{A}(p, u)=-\mu\Delta_{\mathcal{A}}u+\nabla_{\mathcal{A}}p$ for those $u$ satisfying $\diverge_\mathcal{A} u =0$.

Applying the temporal differential operator $\partial_t^\ell $ to \eqref{nosurface}, the resulting equations are the following system for $( \partial_t^\ell  u, \partial_t^\ell p,\partial_t^\ell  \eta)$,
\begin{equation}\label{geoform}
\left\{\begin{array}{lll}
\rho\partial_t (\partial_t^\ell  u) -\rho W\partial_3(\partial_t^\ell
u) + \rho u\cdot \nabla_{\mathcal{A}}(\partial_t^\ell  u) +{\rm
div}_{\mathcal{A}}S_{\mathcal{A}}(\partial_t^\ell p,\partial_t^\ell
u)=F^{\ell,1} &\hbox{in } \Omega
\\ \diverge_{\mathcal{A}}(\partial_t^\ell  u)=F^{\ell,2}&\hbox{in } \Omega
\\ \partial_t(\partial_t^\ell   \eta)=\partial_t^\ell  u\cdot \mathcal{N}+F^{\ell,4}&\hbox{on } \Sigma
\\  S_{\mathcal{A}}(\partial_t^\ell
p_+,\partial_t^\ell  u_+) \mathcal{ N}_+=\rho_+g\partial_t^\ell
\eta_+ \mathcal{N}_++F^{\ell,3}_+&\hbox{on } \Sigma_+
\\\llbracket \partial_t^\ell  u\rrbracket=0,\quad\llbracket S_{\mathcal{A}}(\partial_t^\ell
p,\partial_t^\ell  u)\rrbracket \mathcal{N}_-=\rj g\partial_t^\ell
\eta_- \mathcal{N}_--F^{\ell,3}_-&\hbox{on } \Sigma_-
\\ \partial_t^\ell  u_-=0 &\hbox{on } \Sigma_b,
\end{array}\right.
\end{equation}
where
\begin{equation}\label{formF1}
\begin{split}
&F^{\ell,1}_i=\sum_{0<m<\ell }C_\ell^m \rho\partial_t^m W\partial_t^{\ell-m }\partial_3u_i
 +
\sum_{0<m\le \ell }C_\ell^m\partial_t^{\ell-m }\partial_t
\Theta^3\partial_t^mK\partial_3u_i
\\&\qquad\ -\sum_{0<m\le
\ell}C_\ell^m(\rho\partial_t^m(u_j\mathcal{A}_{jk})\partial_t^{\ell-m
}\partial_ku_i+\partial_t^m\mathcal{A}_{ik}\partial_t^{\ell-m
}\partial_k p)
\\&\qquad\ +\sum_{0<m\le \ell }C_\ell^m \mu\partial_t^m\mathcal{A}_{jk}\partial_t^{\ell-m }
\partial_k(\mathcal{A}_{is}\partial_su_j+\mathcal{A}_{js}\partial_su_i)
\\&\qquad\ +\sum_{0<m<\ell }C_\ell^m\mathcal{A}_{jk}\partial_k(\partial_t^{m}\mathcal{A}_{is}\partial_t^{\ell-m }\partial_su_j
+\partial_t^{m}\mathcal{A}_{js}\partial_t^{\ell-m }\partial_su_i)
\\&\qquad\
 +\partial_t^\ell \partial_t \Theta^3 K\partial_3u_i+\mathcal{A}_{jk}\partial_k(\partial_t^{\ell}
 \mathcal{A}_{is} \partial_su_j+\partial_t^{l}\mathcal{A}_{js} \partial_su_i),
  \end{split}
\end{equation}
\begin{equation}\label{formF2}
\begin{split}
&F^{\ell,2}= -\sum_{0<m<\ell}C_\ell^m\partial_t^m \mathcal{A}_{ij}\partial_t^{\ell-m }
 \partial_ju_i-\partial_t^\ell \mathcal{A}_{ij} \partial_ju_i ,
 \end{split}
\end{equation}
 \begin{equation}\label{formF31}
\begin{split}
&F^{\ell,3}_+= \sum_{0<m\le \ell }C_\ell^m\partial_t^m \nabla\eta_+(\partial_t^{\ell-m }\eta_+-\partial_t^{\ell-m }  p_+)
\\&\qquad\ -\sum_{0<m\le \ell }C_\ell^m \mu
_+\left(\partial_t^m(\mathcal{N}_{j,+} \mathcal{A}_{is,+})
\partial_t^{\ell-m }\partial_su_{j,+}+\partial_t^m(\mathcal{N}_{j,+}
\mathcal{A}_{js,+})  \partial_t^{\ell-m }\partial_su_{i,+}\right),
\end{split}
\end{equation}
 \begin{equation}\label{formF32}
\begin{split}&F^{{\ell,3}}_-=-\sum_{0<m\le \ell }C_\ell^m\partial_t^m \nabla\eta_-(\partial_t^{\ell-m }\eta_--\partial_t^{\ell-m }\llbracket p\rrbracket)
\\&\qquad\ -\sum_{0<m\le
l}C_\ell^m\left(\partial_t^m(\mathcal{N}_{j,-}
\mathcal{A}_{is})\llbracket \mu
\partial_t^{\ell-m }\partial_su_j\rrbracket+\partial_t^m(\mathcal{N}_{j,-}
\mathcal{A}_{js})\llbracket \mu
\partial_t^{\ell-m }\partial_su_i\rrbracket\right),
\end{split}
\end{equation}
 \begin{equation}\label{formF34}
\begin{split}
&F^{\ell,4}=\sum_{0<m\le \ell }C_\ell^m\partial_t^m \nabla\eta\cdot \partial_t^{\ell-m } u.
\end{split}
\end{equation}

We present the estimates of these nonlinear perturbation terms $F^{\ell,1},F^{\ell,2},F^{\ell,3}$ and $F^{\ell,4}$ in the following lemma.

\begin{lemma}
Let the $F^{\ell,i},\ i=1,2,3,4$ be given as above, then for $0\le \ell\le 2N$, we have
\begin{equation} \label{Fes1}
\|F^{\ell,1}\|_0^2+\|\partial_t(JF^{\ell,2})\|_0^2
+\|F^{\ell,3}\|_0^2+\|F^{\ell,4}\|_0^2\lesssim \mathcal{E}_{2N}
\mathcal{D}_{2N}.
\end{equation}
\end{lemma}

\begin{proof}
Our perturbations $F^{\ell,1},$ $F^{\ell,2},$ $F^{\ell,3},$ and $F^{\ell,4}$ have the same structure as those of \cite{GT_per}. As such, the estimates in \eqref{Fes1} are recorded in Theorem 4.1 of \cite{GT_per}.
\end{proof}

We now  estimate the evolution of the pure temporal derivatives.

\begin{lemma}\label{temp}
It holds that
 \begin{equation}
 \begin{split}\label{geoes}
 &\sum_{\ell=0}^{2N}\|\partial_t^\ell
u(t)\|_0^2 + \int_0^t\sum_{\ell=0}^{2N}\| \partial_t^\ell
u(s)\|_1^2\,ds
\\&\quad\lesssim \mathcal{E}
_{2N}(0)+({\mathcal{E}_{2N}}(t))^{3/2}+\int_0^t\sqrt{{\mathcal{E}
_{2N}(s)}} \mathcal{D}_{2N}(s)\,ds  +  \int_0^t\|\eta(s)\|_0^2\,ds
.
\end{split}
\end{equation}
\end{lemma}

\begin{proof}
For $\ell=0,1,\dots,2N$, we take the dot product of $\eqref{geoform}_1$ with $J \partial_t^\ell  u$, integrate by parts over the domain $\Omega$ and then integrate in time from $0$ to $t$; by using the other conditions in \eqref{geoform} and some easy geometric identities, we obtain the following energy identity:
\begin{equation}
\begin{split}\label{fiden0}
& \frac{1}{2}\int_\Omega \rho J|\partial_t^\ell u(t)|^2
+\frac{1}{2}\int_0^t\int_{\Omega}\mu
J|\mathbb{D}_\mathcal{A}\partial_t^\ell u|^2
\\&\quad =\frac{1}{2}\int_\Omega \rho J|\partial_t^\ell u(0)|^2
+ \int_0^t\int_\Omega J(\partial_t^\ell u\cdot F^1+\partial_t^\ell
pF^2)-\int_0^t\int_{\Sigma}\partial_t^\ell u\cdot F^3
\\&\qquad  -\frac{1}{2}\int_0^t\int_{\Sigma_+}\rho_+g \partial_t^\ell \eta_+ \partial_t^\ell u_+\cdot \mathcal{N}_++\frac{1}{2}\int_0^t\int_{\Sigma_-}
 \rj g \partial_t^\ell \eta_- \partial_t^\ell u\cdot \mathcal{N}_-.
 \end{split}
 \end{equation}

First, we estimate the left hand side of \eqref{fiden0}. For this we write
\begin{equation}
\label{j1}
J|\mathbb{D}_\mathcal{A}\partial_t^\ell u|^2
=|\mathbb{D} \partial_t^\ell u|^2+(J-1)|\mathbb{D}
\partial_t^\ell u|^2+J(\mathbb{D}_\mathcal{A}\partial_t^\ell u+\mathbb{D}
\partial_t^\ell u) :(\mathbb{D}_\mathcal{A}\partial_t^\ell u-\mathbb{D}
\partial_t^\ell u).
\end{equation}
For the last term in the above, since
\begin{equation}
\mathbb{D}_\mathcal{A}\partial_t^\ell u\pm\mathbb{D} \partial_t^\ell u=
(\mathcal{A}_{ik}\pm\delta_{ik})\partial_k\partial_t^\ell
u^j+(\mathcal{A}_{jk}\pm\delta_{jk})\partial_k\partial_t^\ell u^i,
\end{equation}
we have that (using Lemma \ref{Poi} to estimate $\mathcal{A} \pm I$)
\begin{equation}\label{j2}
\int_\Omega
J(\mathbb{D}_\mathcal{A}\partial_t^\ell u+\mathbb{D} \partial_t^\ell
u) :(\mathbb{D}_\mathcal{A}\partial_t^\ell u-\mathbb{D}
\partial_t^\ell u)\lesssim \sqrt{{\mathcal{E}_{2N}}} \mathcal{D}_{2N}.
\end{equation}
On the other hand, we similarly have the estimate
\begin{equation} \label{j3}
\int_{\Omega}|J-1||\mathbb{D} \partial_t^\ell u|^2\lesssim \sqrt{{\mathcal{E}_{2N}}} \mathcal{D}_{2N}
\hbox{ and } \int_\Omega \rho |J-1|| \partial_t^\ell u
|^2(t)\lesssim({\mathcal{E}_{2N}})^{3/2}.
\end{equation}
This allows us to estimate the left hand side as
\begin{equation}\label{tempor0}
\begin{split}
& \frac{1}{2}\int_\Omega \rho J|\partial_t^\ell u(t)|^2
+\frac{1}{2}\int_0^t\int_{\Omega}\mu
J|\mathbb{D}_\mathcal{A}\partial_t^\ell u|^2 \\&
\quad\ge \frac{1}{2}\int_\Omega \rho  |\partial_t^\ell
u(t)|^2 +\frac{1}{2}\int_0^t\int_{\Omega}\mu
 |\mathbb{D}\partial_t^\ell u|^2-C({\mathcal{E}_{2N}}(t))^{3/2}-C\int_0^t\sqrt{{\mathcal{E}_{2N}(s)}} \mathcal{D}_{2N}(s)\,ds.
 \end{split}
 \end{equation}

We now estimate the right hand side of  \eqref{fiden0}. For the $F^1$ term, by \eqref{Fes1}, we have
\begin{equation}\label{tempor1}
\int_0^t\int_\Omega J\partial_t^\ell u\cdot F^1
 \le  \int_0^t\|\partial_t^\ell u\|_0\|J\|_{L^\infty}\|F^1\|_0  \lesssim
\int_0^t\sqrt{{\mathcal{E}_{2N}}} \mathcal{D}_{2N} .
\end{equation}
For the $F^3$ term, by again \eqref{Fes1}  and the trace theorem,  we have
\begin{equation}\label{tempor2}
\begin{split}
&-\int_0^t\int_{\Sigma}\partial_t^\ell u\cdot F^3 \le \int_0^t \|\partial_t^\ell u\|_{L^2(\Sigma)} \|
F^3\|_{0}  \lesssim \int_0^t \|\partial_t^\ell u\|_{1}
\sqrt{\mathcal{E}_{2N}  \mathcal{D}_{2N} } \lesssim \int_0^t
\sqrt{\mathcal{E}_{2N}  }\mathcal{D}_{2N} .
\end{split}
\end{equation}
Next, we estimate the $F^2$ term. Notice that we can not control $\partial_t^{2N} p$ by $\mathcal{D}_{2N}$, hence we need to integrate by parts in time to find
\begin{equation}
\int_0^t\int_\Omega J \partial_t^\ell pF^2
=-\int_0^t\int_\Omega \partial_t^{\ell-1}p\partial_t({JF^2})+
\int_\Omega (\partial_t^{\ell-1}p JF^2)(t) -\int_\Omega
(\partial_t^{\ell-1}p JF^2)(0).
\end{equation}
Then by \eqref{Fes1}, we may estimate
\begin{equation}
-\int_0^t\int_\Omega \partial_t^{\ell-1}p\partial_t({JF^2})
\le \int_0^t  \|\partial_t^{\ell-1}p\|_0\|\partial_t({JF^2})\|_0
\lesssim \int_0^t\sqrt{{\mathcal{E}_{2N}}} \mathcal{D}_{2N}
.
\end{equation} Also, it is easy to deduce that
\begin{equation} \label{tempor3}
\int_\Omega (\partial_t^{\ell-1}p JF^2)(t) -\int_\Omega (\partial_t^{\ell-1}p JF^2)(0)\lesssim \mathcal{E}_{2N}(0)
 + (\mathcal{E}_{2N}(t))^{ {3}/{2}}.
 \end{equation}
Hence, we have
\begin{equation}\label{tempor4}
\int_0^t\int_\Omega J \partial_t^\ell pF^2\lesssim \mathcal{E}_{2N}(0)
 + (\mathcal{E}_{2N}(t))^{ {3}/{2}}+ \int_0^t\sqrt{{\mathcal{E}_{2N}}} \mathcal{D}_{2N} .
 \end{equation}
For the last line of \eqref{fiden0}, by the trace theorem and Cauchy's inequality, we have
 \begin{equation}\label{tempor6}
 \begin{split}
& -\frac{1}{2}\int_0^t\int_{\Sigma_+}\rho_+g \partial_t^\ell \eta_+
\partial_t^\ell u_+\cdot
\mathcal{N}_++\frac{1}{2}\int_0^t\int_{\Sigma_-}
 \rj g \partial_t^\ell \eta_- \partial_t^\ell u\cdot
 \mathcal{N}_-
\\&\quad\le C\int_0^t\|
\partial_t^\ell \eta(s)\|_{0} \|\partial_t^\ell u(s)\|_0\,ds
\le C_\varepsilon \int_0^t\|
\partial_t^\ell \eta(s)\|_{0}^2\,ds+\varepsilon\int_0^t \|\partial_t^\ell u(s)\|_1^2\,ds .
 \end{split}
 \end{equation}

Using Korn's inequality together with the estimates \eqref{tempor0}, \eqref{tempor1}, \eqref{tempor2}, \eqref{tempor4}, and \eqref{tempor6}, we find that we may take  $\varepsilon$  sufficiently small to deduce from \eqref{fiden0} that
\begin{equation}\label{tempor5}
\begin{split}
&\|\partial_t^\ell
u(t)\|_0^2 + \int_0^t\| \partial_t^\ell
u(s)\|_1^2\,ds\\&\quad\lesssim \mathcal{E}
_{2N}(0)+({\mathcal{E}_{2N}}(t))^{3/2}+\int_0^t\sqrt{{\mathcal{E}
_{2N}(s)}} \mathcal{D}_{2N}(s)\,ds  +  \int_0^t\|
\partial_t^\ell \eta(s)\|_{0} ^2\,ds .
\end{split}
\end{equation}

Now taking $\ell=0$ in \eqref{tempor5}, we have
\begin{equation}\label{tempor7}
\begin{split}
  &\|
u(t)\|_0^2 + \int_0^t\|  u(s)\|_1^2\,ds \\&\quad\lesssim
\mathcal{E}_{2N}(0)+({\mathcal{E}
_{2N}}(t))^{3/2}+\int_0^t\sqrt{{\mathcal{E}_{2N}(s)}} \mathcal{D}
_{2N}(s)\,ds  +  \int_0^t\|  \eta(s)\|_{0} ^2\,ds
.
\end{split}
\end{equation}
For $\ell=1,\dots,2N$ the kinematic boundary condition $\eqref{nosurface}_3$ and the estimate \eqref{Fes1} imply that
\begin{equation}\label{tempor8}
  \|
\partial_t^\ell \eta\|_{0} \le\|
\partial_t^{\ell-1}u_3 \|_{L^2(\Sigma)} +\|
\partial_t^{\ell-1} (u \cdot \nabla_\ast\eta) \|_{L^2(\Sigma)}\le\|
\partial_t^{\ell-1}u \|_{1}+\sqrt{{\mathcal{E}
_{2N} }} \sqrt{\mathcal{D}_{2N}}.
\end{equation}
Plugging \eqref{tempor8} into \eqref{tempor5}, we obtain
\begin{equation}\label{tempor9}
\begin{split}
 &\|\partial_t^\ell
u(t)\|_0^2 + \int_0^t\| \partial_t^\ell
u(s)\|_1^2\,ds
\\&\quad\lesssim \mathcal{E}
_{2N}(0)+({\mathcal{E}_{2N}}(t))^{3/2}+\int_0^t\sqrt{{\mathcal{E}
_{2N}(s)}} \mathcal{D}_{2N}(s)\,ds  +  \int_0^t\|
\partial_t^{\ell-1} u(s)\|_1^2\,ds .
\end{split}
\end{equation}

Consequently, chaining \eqref{tempor9} and \eqref{tempor7}, we get \eqref{geoes}.
\end{proof}

 \subsubsection{Energy evolution in perturbed form}

We shall now use the perturbed equation \eqref{nosurface2} to derive the energy evolution of the mixed horizontal space-time derivatives. To proceed with, we present the estimates of $G^i$ defined by \eqref{G1}--\eqref{G4} in the following lemma.

\begin{lemma}\label{Gesle}
There exists a $\theta>0$ so that
\begin{equation} \label{Ges1}
\|\bar{\nabla}^{4N-2}G^1\|_0^2+ \|\bar{\nabla}^{4N-2}G^2\|_0^2 +\| \bna^{4N-2}G^3\|_{1/2}^2
+ \| \bna^{4N-2} G^4 \|_{1/2}^2
\lesssim \mathcal{E}_{2N}^{1+\theta},
\end{equation}
\begin{equation}\label{Ges2}
\begin{split}
&\|\bar{\nabla}^{4N-2} G^1\|_0^2 + \|\bar{\nabla}^{4N-2} G^2\|_0^2 + \| \bna^{4N-2}G^3\|_{1/2}^2
+ \|\bna^{4N-2} G^4\|_{1/2}^2 \\&\quad
+ \|\bar{\nabla}^{4N-3} \partial_tG^1\|_0^2 + \|\bar{\nabla}^{4N-3} \partial_t G^2\|_0^2
+ \|\bna^{4N-3} \partial_t G^3\|_{1/2}^2 + \|\bna^{4N-3} \partial_tG^4\|_{1/2}^2
\\&\qquad
\lesssim \mathcal{E}_{2N}^{\theta} \mathcal{D}_{2N},
\end{split}
\end{equation}
\begin{equation}\label{Ges3}
\| \nabla^{4N-1} G^1\|_0^2 + \|\nabla^{4N-1} G^2\|_0^2 +\| \na^{4N-1} G^3\|_{1/2}^2
+ \| \na^{4N-1} G^4\|_{1/2}^2
 \lesssim \mathcal{E}_{2N}^{\theta} \mathcal{D}_{2N} +\mathcal{E}_{2N}\mathcal{F}_{2N}.
\end{equation}

\end{lemma}
\begin{proof}
Since our perturbations $G^1,G^2,G^3,G^4$ have the same structure as those of \cite{GT_per}, the estimates \eqref{Ges1}--\eqref{Ges3} follow from Theorem 3.2 of \cite{GT_per}. In the statement of Theorem 3.2 of \cite{GT_per}, the left hand side of \eqref{Ges3} is bounded by $\mathcal{E}_{2N}^{\theta}\mathcal{D}_{2N}+\mathcal{K}\mathcal{F}_{2N}$ with
\begin{equation}
  \mathcal{K}:= \|\nabla u\|_{L^\infty}^2 + \|\nabla^2 u\|_{L^\infty}^2+\sum_{i=1}^2\|D u_i\|_{H^2(\Sigma)}^2.
\end{equation}
However, in the present case we may use the Sobolev embeddings to estimate  $\mathcal{K}\lesssim \mathcal{E}_{2N}$, and so we have \eqref{Ges3}.
\end{proof}

Now we estimate the energy evolution of the mixed horizontal space-time derivatives.
\begin{lemma}\label{hori}
There exists a $\theta>0$ such that for any $\kappa>0$, there exists a constant $ C_\kappa >0$ so that
\begin{equation}\label{peres}
\begin{split}
\sum_{|\alpha|\le 4N} &\|\partial^\alpha   u (t)\|_0^2
  +\int_0^t\sum_{|\alpha|\le 4N} \|\partial^\alpha   u (s)\|_1^2\,ds
  \\& \lesssim  {\mathcal{E}}_{2N}(0)+\int_0^t \left( \mathcal{E}_{2N}^{\theta}
 \mathcal{D}_{2N}(s) +\sqrt{\mathcal{E}_{2N} }\mathcal{F}_{2N}(s) + \kappa\mathcal{F}_{2N}(s) \,ds\right) +C_\kappa \int_0^t \| \eta(s)\|_{0}^2\,ds.
 \end{split}
\end{equation}
\end{lemma}
\begin{proof}
Since the boundaries of $\Omega_\pm$ are flat we are free to take the time derivatives and horizontal derivatives in the equations \eqref{nosurface2}.  Let $\alpha\in \mathbb{N}^{1+2}$ so that $\alpha_0\le 2N-1$ and $|\alpha|\le 4N$. We apply $\partial^\alpha$  to $\eqref{nosurface2}_1$, multiply the resulting equations by $\partial^\alpha u$, and then integrate  over $\Omega$,  using the other conditions in $\eqref{nosurface2}$;  using  trace estimates on the resulting equality, we find that
\begin{equation}\label{peres1}
\begin{split}
\frac{1}{2}\frac{d}{dt} \int_\Omega \rho |\partial^\alpha   u |^2
  +\frac{1}{2}\int_{\Omega}\mu|\mathbb{D} \partial^\alpha u
|^2&=\int_\Omega   \partial^\alpha u\cdot \partial^\alpha G^1 +
\partial^\alpha p \partial^\alpha G^2-\int_{\Sigma}\partial^\alpha
u\cdot
\partial^\alpha G^3
\\&\quad-\frac{1}{2}\int_{\Sigma_+} \rho_+g
\partial^\alpha\eta_+ \partial^\alpha u_{3,+}
+\frac{1}{2}\int_{\Sigma_-}\rj g
\partial^\alpha\eta_-\partial^\alpha u_{3} .
 \end{split}
\end{equation}

We will estimate the right hand side of \eqref{peres1}. First, we estimate the $G^1,G^2,G^3$ terms. Assume that $|\alpha|\le 4N-1$, then by the estimates \eqref{Ges2}--\eqref{Ges3} in Lemma \ref{Gesle} we have
\begin{equation}\label{m1}
\begin{split}
\displaystyle\int_\Omega   \partial^\alpha u\cdot
 \partial^\alpha G^1 +  \partial^\alpha p \partial^\alpha G^2
 &\lesssim \|\partial^\alpha u\|_0\|\partial^\alpha G^1\|_0+\|\partial^\alpha p\|_0\|\partial^\alpha G^2\|_0
 \\&\lesssim
\sqrt{\mathcal{D}_{2N}}\sqrt{\mathcal{E}_{2N}^{\theta}\mathcal{D}_{2N}+\mathcal{E}_{2N}\mathcal{F}_{2N}}
  . \end{split}
\end{equation}
Again by  \eqref{Ges2}--\eqref{Ges3}, together with the trace theorem, we have
\begin{equation}\label{m2}
\begin{split}-\int_{\Sigma}\partial^\alpha
u\cdot
\partial^\alpha G^3\lesssim \|\partial^\alpha
u\|_{L^2(\Sigma)}\|\partial^\alpha G^3\|_0 \lesssim
\sqrt{\mathcal{D}_{2N}}\sqrt{\mathcal{E}_{2N}^{\theta}\mathcal{D}_{2N}+\mathcal{E}_{2N}\mathcal{F}_{2N}}
 .\end{split}
\end{equation}
Now  we assume that $|\alpha|=4N$. Since $\alpha_0\le 2N-1$, we have $\alpha_1+\alpha_2\ge 2$, so we can integrate by parts in the horizontal directions. We write $\partial^\alpha=\partial^\beta\partial^\gamma$ so that $|\gamma|=4N-1$ and $|\beta| = 1$.  Integrating by parts and using the estimates \eqref{Ges2}--\eqref{Ges3}, we find that
\begin{equation}\label{m3}
\begin{split}
 \int_\Omega   \partial^\alpha u\cdot  \partial^\alpha G^1  &
 =-\int_\Omega   \partial^{\alpha+\beta} u\cdot  \partial^\gamma G^1
\lesssim \|\partial^{\alpha+\beta} u\|_0\|\partial^\gamma G^1\|_0
\\&\quad\lesssim
\sqrt{\mathcal{D}_{2N}}\sqrt{\mathcal{E}_{2N}^{\theta}\mathcal{D}_{2N}+\mathcal{E}_{2N}\mathcal{F}_{2N}}.
\end{split}
\end{equation}
For the $G^2$ term, we do not need to integrate by parts:
\begin{equation}\label{m4}
 \int_\Omega   \partial^\alpha p \partial^\alpha G^2
 \lesssim \|\partial^\alpha p\|_0\| \partial^\gamma G^2\|_1
 \lesssim \sqrt{\mathcal{D}_{2N}}\sqrt{\mathcal{E}_{2N}^{\theta}\mathcal{D}_{2N}+\mathcal{E}_{2N}\mathcal{F}_{2N}}.\end{equation}
For the $G^3$ term, we use the trace theorem to estimate
\begin{equation}\label{m5}
\begin{split}
-\int_{\Sigma}\partial^\alpha
u\cdot
\partial^\alpha G^3
&\lesssim \left| \int_{\Sigma}\partial^{\alpha+\beta} u \cdot
\partial^\gamma G^3\right|
 \lesssim\|\partial^{\alpha+\beta} u \|_{H^{-1/2}(\Sigma)}\|\partial^\gamma G^3\|_{1/2}
\\&\lesssim\|\partial^{\alpha} u \|_{H^{1/2}(\Sigma)}\|\partial^\gamma G^3\|_{1/2}
  \lesssim\|\partial^{\alpha} u \|_{1}\|\partial^\gamma G^3\|_{1/2}
\\&\lesssim \sqrt{\mathcal{D}_{2N}}\sqrt{\mathcal{E}_{2N}^{\theta}\mathcal{D}_{2N}+\mathcal{E}_{2N}\mathcal{F}_{2N}}
 .\end{split}
\end{equation}
Now we turn to estimating the last two terms in \eqref{peres1}. By the trace theorem and Cauchy's inequality, since $\alpha_0\le 2N-1$ and $|\alpha|\le 4N$, we have
\begin{equation}\label{m6}
\begin{split}
 -\frac{1}{2}\int_{\Sigma_+} \rho_+g
\partial^\alpha\eta_+ \partial^\alpha u_{3,+}
+\frac{1}{2}\int_{\Sigma_-}\rj g
\partial^\alpha\eta_-\partial^\alpha u_{3}&\le C\|\partial^\alpha\eta\|_{-1/2}\| \partial^\alpha u_{3}\|_{H^{1/2}(\Sigma)}
\\&\le
C_\varepsilon \|\partial_t^{\alpha_0}\eta\|_{4N-2\alpha_0-1/2}^2+\varepsilon\|
\partial^\alpha u \|_{1}^2.
\end{split}
\end{equation}

Consequently, in light of \eqref{m1}--\eqref{m6}, we may integrate \eqref{peres1} from $0$ to $t$ and sum over such $\alpha$.  Using Korn's inequality and taking $\varepsilon$ sufficiently small, we find that
\begin{equation}\label{ls0}
\begin{split}
 \sum_{\substack{|\alpha|\le 4N  \\ \alpha_0<2N}} &\|\partial^\alpha   u (t)\|_0^2
  +\int_0^t\sum_{\substack{|\alpha|\le 4N \\ \alpha_0<2N}} \|\partial^\alpha   u (s)\|^2\,ds
 \lesssim  {\mathcal{E}}_{2N}(0) \\
&+  \int_0^t \left( \mathcal{E}_{2N}^{\theta/2} \mathcal{D}_{2N}(s) +\sqrt{\mathcal{D}_{2N}\mathcal{E}_{2N} \mathcal{F}_{2N}}(s)\right) ds
 +\int_0^t\sum_{0\le \alpha_0\le 2N-1}\|\partial_t^{\alpha_0}\eta(s)\|_{4N-2\alpha_0-1/2}^2\,ds.
 \end{split}
 \end{equation}
If $\alpha_0=0$, we use  Sobolev interpolation to bound
\begin{equation}
\|\partial_t^{\alpha_0}\eta\|_{4N-2\alpha_0-1/2}^2=\| \eta\|_{4N
-1/2}^2\le C_\kappa \|\eta\|_{0}^2+\kappa\|\eta\|_{4N+1/2}^2=
C_\kappa \|\eta\|_{0}^2+\kappa \mathcal{F}_{2N}.\label{ls1}
\end{equation}
If $1\le \alpha_0\le 2N-1$, we use the kinematic boundary condition, the trace theorem and \eqref{Ges2}, to obtain
\begin{equation}\label{ls2}
\begin{split}
\|\partial_t^{\alpha_0}\eta\|_{4N-2\alpha_0-1/2}^2&\le
\|\partial_t^{\alpha_0-1}u_3\|_{H^{4N-2\alpha_0-1/2}}^2+\|\partial_t^{\alpha_0-1}G^4\|_{H^{4N-2\alpha_0-1/2}}^2
\\&\le \|\partial_t^{\alpha_0-1}u \|_{ {4N-2(\alpha_0-1)-2}} ^2+\|\partial_t^{\alpha_0-1}G^4\|_{H^{4N-2(\alpha_0-1)-5/2}}^2
\\&\le \|\partial_t^{\alpha_0-1}u \|_{
{4N-2(\alpha_0-1)-2}}
^2+\mathcal{E}_{2N}^{\theta}\mathcal{D}_{2N}.
\end{split}
\end{equation}
Hence, by \eqref{ls1}--\eqref{ls2}, we deduce from \eqref{ls0} that
\begin{equation}\label{ls3}
\begin{split}
&\sum_{\substack{|\alpha|\le 4N \\ \alpha_0<2N}} \|\partial^\alpha   u (t)\|_0^2
  +\int_0^t\sum_{\substack{|\alpha|\le 4N \\ \alpha_0<2N}} \|\partial^\alpha   u (s)\|_1^2\,ds
\\&\quad\lesssim  {\mathcal{E}}_{2N}(0)+\int_0^t \mathcal{E}_{2N}^{\theta/2}
 \mathcal{D}_{2N}+\sqrt{ \mathcal{E}_{2N}}\mathcal{F}_{2N}+\kappa\mathcal{F}_{2N}\,ds
 \\&\qquad+C_\kappa \int_0^t \| \eta(s)\|_{0}^2\,ds
 +\int_0^t\sum_{0\le \ell\le 2N-2}\|\partial_t^{\ell}u(s)\|_{4N-2\ell-2}^2\,ds.
 \end{split}
 \end{equation}

To conclude, we  combine \eqref{ls3} and \eqref{geoes} of  Lemma \ref{temp},  renaming $\theta$; This yields \eqref{peres}.
\end{proof}

 \subsubsection{Transport estimates}

Now we derive  estimates for $\eta$ by using the kinematic transport equation
 \begin{equation}\label{transport}
 \partial_t\eta+u\cdot \nabla_\ast\eta=u_3\quad
 \hbox{in }\Sigma,
 \end{equation}
where $u\cdot \nabla_\ast\eta=u_1\partial_1\eta+u_2\partial_2\eta$.
\begin{lemma}\label{transportlemma}
For any $\varepsilon>0$, there exists a constant $C_\varepsilon >0$ such that
\begin{equation}\label{transportes}
\begin{split}  \sum_{\ell=0}^{2N} \|\partial_t^\ell   \eta
(t)\|^2_{4-2\ell}+\mathcal{F}_{2N}(t) &\lesssim \mathcal{E}_{2N}(0)+\mathcal{F}_{2N}(0)+\int_0^t
\sqrt{\mathcal{E}_{2N}}(\mathcal{D}_{2N}+\mathcal{F}_{2N})\,ds
 \\&\quad+\varepsilon\int_0^t( \mathcal{E}_{2N} +
\mathcal{F}_{2N} )\,ds+C_\varepsilon \int_0^t\sum_{|\alpha|\le
4N}\|\partial^\alpha u(s)\|_{1}^2\,ds.
\end{split}
\end{equation}
\end{lemma}
\begin{proof}

 Let $\alpha\in \mathbb{N}^{1+2}$. Applying $\partial^\alpha$ with  $|\alpha|\le 4N$ to \eqref{transport}, we obtain
 \begin{equation}\label{transportalpha}
 \partial_t(\partial^\alpha\eta)+u\cdot \nabla_\ast(\partial^\alpha\eta)=\partial^\alpha u_3+
 R^\alpha.
 \end{equation}
Here the remainder term $R^\alpha$ is a sum of terms $\partial^\beta u\cdot \nabla_\ast(\partial^{\alpha-\beta}\eta)$ with $0\neq\beta\le \alpha$. As in Lemma \ref{Gesle}, we may bound this term as
 \begin{equation}\label{R_alpha}
 \|
 R^\alpha\|_{L^2(\Sigma)}^2\lesssim \mathcal{E}_{2N}(\mathcal{D}_{2N}+\mathcal{F}_{2N}).
 \end{equation}
We multiply \eqref{transportalpha} by $\partial^\alpha\eta$ and then integrate over $\Sigma$; by the Sobolev embedding theory on $\Sigma$, the trace theorem and \eqref{R_alpha}, we obtain the bound
 \begin{equation}
 \begin{split}
  \frac{1}{2}\frac{d}{dt}\int_\Sigma|\partial^\alpha\eta|^2&=
 -\frac{1}{2}\int_\Sigma u\cdot \nabla_\ast|\partial^\alpha\eta|^2+\int_\Sigma (\partial^\alpha u_3+R^\alpha)\partial^\alpha\eta
  \\&=
 \frac{1}{2}\int_\Sigma (\partial_1u_1+\partial_2u_2)|\partial^\alpha\eta|^2+\int_\Sigma (\partial^\alpha u_3+R^\alpha)\partial^\alpha\eta
   \\&\lesssim\|\nabla_\ast u\|_{L^\infty(\Sigma)}\|\partial^\alpha\eta\|_0^2
+(\|\partial^\alpha
u_3\|_{L^2(\Sigma)}+\|R^\alpha\|_{L^2(\Sigma)})\|\partial^\alpha\eta\|_{0}
 \\&\lesssim\| u\|_{4}\|\partial^\alpha\eta\|_0^2
+(\|\partial^\alpha
u_3\|_{1}+\|R^\alpha\|_{L^2(\Sigma)})\|\partial^\alpha\eta\|_{0}
 \\&\lesssim
\sqrt{\mathcal{E}_{2N}}(\mathcal{D}_{2N}+\mathcal{F}_{2N})+\sqrt{\mathcal{E}_{2N}}\|\partial^\alpha
u\|_{1}.
 \end{split}
 \end{equation}
Integrating the previous inequality in time from $0$ to $t$, summing over $\alpha$, we find that
 \begin{equation}\label{2nest}
 \begin{split}
 \sum_{\ell=0}^{2N}\|\partial_t^\ell\eta(t)\|_{4-2\ell}^2\lesssim \mathcal{E}_{2N}(0)+
\int_0^t\left(\sqrt{\mathcal{E}_{2N}}(\mathcal{D}_{2N}+\mathcal{F}_{2N})+\sqrt{\mathcal{E}_{2N}}\sum_{|\alpha|\le
4N} \|\partial^\alpha   u (s)\|_1\right)\,ds.
 \end{split}
 \end{equation}

Now we define the operator $\mathcal{J}=\sqrt{1-\Delta_\ast}$. We apply $\mathcal{J}^{4N+1/2}$  to \eqref{transport}, multiply the resulting equation by $\mathcal{J}^{4N+1/2}\eta$, and then integrate  over $(0,t)\times \Sigma$; using the commutator estimate of Lemma \ref{commutator} with $s=4N+1/2$, we find that
 \begin{equation}
 \begin{split}
  &\frac{1}{2}\frac{d}{dt}\int_\Sigma|\mathcal{J}^{4N+1/2}\eta|^2 \\&\quad=
 -\frac{1}{2}\int_\Sigma u\cdot \nabla_\ast|\mathcal{J}^{4N+1/2}\eta|^2+\int_\Sigma
   \left(\mathcal{J}^{4N+1/2} u_3-\left[\mathcal{J}^{4N+1/2},u\right]\cdot\nabla_\ast\eta\right) \mathcal{J}^{4N+1/2}\eta
  \\ & \quad=
 \frac{1}{2}\int_\Sigma (\partial_1u_1+\partial_2u_2)|\mathcal{J}^{4N+1/2}\eta|^2+
 \int_\Sigma  \left(\mathcal{J}^{4N+1/2} u_3-\left[\mathcal{J}^{4N+1/2},u\right]\cdot\nabla_\ast\eta\right) \mathcal{J}^{4N+1/2}\eta
    \\&\quad\lesssim\|\nabla_\ast u\|_{L^\infty(\Sigma)}\|\mathcal{J}^{4N+1/2}\eta\|_0^2
+\left(\|\mathcal{J}^{4N+1/2} u_3\|_{L^2(\Sigma)}+\|\nabla_\ast
u\|_{L^\infty(\Sigma)}\|\mathcal{J}^{4N-1/2}\nabla_\ast\eta\|_{0}\right.
 \\&\qquad+\left.\|\mathcal{J}^{4N+1/2}
u\|_{0}\|\nabla_\ast\eta\|_{L^\infty(\Sigma)}\right)\|\mathcal{J}^{4N+1/2}\eta\|_{0}
 \\&\quad\lesssim\| u\|_{4}\| \eta\|_{4N+1/2}^2 +(\|\mathcal{J}^{4N}
u_3\|_{1}+\| u\|_{4}\| \eta\|_{4N+1/2}+\|\mathcal{J}^{4N} u_3\|_{1}\|\eta\|_3)
\| \eta\|_{4N+1/2}  \\&\quad\lesssim
\sqrt{\mathcal{E}_{2N}}(\mathcal{D}_{2N}+\mathcal{F}_{2N})+\sqrt{\mathcal{F}_{2N}}\|\partial^\alpha
u\|_{1}.
 \end{split}
 \end{equation}
Integrating the above in time from $0$ to $t$, we find that
 \begin{equation}\label{2n1est}
 \begin{split}
 \mathcal{F}_{2N}(t)\lesssim
 \mathcal{F}_{2N}(0)+\int_0^t
\sqrt{\mathcal{E}_{2N}}(\mathcal{D}_{2N}+\mathcal{F}_{2N})+\sqrt{\mathcal{F}_{2N}}\|\partial^\alpha
u(s)\|_{1}\,ds.
 \end{split}
 \end{equation}

Summing  \eqref{2nest} and \eqref{2n1est} and applying Cauchy's inequality then yields \eqref{transportes}.
\end{proof}

 \subsubsection{Full energy estimates}

We will improve the estimates derived previously. To simplify notation, we define the ``horizontal'' energy as
\begin{equation}
\bar{\mathcal{E}}_{2N}:=\sum_{|\alpha|\le 4N} \|\partial^\alpha   u  \|_0^2+\sum_{\ell=0}^{2N} \|\partial_t^\ell   \eta
 \|_{4-2\ell}^2
\end{equation}
and the horizontal dissipation as
\begin{equation}
\bar{\mathcal{D}}_{2N}:=\sum_{|\alpha|\le 4N} \|\partial^\alpha   u  \|_1^2.
\end{equation}

First, chaining  \eqref{peres} and \eqref{transportes},  we obtain that for any $\varepsilon>0$, there exist a constant  $C_\varepsilon >0$ and a universal $C_0>0$ such that
\begin{equation}
\begin{split}
 & \bar{\mathcal{E}}_{2N}(t) +\mathcal{F}_{2N}(t)
  +\int_0^t \bar{\mathcal{D}}_{2N}\,ds
   \\&\quad\le  C_\varepsilon (\mathcal{E}_{2N}(0)+\mathcal{F}_{2N}(0))
+C_\varepsilon \int_0^t \mathcal{E}_{2N}^{\theta}
 \mathcal{D}_{2N}\,ds
  +C_\varepsilon \int_0^t
\sqrt{\mathcal{E}_{2N}} \mathcal{F}_{2N}\,ds
 \\&\qquad+\varepsilon C_0 \int_0^t( \mathcal{E}_{2N} +
\mathcal{F}_{2N} )+C_\varepsilon \kappa\mathcal{F}_{2N}\,ds +C_\kappa C_\varepsilon \int_0^t \| \eta(s)\|_{0}^2\,ds.
 \end{split}
 \end{equation}
If we set $\kappa=(\varepsilon  C_0)/C_\varepsilon$, then we may rewrite the inequality above as
\begin{equation}\label{horizontales}
\begin{split}
 & \bar{\mathcal{E}}_{2N}(t) +\mathcal{F}_{2N}(t)
  +\int_0^t \bar{\mathcal{D}}_{2N}\,ds
   \\&\quad\le  C_\varepsilon (\mathcal{E}_{2N}(0)+\mathcal{F}_{2N}(0))
+C_\varepsilon \int_0^t \mathcal{E}_{2N}^{\theta}
 \mathcal{D}_{2N}\,ds
  +C_\varepsilon \int_0^t
\sqrt{\mathcal{E}_{2N}} \mathcal{F}_{2N}\,ds
 \\&\qquad+ 2 \varepsilon C_0 \int_0^t( \mathcal{E}_{2N} +
\mathcal{F}_{2N} )\,ds + C_\varepsilon \int_0^t \|
\eta(s)\|_{0}^2\,ds.
\end{split}
\end{equation}

To conclude our estimates, we shall replace the left hand side of \eqref{horizontales} with the full energy and dissipation. Hence, in what follows we want to  show  that ${\mathcal{E}}_{2N}$ is comparable to $ \bar{\mathcal{E}}_{2N}$ and that ${\mathcal{D}}_{2N}$ is comparable to $ \bar{\mathcal{D}}_{2N}$, provided that $\delta$ is  sufficiently small.   We begin with the energy estimate.

\begin{lemma}\label{eth}
It holds that
\begin{equation}\label{e2n}
{\mathcal{E}}_{2N}\lesssim  \bar{\mathcal{E}}_{2N}  .
\end{equation}
\end{lemma}

\begin{proof}
We compactly write
 \begin{equation}\label{n1}
 \mathcal{Z}_{2N}=\sum_{j=0}^{n-1}\|\partial_t^{j} G^1\|_{4N-2j-2}^2
 +\|\partial_t^{j} G^2\|_{4N-2j-1}^2+\|\partial_t^{j} G^3\|_{4N-2j-3/2}^2.
 \end{equation}
 Note that by the definitions of $\bar{\mathcal{E}}_{2N}$, we have
\begin{equation} \label{n2}
\|\partial_t^{2N}u\|_{0}^2+\sum_{j=0}^{2N}\|\partial_t^j\eta\|_{4N-2j}^2\lesssim  \bar{\mathcal{E}}_{2N}.
\end{equation}

 Now we let $j=0,\dots,2N-1$ and then apply $\partial_t^j$ to the equations in \eqref{nosurface2} to find
\begin{equation}
\label{jellip}
\left\{\begin{array}{lll}-\mu\Delta \partial_t^j u+\nabla\partial_t^j p=-\rho\partial_t^{j+1} u+\partial_t^j G^1
\quad&\hbox{in }\Omega
\\ \diverge \partial_t^j u=\partial_t^j G^2&\hbox{in }\Omega
\\ \llbracket \partial_t^j p_+I-\mu_+\mathbb{D}(\partial_t^j u_+)\rrbracket e_3=\rho_+g\partial_t^j\eta_+ e_3+\partial_t^jG^3_+&\hbox{on }\Sigma_+
\\ \llbracket \partial_t^j u\rrbracket=0,
\ \llbracket\partial_t^j pI-\mu\mathbb{D}(\partial_t^j u)\rrbracket
e_3=\rj g\partial_t^j\eta_- e_3-\partial_t^jG_-^3&\hbox{on
}\Sigma_-
\\ \partial_t^j u_-=0 &\hbox{on }\Sigma_b.
\end{array}\right.
\end{equation}
Applying the elliptic estimates in Theorem \ref{cStheorem} with $r=4N-2j\ge 2$ to the problem \eqref{jellip} and using \eqref{n1}--\eqref{n2}, we obtain
\begin{equation}\label{n3}
\begin{split}
\|\partial_t^j u\|_{4N-2j}^2+\|\partial_t^j  p\|_{4N-2j-1}^2
&\lesssim\|\partial_t^{j+1} u\|_{4N-2 j-2}^2 + \|\partial_t^{j} G^1\|_{4N-2j-2}^2
\\
&\quad +\|\partial_t^{j} G^2\|_{4N-2j-1}^2
+\|\partial_t^{j} \eta\|_{4N-2j-3/2}^2+\|\partial_t^{j}
G^3\|_{4N-2j-3/2}^2
\\
& \lesssim\|\partial_t^{j+1} u\|_{4N-2(j+1)}^2+\bar{\mathcal{E}}_{2N}+\mathcal{Z}_{2N}.
\end{split}
\end{equation}

We claim that
\begin{equation}\label{claim}
\mathcal{E}_{2N}\lesssim\bar{\mathcal{E}}_{2N}+\mathcal{Z}_{2N}.
\end{equation}
We prove the claim \eqref{claim} by a finite induction based on the estimate \eqref{n3}. For $j=2N-1$, we obtain from \eqref{n3} and \eqref{n2} that
\begin{equation}\|\partial_t^{2N-1} u\|_{2}^2+\|\partial_t^{2N-1} p\|_{1}^2
\lesssim\|\partial_t^{2N} u\|_{0}^2
+\bar{\mathcal{E}}_{2N}+\mathcal{Z}_{2N}
\lesssim\bar{\mathcal{E}}_{2N}+\mathcal{Z}_{2N}.
\end{equation}
Now suppose that the following holds for $1\le \ell<2N$
\begin{equation}\label{n4}
\|\partial_t^{2N-\ell} u\|_{2\ell}^2+\|\partial_t^{2N-\ell}  p\|_{2\ell-1}^2
\lesssim\bar{\mathcal{E}}_{2N}+\mathcal{Z}_{2N}.
\end{equation}
We apply \eqref{n3} with $j=2N-(\ell+1)$ and use the induction hypothesis \eqref{n4} to find
\begin{equation}
\|\partial_t^{2N-(\ell+1)} u\|_{2(\ell+1)}^2+\|\partial_t^{4N-(\ell+1)}  p\|_{2(\ell+1)-1}^2
 \lesssim\|\partial_t^{2N-\ell} u\|_{2\ell}^2+\bar{\mathcal{E}}_{2N}+\mathcal{Z}_{2N}
\lesssim\bar{\mathcal{E}}_{2N}+\mathcal{Z}_{2N}.
\end{equation}
Hence by  finite induction, the bound \eqref{n4} holds for all $\ell=1,\dots,n.$ Summing \eqref{n4} over  $\ell=1,\dots,n$ and changing the index, we then have
 \begin{equation}\label{n5}
 \sum_{j=0}^{2N-1}\|\partial_t^j u\|_{4N-2j}^2
 +\|\partial_t^j  p\|_{4N-2j-1}^2\lesssim \bar{\mathcal{E}}_{2N}+\mathcal{Z}_{2N}.
 \end{equation}
We then conclude the claim \eqref{claim} by summing \eqref{n2} and \eqref{n5}.

Finally, using \eqref{Ges1} in Lemma \ref{Gesle} to bound $\mathcal{Z}_{2N}\lesssim ({\mathcal{E}}_{2N})^{1+\theta}$,  we obtain \eqref{e2n} if we take ${\mathcal{E}}_{2N}< \delta^2$ with $\delta$  small enough to absorb onto the left.
\end{proof}

Now we consider a similar estimate for the dissipation.

\begin{lemma}\label{dth}
It holds that
\begin{equation}\label{d2n}
{\mathcal{D}}_{2N}\lesssim  \bar{\mathcal{D}}_{2N} +
\mathcal{E}_{2N}{\mathcal{F}}_{2N}  .
\end{equation}
\end{lemma}
\begin{proof}

We  compactly write
 \begin{equation}\label{n11}
\mathcal{Z}_{2N}= \| \bar{\nabla}^{4N-1} G^1\|_{0}^2+\|
\bar{\nabla}^{4N-1} G^2\|_{1}^2
 +\|\bar{\nabla}_\ast^{4N-1} G^3\|_{1/2}^2+\| \bar{\nabla}_\ast^{4N-1} G^4\|_{1/2}^2+\| \bar{\nabla}_\ast^{4N-2} \partial_tG^4\|_{1/2}^2.
\end{equation}

First,  by the definition of $\bar{{\mathcal{D}}}_{ 2N} $ and Korn's inequality, we obtain
\begin{equation} \label{n14} \|\bar{\nabla}_\ast^{4N} u\|_{1}^2 \lesssim \bar{\mathcal{D}}_{2N}.\end{equation}

Notice that we have not yet derived at estimate of $\eta$ in terms of the dissipation, so we can not apply the two-phase elliptic estimates of Lemma \ref{cStheorem} as in Lemma \ref{eth}.  It is crucial to observe that  from \eqref{n14} we can get higher regularity estimates of $u$ on the boundaries $\Sigma=\Sigma_+\cup\Sigma_-$. Indeed, since $\Sigma_\pm$ are flat, we may use the definition of the Sobolev norm on $\mathrm{T}^2$  and the trace theorem to see from \eqref{n14} that
 \begin{equation}\label{n21}
 \begin{split}\|\partial_t^{j} u\|_{H^{4N-2j+1/2}(\Sigma)}^2 &\lesssim \|\partial_t^{j}   u \|_{L^2(\Sigma)}^2
 +\|\bar{\nabla}_\ast^{4N-2j}\partial_t^{j} u \|_{H^{1/2}(\Sigma )}^2
  \\&\lesssim \| \partial_t^{j}   u \|_{1}^2+\|\bar{\nabla}_\ast^{4N-2j}\partial_t^{j}   u \|_{1}^2
 \lesssim \bar{\mathcal{D}}_{2N}.
 \end{split}
 \end{equation}
 This motivates us to use the one-phase elliptic estimates of Lemma \ref{cS1phaselemma2}.

Let $j=0,\dots,2N-1$, and observe that $(\partial_t^j u_+,\partial_t^j p_+) $ solve the  problem
 \begin{equation} \label{pro+1}
 \left\{\begin{array}{lll}-\mu_+\Delta \partial_t^j u_+
 +\nabla\partial_t^j p_+=-\rho_+\partial_t^{j+1} u_++\partial_t^j G^1_+ &\hbox{in } \Omega_+
\\ \diverge \partial_t^j u_+=\partial_t^j G^2_+&\hbox{in } \Omega_+
\\ \partial_t^j u_+=\partial_t^j u_+&\hbox{on } \Sigma,
\end{array}\right.
\end{equation}
and  $(\partial_t^j u_-,\partial_t^j p_-) $ solve the  problem
 \begin{equation} \label{pro-1}
 \left\{\begin{array}{lll}-\mu_-\Delta \partial_t^j u_-+\nabla\partial_t^j p_-
 =-\rho_-\partial_t^{j+1} u_-+\partial_t^j G^1_- &\hbox{in } \Omega_-
\\ \diverge \partial_t^j u_-=\partial_t^j G^2_-&\hbox{in } \Omega_-
\\ \partial_t^j u_-=\partial_t^j u_-&\hbox{on } \Sigma_-
\\ \partial_t^j u_-=0 &\hbox{on } \Sigma_b.
\end{array}\right.
\end{equation}
We apply Lemma \ref{cS1phaselemma2} with $r=4N-2j+1$ to the problem \eqref{pro+1} for $u_+,\ p_+$ and to the  problem \eqref{pro-1} for $u_-,\ p_-$, respectively; using \eqref{n11}, \eqref{n14}, \eqref{n21} and summing up,  we find
 \begin{equation}\label{n31}
 \begin{split}
 &\|\partial_t^j u\|_{4N-2j+1}^2+\| \nabla\partial_t^j p\|_{4N-2j-1}^2
\\&\quad\lesssim\|\partial_t^{j+1} u\|_{4N-2j-1}^2+
\|\partial_t^j G^1\|_{4N-2j-1}^2+\|\partial_t^{j}
G^2\|_{4N-2j}^2+\|\partial_t^{j}
u\|_{H^{4N-2j+1/2}(\Sigma)}^2
\\&\quad\lesssim\|\partial_t^{j+1} u\|_{4N-2j-1}^2+\mathcal{Z}_{2N}+  \bar{\mathcal{D}}_{2N}.
\end{split}
\end{equation}

We now claim that
\begin{equation}\label{claim2}
\sum_{j=0}^{ 2N}\|\partial_t^ju\|_{4N-2j+1}^2+\sum_{j=0}^{ 2N-1}\|\partial_t^j\nabla p\|_{4N-2j-1}
\lesssim\mathcal{Z}_{2N}+  \bar{\mathcal{D}}_{2N}.
\end{equation}
We prove the claim \eqref{claim2} by a finite induction  as in Lemma \ref{eth}. For $j=2N-1$, by \eqref{n11} and \eqref{n31}, we obtain
\begin{equation}
\|\partial_t^{2N-1} u\|_{3}^2+\| \nabla\partial_t^{2N-1} p\|_{1}^2
\lesssim\|\partial_t^{2N} u\|_{1}^2 +\mathcal{Z}_{2N}+
\bar{\mathcal{D}}_{2N} \lesssim \mathcal{Z}_{2N}+
\bar{{\mathcal{D}}}_{ 2N}.
\end{equation}
Now suppose that the following holds for $1\le \ell<2N$:
\begin{equation}\label{n41}
\|\partial_t^{2N-\ell} u\|_{ 2\ell+1}^2+\| \nabla\partial_t^{2N-\ell} p\|_{2\ell-1}^2
\lesssim\mathcal{Z}_{2N}+  \bar{{\mathcal{D}}}_{ 2N}.
\end{equation}
We apply \eqref{n31} with $j=2N-(\ell+1)$ and use the induction hypothesis \eqref{n41} to find
 \begin{equation}
 \begin{split}\|\partial_t^{2N-(\ell+1)} u\|_{2(\ell+1)+1}^2+\| \nabla\partial_t^{2N-(\ell+1) }p\|_{2(\ell+1)-1}^2
&\lesssim\|\partial_t^{2N-\ell} u\|_{2\ell+1}^2+\mathcal{Z}_{2N}+
\bar{\mathcal{D}}_{2N}
\\&\lesssim\mathcal{Z}_{2N}+  \bar{\mathcal{D}}_{2N}.
\end{split}
\end{equation}
Hence the bound \eqref{n41} holds for all $l=1,\dots,2N.$ We then conclude the claim \eqref{claim2} by summing this over $\ell=1,\dots, n$, adding \eqref{n14} and  changing the index.

Now that we have obtained \eqref{claim2}, we estimate the remaining parts in $\mathcal{D}_{2N}$. We will turn to the boundary conditions in \eqref{nosurface2}. First we derive estimates for $\eta$. For the term $\dt^j \eta$ for $j\ge 2$ we use the boundary condition
\begin{equation}\label{n61}
\partial_t\eta=u_3+G^4\hbox{ on }\Sigma.
\end{equation}
Indeed, for $j=2,\dots,2N+1$ we apply $\partial_t^{j-1}$ to \eqref{n61} to see, by \eqref{claim2} and \eqref{n11}, that
\begin{equation}\label{eta1}
\begin{split}\|\partial_t^j\eta\|_{4N-2j+5/2}^2&\lesssim \|\partial_t^{j-1}u_3\|_{H^{4N-2j+5/2}(\Sigma)}^2
+\|\partial_t^{j-1}G^4\|_{4N-2j+5/2}^2
\\&\lesssim \|\partial_t^{j-1}u_3\|_{{2n-2(j-1)+1}}^2+\|\partial_t^{j-1}G^4\|_{4N-2(j-1)+1/2}^2
 \lesssim\mathcal{Z}_{2N}+  \bar{\mathcal{D}}_{2N}.
\end{split}
\end{equation}
For the term $\partial_t\eta$, we again use \eqref{n61}, \eqref{claim2} and \eqref{n11} to find
\begin{equation}\label{eta2}
\begin{split}\|\partial_t \eta\|_{4N-1/2}^2&\lesssim \| u_3\|_{H^{4N-1/2}(\Sigma)}^2+\|\partial_t^{j-1}G^4\|_{4N-1/2}^2
\\&\lesssim \| u_3\|_{{4N }}^2+\| G^4\|_{4N-1/2}^2 \lesssim\mathcal{Z}_{2N}+  \bar{\mathcal{D}}_{2N}.
\end{split}
\end{equation}
For the remaining $\eta$ term, i.e. those without temporal derivatives, we use the boundary conditions
\begin{equation}\label{pb1}
 \rho g\eta_+= p_+- \mu_+\partial_3u_{3,+} -G_{3,+}^3\hbox{ on }\Sigma_+
  \end{equation}
 and
\begin{equation} \label{pb2}
\rj g\eta_-=\llbracket
p\rrbracket-\llbracket\mu\partial_3u_3\rrbracket +G_{3,-}^3\hbox{ on
}\Sigma_-.
\end{equation}
Notice that at this point we do not have any bound on $p$ on the boundary $\Sigma$, but we have bounded  $\nabla p$ in $\Omega$. Applying $ {\nabla}_\ast$ to \eqref{pb1} and \eqref{pb2}, respectively,  by \eqref{claim2} and \eqref{n11}, we obtain
\begin{equation}\label{n51}
\begin{split}
\|\nabla_\ast\eta\|_{4N-3/2}^2 &\lesssim \|  {\nabla}_\ast p \|_{H^{4N-3/2}(\Sigma)}^2 + \|  {\nabla}_\ast\partial_3u_3 \|_{H^{2n-3/2}(\Sigma)}^2 +\|{\nabla}_\ast G^3\|_{ 2n-3/2 }^2  \\
& \lesssim \|\nabla p\|_{4N-1}^2 + \|u_3\|_{4N+1}^2 +
\|G^3\|_{4N-1/2}^2 \lesssim \mathcal{Z}_{2N}+
\bar{\mathcal{D}}_{2N}.
\end{split}
\end{equation}
Since $\int_{\mathrm{T}^2}\eta=0$, we may then use Poincar\'e's inequality on $\Sigma_\pm$   to obtain from \eqref{n51} that
\begin{equation} \label{eta3}
\|\eta\|_{4N-1/2}^2\lesssim \|\eta\|_{0}^2+\|\nabla_\ast \eta\|_{4N-3/2}^2\lesssim\|\nabla_\ast\eta\|_{4N-3/2}^2
\lesssim\mathcal{Z}_{2N}+  \bar{{\mathcal{D}}}_{2N}.
\end{equation}
Summing \eqref{eta1}, \eqref{eta2}  and \eqref{eta3}, we complete the estimates for $\eta$:
\begin{equation} \label{eta0}
\|\eta\|_{4N-1/2}^2+\|\partial_t\eta\|_{4N-1/2}^2
+\sum_{j=2}^{
n+1}\|\partial_t^j\eta\|_{4N-2j+5/2}\lesssim\mathcal{Z}_{2N}+
\bar{\mathcal{D}}_{2N}.
\end{equation}

It remains to bound $\|\partial_t^jp\|_0$. Applying $\partial_t^j,\ j=0,\dots,2N-1$ to \eqref{pb1}--\eqref{pb2}, by \eqref{claim2}, \eqref{eta0} and \eqref{n11}, we find
\begin{equation}\label{pp1}
\begin{split}\|\partial_t^j p_+\|_{L^2(\Sigma_+)}^2+\|\llbracket \partial_t^j p\rrbracket\|_{L^2(\Sigma_-)}^2&\lesssim \|\partial_t^j \eta\|_{0}^2+\|\partial_3\partial_t^j u_3 \|_{L^2(\Sigma)}^2+\|\partial_t^j G^3\|_{0}^2
 \\&\lesssim \|\partial_t^j \eta\|_{0}^2+\| \partial_t^j u_3
\|_{2}^2+\|\partial_t^j G^3\|_{0}^2 \lesssim\mathcal{Z}_{2N}+
\bar{\mathcal{D}}_{2N}.
\end{split}
\end{equation}
By Poincar\'e's inequality on $\Omega_+$ (Lemma \ref{poincare}) and \eqref{claim2} and \eqref{pp1}, we have
\begin{equation}\label{pp2}
\begin{split}\|\partial_t^j p\|_{H^1(\Omega_+)}^2=\|\partial_t^j p\|_{L^2(\Omega_+)}^2+\|\nabla\partial_t^j p\|_{L^2(\Omega_+)}^2&\lesssim\|\partial_t^j p\|_{L^2(\Sigma_+)}^2+\|\nabla\partial_t^j p\|_{L^2(\Omega_+)}^2
  \\&\lesssim\mathcal{Z}_{2N}+  \bar{\mathcal{D}}_{2N}.
  \end{split}
  \end{equation}
On the other hand, by the trace theorem and \eqref{pp1}--\eqref{pp2}, we have
\begin{equation}\label{pp3}
\begin{split}\|\partial_t^j p_-\|_{L^2(\Sigma_-)}^2\le \|\partial_t^j p_+\|_{L^2(\Sigma_-)}^2+\|\llbracket \partial_t^j p\rrbracket\|_{L^2(\Sigma_-)}^2& \lesssim \|\partial_t^j p_+\|_{H^1(\Omega_+)}^2+\|\llbracket \partial_t^j p\rrbracket\|_{L^2(\Sigma_-)}^2
 \\&\lesssim\mathcal{Z}_{2N}+  \bar{{\mathcal{D}}}_{ n},
\end{split}
\end{equation}
so again by Poincar\'e's inequality on $\Omega_-$ as well as \eqref{claim2} and \eqref{pp1}, we have
\begin{equation}\label{pp4}
\begin{split}\|\partial_t^j p\|_{H^1(\Omega_-)}^2=\|\partial_t^j p\|_{L^2(\Omega_-)}^2+\|\nabla\partial_t^j p\|_{L^2(\Omega_-)}^2&\lesssim\|\partial_t^j p\|_{L^2(\Sigma_-)}^2+\|\nabla\partial_t^j p\|_{L^2(\Omega_-)}^2
  \\&\lesssim\mathcal{Z}_{2N}+  \bar{\mathcal{D}}_{2N}.
  \end{split}
  \end{equation}
In light of \eqref{pp2}  and \eqref{pp4}, we may improve the estimate \eqref{claim2} to be
\begin{equation}\label{up0}
\sum_{j=0}^{2N}\|\partial_t^ju\|_{4N-2j+1}^2+\sum_{j=0}^{ 2N-1}\|\partial_t^j p\|_{4N-2j}^2\lesssim\mathcal{Z}_{2N}+  \bar{\mathcal{D}}_{2N}.
\end{equation}

Consequently, summing \eqref{eta0} and \eqref{up0}, we obtain
\begin{equation}\label{upeta0}
{\mathcal{D}}_{n}\lesssim\mathcal{Z}_{2N}+  \bar{\mathcal{D}}_{2N}.
\end{equation}
Thus,  using \eqref{Ges2}--\eqref{Ges3} in Theorem \ref{Gesle} to bound $\mathcal{Z}_{2N}\lesssim ({\mathcal{E}}_{2N})^{\theta} {\mathcal{D}}_{2N}+\mathcal{E}_{2N}{\mathcal{F}}_{2N}$, we obtain \eqref{d2n} provided that  ${\mathcal{E}}_{2N}< \delta^2$ with $\delta$ small enough to absorb onto the left.
\end{proof}

Now we are ready to state the full energy estimates.

\begin{theorem}\label{energywithoutout}
Let $\Lambda_\ast$ be defined by \eqref{littlelambda}, and let $(u,p,\eta)$ solve \eqref{nosurface2}. If $\mathcal{E}_{2N}(t)+\mathcal{F}_{2N}(t) \le \delta^2$ for sufficiently small $\delta$ for all $t \in [0,T]$, then we have that
\begin{equation}\label{fulles}
\begin{split} &\mathcal{E}_{2N}(t)+\mathcal{F}_{2N}(t)
  +\int_0^t\mathcal{D}_{2N}(s)\,ds
 \\&\quad\le  C (\mathcal{E}_{2N}(0)+\mathcal{F}_{2N}(0))
   +  \frac{\Lambda_\ast}{2}\int_0^t( \mathcal{E}_{2N} +
\mathcal{F}_{2N} )\,ds+C\int_0^t \sqrt{\mathcal{E}_{2N}}
\mathcal{F}_{2N} \,ds + C \int_0^t \|
\eta(s)\|_{0}^2\,ds
\end{split}
\end{equation}
for all $t \in [0,T]$.
\end{theorem}
\begin{proof}
By \eqref{e2n} and \eqref{d2n}, we deduce from \eqref{horizontales} that for any $\varepsilon>0$, there exist $C_\varepsilon>0$ and $C_1>0$ such that
\begin{equation}\label{tot}
\begin{split} &\mathcal{E}_{2N}(t)+\mathcal{F}_{2N}(t)
  +\int_0^t\mathcal{D}_{2N}(s)\,ds
   \\&\quad\le  C_\varepsilon (\mathcal{E}_{2N}(0)+\mathcal{F}_{2N}(0))
+C_\varepsilon \int_0^t \mathcal{E}_{2N}^{\theta}
 \mathcal{D}_{2N}\,ds
  +C_\varepsilon \int_0^t
\sqrt{\mathcal{E}_{2N}} \mathcal{F}_{2N}\,ds
 \\&\qquad+\varepsilon C_1 \int_0^t( \mathcal{E}_{2N} +
\mathcal{F}_{2N} )\,ds + C_\varepsilon \int_0^t \|
\eta(s)\|_{0}^2\,ds.
\end{split}
\end{equation}
We fix the value of $\varepsilon$ by $\varepsilon=\Lambda_\ast/(2 C_1)$ with $\Lambda_\ast$ defined by \eqref{littlelambda}.  Taking  $\delta$ sufficiently small, we deduce from \eqref{tot} that there exists a constant $C=C_{\Lambda_\ast}$ so that
\begin{equation}
\begin{split}
&\mathcal{E}_{2N}(t)+\mathcal{F}_{2N}(t)
  +\int_0^t\mathcal{D}_{2N}(s)\,ds
  \\&\quad\le  C (\mathcal{E}_{2N}(0)+\mathcal{F}_{2N}(0))
   +  \frac{\Lambda_\ast}{2}\int_0^t( \mathcal{E}_{2N} +
\mathcal{F}_{2N} )\,ds+C\int_0^t \sqrt{\mathcal{E}_{2N}}
\mathcal{F}_{2N}  + C \int_0^t \|
\eta(s)\|_{0}^2\,ds.
\end{split}
\end{equation}
This is exactly \eqref{fulles} and the proof is completed.
\end{proof}

 \section{Nonlinear instability}

 \subsection{Unified operator form}\label{unified operator}

Note that the nonlinear problems \eqref{surface} and \eqref{nosurface2} can be written uniformly as
\begin{equation}\label{surfaceLP}
  \left\{\begin{array}{lll}
  \rho\partial_t u-\mu\Delta u+\nabla p=F^1 &\hbox{in } \Omega
\\ \diverge u=\diverge F^2&\hbox{in } \Omega
\\ \partial_t\eta-u_3= -F^2\cdot e_3 &\hbox{on } \Sigma
\\ ( p_+I-\mu_+\mathbb{D}(u_+)) e_3-(\rho_+g\eta_+  -\sigma_+\Delta_\ast\eta_+) e_3=F^3_{+}&\hbox{on }\Sigma_+
\\\Lbrack {u}\Rbrack=0,\quad \Lbrack pI-\mu\mathbb{D}(u)\Rbrack e_3 -(\rj g\eta_- +\sigma_-\Delta_\ast\eta_- )e_3=-F^3_{-}&\hbox{on }\Sigma_-
\\ u_-=0 &\hbox{on } \Sigma_b
\\(u,\eta )\mid_{t=0}=(u_0,\eta_0).
\end{array}\right.
\end{equation}
Here, the nonlinear terms in the right hand sides of \eqref{surfaceLP} are given as follows: for $\sigma_\pm>0$, \eqref{surface} corresponds to \eqref{surfaceLP} with $F^1=f$, $F^2=0$ and $F^3=g$, where $f$ and $g$ are defined by \eqref{f}--\eqref{g_-^3}; for $\sigma_\pm=0$, \eqref{nosurface2} corresponds to \eqref{surfaceLP} with $F^1=G^1$, $F_i^2=(\delta_{ji}-J\mathcal{A}_{ji})u_j,\,i=1,2,3$ and $F^3=G^3$, where $G^1$ and $G^3$ are defined by \eqref{G1}, \eqref{G3+} and \eqref{G3-}. Note that for $\sigma_\pm=0$, $\eqref{surfaceLP}_2$ is derived as follows:
\begin{equation}
\diverge_{\mathcal{A}}u=0\Longleftrightarrow J\mathcal{A}_{ij}\partial_ju_i=0\Longleftrightarrow \partial_j(J\mathcal{A}_{ij} u_i)=0
\Longleftrightarrow \diverge u=\diverge F^2.
\end{equation}
This implies that $\diverge{F^2} = G^2$, with $G^2$ defined by \eqref{G2}. We can also check, using the definitions in \eqref{ABJ_def}, that $F^2\cdot e_3 = -G^4$ with $G^4$ defined by \eqref{G4} as follows: on the boundary $\Sigma$,
\begin{equation}
F^2\cdot e_3=(\delta_{ji}-J\mathcal{A}_{ji})u_j\delta_{i3}=u_3-u\cdot \mathcal{N}=-G^4.
\end{equation}

In order to use the linear theory we have developed in  Section \ref{linear instability}, we shall write the system \eqref{surfaceLP} in the operator form as in \cite{B1,B2,BN}.  First, we want to reduce the divergence of velocity to be zero. This can be  easily done by setting $\tilde{u}=u-F^2$. Then $(\tilde{u},\eta,p)$ satisfy \eqref{surfaceLP},   replacing $F^1$ by $\tilde{F}^1=F^1-\rho \partial_t F^2+\mu\Delta F^2$, $F^2$ by $0$,  and $F^3$ by $\tilde{F}^3=F^3+\mu\mathbb{D}(F^2)$. Second, we reduce the tangential components of the stress tensor $\tilde{F}^3_i,\ i=1,2$ to be zero. We choose the vector $z$ satisfying the conditions (see Page 6 of \cite{BN}, for instance, for details of how to construct $z$)
\begin{equation}
\left\{\begin{array}{ccc}
z_+=0,\ \partial_3z_+=0,\ -\mu_+\partial_3^2z_+
=(\tilde{F}_{2,+}^3,-\tilde{F}_{1,+}^3,0)\hbox{ on }\Sigma_+,\  z_+\hbox{ vanishes near }\Sigma_-,
\\
z_-=0,\ \partial_3z_-=0,\ -\mu_-\partial_3^2z_-=(-\tilde{F}_{2,-}^3,\tilde{F}_{1,-}^3,0)\hbox{
on }\Sigma_-,\  z_-\hbox{ vanishes near
}\Sigma_b.
\end{array}\right.
\end{equation}
Then $w=\nabla\times z$ satisfies
\begin{equation}
\left\{\begin{array}{lll}
w=0,\ -\mu_+(\partial_3w_i+\partial_iw_3)=\tilde{F}_{i,+}^3,\ \partial_3w_3=0\hbox{ on }\Sigma_+,
\\
w=0,\
-\llbracket\mu(\partial_3w^i+\partial_iw^3)\rrbracket=-\tilde{F}_{i,-}^3,\  \partial_3w_3=0\hbox{
on }\Sigma_-, \\ \diverge w=0\hbox{ in }\Omega,\hbox{ and } w=0\hbox{ on
}\Sigma_b,
\end{array}\right.
\end{equation}
and we have the estimate
\begin{equation}\label{westimate}
\norm{w}_{r}\lesssim \|\tilde{F}^3\|_{r-3/2},\ r\ge 2.
\end{equation}

Therefore, let $\bar{u}=\tilde{u}-w=u-F^2-w$, then $(\bar{u},p,\eta)$ satisfy
  \begin{equation}\label{surfaceLP2}
  \left\{\begin{array}{lll}
  \rho\partial_t \bar{u}-\mu\Delta \bar{u}+\nabla p
  =\bar{F}^1:=\tilde{F}^1-\rho\partial_tw+\mu\Delta w&\hbox{in }\Omega
\\ \diverge \bar{u}=0&\hbox{in }\Omega
\\ \partial_t\eta-\bar{u}_3=0 &\hbox{on }\Sigma
\\ S_{tan}( \bar{u}_+)=0&\hbox{on }\Sigma_+
\\  p_+-2\mu_+\partial_3\bar{u}_{3,+}- \rho_+g\eta_+  +\sigma_+\Delta_\ast\eta_+=\tilde{F}^3_{3,+}&\hbox{on }\Sigma_+
\\\Lbrack \bar{u}\Rbrack=\Lbrack  S_{tan}(\bar{u})\Rbrack=0&\hbox{on }\Sigma_-
\\\Lbrack p-2\mu\partial_3\bar{u}_{3}\Rbrack -\rj g\eta_-
-\sigma_-\Delta_\ast\eta_-=-\tilde{F}^3_{3,-}&\hbox{on }\Sigma_-
\\ \bar{u}_-=0 &\hbox{on }\Sigma_b
\\(\bar{u},\eta )\mid_{t=0}=(\bar{u}_0,\eta_0),
\end{array}\right.
\end{equation}
where we have written $S_{tan}( \bar{u})$ for the tangential part of $S(p,\bar{u})$.

We want to introduce a projection on divergence-free vectors to remove the pressure as an unknown. With fixed boundaries the pressure term has no influence on the projected equation; in the present case, we shall see that it is replaced by a gradient term which is determined by the other unknowns. Let $\mathcal{P}$ be the orthogonal projection in $L^2(\Omega)$  onto the space orthogonal to $\{\nabla\phi\mid\phi\in {}^0H^1(\Omega)\}$, where ${}^0H^1(\Omega) := \{ u \in H^1(\Omega) \; \vert\; u\vert_{\Sigma_+} =0\}$.  Following  Lemma 3.1 of \cite{B1}, we can prove that $\mathcal{P}$ is a bounded operator on $\ddot{H}^s(\Omega)$ for $s\ge 0$ and that if $\phi\in \ddot{H}^{s+1}(\Omega)$, then $\mathcal{P}(\nabla\phi)=\nabla\pi$, where $\pi$ solves the problem
\begin{equation}\label{harmonic}
\Delta\pi=0\hbox{ in }\Omega,\ \pi=\phi\hbox{ on }\Sigma_+
,\ \llbracket\pi\rrbracket=\llbracket\phi\rrbracket\hbox{ and }
\llbracket\partial_3\pi\rrbracket=0\hbox{ on }\Sigma_-,\
 \partial_\nu\pi=0\hbox{ on
}\Sigma_b.
\end{equation}
We may regard $\pi$ as a harmonic representation of $\phi$ and we denote it  as $\pi=\mathcal{H}(\phi|_{\Sigma_+},\llbracket\phi\rrbracket|_{\Sigma_-} )$ to illustrate its dependence on the boundary values of $\phi$.

Now we apply $\mathcal{P}$ to the momentum equation $\eqref{surfaceLP2}_1$  to obtain
\begin{equation}
\rho\partial_t\bar{u}+\mathbb{A}\bar{u}+\nabla \mathcal{H}(\rho_+g\eta_+
-\sigma_+\Delta_\ast\eta_+, \rj g\eta_-
+\sigma_-\Delta_\ast\eta_-)=\mathcal{P}\bar{F}^1-\nabla\mathcal{H}(\tilde{F}^3_{3,+},-\tilde{F}^3_{3,-}),
\end{equation}
where $\mathbb{A}$ is defined by
\begin{equation}
\mathbb{A}v=-\mu \mathcal{P}\Delta v+\nabla \mathcal{H}(2\mu_+\partial_3v_3|_{\Sigma_+},
2\llbracket\mu\partial_3v_3\rrbracket|_{\Sigma_-}).
\end{equation}
The domain condition of $\mathbb{A}$ is given by
\begin{equation}
\diverge v=0\hbox{ in }\Omega,\
 S_{tan}(v)=0\hbox{ on }\Sigma_+,\ \llbracket v
\rrbracket=\llbracket S_{tan}(v)\rrbracket=0\hbox{ on
}\Sigma_-,\ v=0\hbox{ on }\Sigma_b.
\end{equation}
Now we define the linear operator $\mathcal{L}$ by
\begin{equation}
\mathcal{L}
\begin{pmatrix}  \bar{u}\\
\eta
 \end{pmatrix}
=\begin{pmatrix}
\rho^{-1}\left(\mathbb{A}\bar{u}+\nabla \mathcal{H}(\rho_+g\eta_+
-\sigma_+\Delta_\ast\eta_+, \rj g\eta_-
+\sigma_-\Delta_\ast\eta_-)\right)\\-u_3
 \end{pmatrix},
\end{equation}
which then allows us to rewrite \eqref{surfaceLP2} as
\begin{equation}\label{surfaceLP3}
\partial_t\begin{pmatrix}  \bar{u}\\
\eta
 \end{pmatrix}
 +\mathcal{L}\begin{pmatrix}  \bar{u}\\
\eta
 \end{pmatrix}
=\begin{pmatrix}   \mathfrak{N}\\0
 \end{pmatrix} ,
\end{equation}
where the nonlinear term $\mathfrak{N}$ is given by
\begin{equation}\label{Ndef}
\mathfrak{N}:=\rho^{-1}\left(\mathcal{P}\bar{F}^1-\nabla\mathcal{H}(\tilde{F}^3_{3,+},-\tilde{F}^3_{3,-})\right).
\end{equation}
Employing Duhamel's principle, we can then solve \eqref{surfaceLP3} via
\begin{equation}
\begin{pmatrix}   \bar{u}(t)\\ \eta(t)
 \end{pmatrix}
 =e^{t\mathcal{L}}
\begin{pmatrix} \bar{u}(0)\\  \eta(0)
 \end{pmatrix}+\int_0^te^{(t-s)\mathcal{L}}
\begin{pmatrix}   \mathfrak{N}(s)\\ 0
 \end{pmatrix}\,ds.
\end{equation}
We can then go back to $u=\bar{u}+w+F^2$ to see that
\begin{equation}\label{duhamel}
\begin{pmatrix}    u(t)\\\eta(t)
 \end{pmatrix} =e^{t\mathcal{L}}
\begin{pmatrix}     \bar{u}(0)\\\eta(0)
 \end{pmatrix}
+\begin{pmatrix}   w(t)+F^2(t)\\  0
 \end{pmatrix}
 +\int_0^te^{(t-s)\mathcal{L}}
\begin{pmatrix}  \mathfrak{N}(s)\\   0
 \end{pmatrix}\,ds.
\end{equation}

 \subsection{Unified estimates}

We now define the norm $\Lvert3\cdot\Rvert3_{0}$ appearing in the Theorems \ref{maintheorem} and \ref{maintheorem2} by
\begin{equation}\label{norm1}
\Lvert3 \begin{pmatrix}     u\\\eta
 \end{pmatrix} \Rvert3_{0} =\sqrt{ \| u\|_0^2+\| \eta\|_{0}^2 }.
 \end{equation}
We define the term $\Lvert3\cdot\Rvert3_{00}$ appearing in the Theorem \ref{maintheorem} by
\begin{equation}\label{norm2}
\Lvert3 \begin{pmatrix}     u\\\eta
 \end{pmatrix} \Rvert3_{00} := \sqrt{\mathcal{E}},
 \end{equation}
where $\mathcal{E}$ is defined by \eqref{energy00}. We define the term $\Lvert3\cdot\Rvert3_{00}$ appearing in the Theorem \ref{maintheorem2} by
\begin{equation}\label{norm3}
\Lvert3 \begin{pmatrix}     u\\\eta
 \end{pmatrix} \Rvert3_{00} := \sqrt{\mathcal{E}_{2N}+\mathcal{F}_{2N}}
\end{equation}
for an integer $N\ge 3$, where $\mathcal{E}_{2N}$ and $\mathcal{F}_{2N}$ are given by \eqref{E_2N} and \eqref{F_2N}, respectively.  We also define another norm $\norm{\cdot}$ by
\begin{equation}\label{norm4}
\left\|\begin{pmatrix}     u\\\eta
 \end{pmatrix} \right\|:= \sqrt{\| u\|_2^2+\| \eta\|_0^2}.
 \end{equation}

\begin{Remark}\label{norm_remark}
Note that the expression for $\Lvert3 \begin{pmatrix}     u\\\eta \end{pmatrix} \Rvert3_{00}$ also involves sums of norms of $p$ and its temporal derivatives, and as such is not actually a  norm for $u$ and $\eta$.   If we were to view $\Lvert3 \cdot \Rvert3_{00}$ as acting  on triples $(u,p,\eta)$ in the obvious way, then this quantity would function as an actual norm.  Here we abuse notation in this manner in order to suppress the appearance of $p$, which we want to view as being determined by $u$ and $\eta$ in a nonlinear fashion.  The reader troubled by this abuse of notation could simply replace the notation $\Lvert3 \begin{pmatrix}     u\\\eta \end{pmatrix} \Rvert3_{00}$ by a nonlinear functional $\mathfrak{E}(u,\eta)$ defined in the same manner.
\end{Remark}

We now restate the main results of Sections \ref{linear instability} and \ref{energy} in our new notation.  Note that our notation has allowed us to unify the cases  $\sigma_\pm=0$ and $\sigma_\pm>0$.

\begin{Proposition}\label{prop1}
 Suppose that $(u, p, \eta)$ is the solution to \eqref{surface} with surface tension or is the solution to \eqref{nosurface2} without surface tension. Let the terms $\Lvert3\cdot\Rvert3_{0}$, $\Lvert3\cdot\Rvert3_{00}$ and $\norm{\cdot}$ be defined by \eqref{norm1}--\eqref{norm4}. Then we have the following.
\begin{enumerate}
\item Let $\Lambda$ be defined by \eqref{Lambda}.  Then we have
\begin{equation}\label{grownth}
\Lvert3 e^{t\mathcal{L}}\begin{pmatrix}     u(0)\\\eta(0)
 \end{pmatrix} \Rvert3_{0}^2 \le C e^{2\Lambda
t} \left\|  \begin{pmatrix}     u(0)\\\eta(0)
 \end{pmatrix} \right\|^2.
\end{equation}
\item There is a growing mode $\begin{pmatrix}     u_\star\\\eta_\star
 \end{pmatrix} $ satisfying $\Lvert3 \begin{pmatrix}     u_\star\\\eta_\star
 \end{pmatrix}\Rvert3_{0}= 1$,
$\Lvert3\begin{pmatrix}     u_\star\\\eta_\star
 \end{pmatrix}\Rvert3_{00}<\infty$, and
\begin{equation}\label{growing}
 e^{Lt} \begin{pmatrix}     u_\star\\\eta_\star
 \end{pmatrix}=e^{\lambda t}\begin{pmatrix}     u_\star\\\eta_\star
 \end{pmatrix}.
 \end{equation}
\item There exists a small constant $\delta$ such that  if $\Lvert3\begin{pmatrix}     u(t)\\\eta(t)
 \end{pmatrix}  \Rvert3_{00}\le\delta$ for all $t \in [0,T]$, then there exists $C_\delta> 0$ so
that the following  energy estimate holds for $t \in [0,T]$:
\begin{equation}\label{energyes}
\begin{split} \Lvert3\begin{pmatrix}     u(t)\\\eta(t)
 \end{pmatrix}  \Rvert3_{00}^2&\le  C_\delta\Lvert3\begin{pmatrix}     u(0)\\\eta(0)
 \end{pmatrix}  \Rvert3_{00}^2
   +  \frac{\lambda}{2}\int_0^t\Lvert3\begin{pmatrix}     u(t)\\\eta(t)
 \end{pmatrix}  \Rvert3_{00}^2\,ds \\&\quad+   C_\delta\int_0^t\Lvert3\begin{pmatrix}     u(t)\\\eta(t)
 \end{pmatrix}  \Rvert3_{00}^3\,ds
   + C_\delta\int_0^t \Lvert3\begin{pmatrix}     u(t)\\\eta(t)
 \end{pmatrix} \Rvert3_{0}^2\,ds.
 \end{split}
 \end{equation}
\end{enumerate}
Here, in the statements of $2$ and $3$, $\lambda>0$ is a number satisfying $\frac{\Lambda}{2}<\lambda\le \Lambda $.
 \end{Proposition}
\begin{proof}
Statement 1 follows from Theorem \ref{lineargrownth}. Statement 2 follows from Theorem \ref{growingmode}.  Statement 3 follows from Theorem \ref{energywith} by taking $\lambda=\Lambda$ for the case with surface tension and from Theorem \ref{energywithoutout} by taking $\lambda=\Lambda_\ast$ for the case without surface tension.
\end{proof}

Now we  will state an estimate for the solution to \eqref{surfaceLP}.

\begin{Proposition} \label{lineargrownth2}
Suppose that  $(u,p, \eta)$ solve \eqref{surfaceLP} and that $\bar{u}=u-w-F^2$ with $w$ constructed in Section \ref{unified operator}. Then we have the following estimates: for both $\sigma_\pm=0$ and $\sigma_\pm>0$, there exists a small constant $\delta$ such that  if $\Lvert3\begin{pmatrix}     u(t)\\\eta(t)  \end{pmatrix}  \Rvert3_{00}\le\delta$ for all  $t \in [0,T]$, then
\begin{equation}\label{lgth}
\begin{split}
\Lvert3\begin{pmatrix}   u(t)\\\eta(t)
 \end{pmatrix} - e^{t\mathcal{L}}\begin{pmatrix}    \bar{u}(0)\\\eta(0)
 \end{pmatrix} \Rvert3_{0}   \le C
\Lvert3\begin{pmatrix}   u(t)\\\eta(t)
 \end{pmatrix}\Rvert3_{00}^2 +C\int_0^te^{\Lambda(t-s) }\Lvert3\begin{pmatrix}   u(s)\\\eta(s)
 \end{pmatrix}\Rvert3_{00}^2 ds
\end{split}
\end{equation}
for all $t\in[0,T]$.
\end{Proposition}

\begin{proof}
From the  formula \eqref{duhamel} and \eqref{grownth}, we have
\begin{equation}\label{es0}
\begin{split}
\Lvert3\begin{pmatrix}   u(t)\\\eta(t)
 \end{pmatrix} - e^{t\mathcal{L}}\begin{pmatrix}    \bar{u}(0)\\\eta(0)
 \end{pmatrix} \Rvert3_{0} &\lesssim \Lvert3\begin{pmatrix}w(t)+F^2(t)\\   0
 \end{pmatrix} \Rvert3_{0}+\int_0^te^{\Lambda(t-s)
}\left\|  \begin{pmatrix}     \mathfrak{N}(s)\\0
 \end{pmatrix} \right\|\,ds
\\&\lesssim \|w(t)\|_{0}+\|F^2(t)\|_{0}+\int_0^te^{\Lambda(t-s)
} \|\mathfrak{N}(s)\|_2
  \,ds.
\end{split}
\end{equation}
On the other hand, by the construction of $\mathfrak{N}$ and $w$ in Section \ref{unified operator}, we can estimate
\begin{equation}\label{es2}
\begin{split}
\|w\|_0\le \|w\|_2\lesssim \|\tilde{F}^3\|_{1/2}\lesssim \|F^3\|_{1/2}+\|F^2\|_2.
\end{split}
\end{equation}
and
\begin{equation}\label{es1}
\begin{split}
\|\mathfrak{N} \|_2&\lesssim
\|\bar{F}^1\|_2+\|
\nabla\mathcal{H}(\tilde{F}^3_{3,+},-\tilde{F}^3_{3,-})\|_{2}
 \\ &\lesssim \|\tilde{F}^1\|_2+\|\partial_tw\|_2+\|w\|_4
+\| \tilde{F}^3\|_{5/2}  \\ &\lesssim \|F^1\|_2+\|\partial_t
F^2\|_2+\|F^2\|_4+\|\partial_tw\|_2+\|w\|_4+
\|F^3\|_{5/2}+\|\mathbb{D}(F^2)\|_{5/2}  \\&\lesssim
\|F^1\|_2+\|\partial_t
F^2\|_2+\|F^2\|_4+\|F^3\|_{5/2}+\|\partial_t\tilde{F}^3\|_{1/2}+\|\tilde{F}^3\|_{5/2}
 \\&\lesssim \|F^1\|_2+\|\partial_t
F^2\|_2+\|F^2\|_4+\|F^3\|_{5/2}.
\end{split}
\end{equation}

Plugging \eqref{es2}--\eqref{es1} into \eqref{es0}, we obtain
\begin{equation}
\begin{split}&
\Lvert3\begin{pmatrix}   u(t)\\\eta(t)
 \end{pmatrix} - e^{t\mathcal{L}}\begin{pmatrix}
 \bar{u}(0)\\\eta(0)
 \end{pmatrix} \Rvert3_{0}\label{es3}\\&\quad \lesssim \|F^2(t)\|_2+\|F^3(t)\|_{1/2}+\int_0^te^{\Lambda(t-s)
} (\|F^1(s)\|_2+\|\partial_t F^2(s)\|_2+\|F^2(s)\|_4+\|F^3(s)\|_{5/2})
  \,ds.
\end{split}
\end{equation}
In both cases $\sigma_\pm=0$ and $\sigma_\pm>0$, we may estimate the nonlinear terms as we did in  Section \ref{energy}.  Indeed,   we have
\begin{equation}\label{es4}
\|F^2(t)\|_2+\|F^3(t)\|_{1/2}\lesssim \Lvert3\begin{pmatrix}   u(t)\\\eta(t)
 \end{pmatrix}\Rvert3_{00}^2
 \end{equation}
 and
 \begin{equation}\label{es5}
 \|F^1(s)\|_2+\|\partial_t F^2(s)\|_2+\|F^2(s)\|_4+\|F^3(s)\|_{5/2} \lesssim \Lvert3\begin{pmatrix}   u(s)\\\eta(s)
 \end{pmatrix}\Rvert3_{00}^2
\end{equation}
where we recall that  $\Lvert3\cdot\Rvert3_{00}$  is given by \eqref{norm2} when $\sigma_\pm>0$  and \eqref{norm3} when $\sigma_\pm>0$.

Substituting \eqref{es4}--\eqref{es5} into \eqref{es3}, we get \eqref{lgth}.
\end{proof}

Thus far we have not elaborated on the local well-posedness theory for our problem that we developed in \cite{WTK}.  We now want to state a version of it specially tailored to our uniform operator setting.  In the result we will refer to the ``necessary compatibility conditions'' required for the local well-posedness in our energy spaces.  These are cumbersome to write out explicitly.  In the case without surface tension we refer to Theorem 2.1 of \cite{WTK} for their explicit statement.  In the case with surface tension we refer to Theorem 2.5 of \cite{WTK} for a lower-regularity version of the conditions; the conditions in the present case can be deduced similarly as those stated in \cite{WTK}.

\begin{theorem}\label{LWP}
Suppose that the initial data $(u_0,\eta_0)$ satisfying the necessary compatibility conditions.  There exist $\delta_0 , T>0$ so that if
\begin{equation}\label{lwp_01}
\Lvert3 \begin{pmatrix} u_0 \\ \eta_0 \end{pmatrix} \Rvert3_{00} < \delta_0,
\end{equation}
then there exists a unique solution $(u,p,\eta)$ to \eqref{surfaceLP} on $[0,T]$ that satisfies the estimate
\begin{equation}\label{lwp_02}
 \Lvert3 \begin{pmatrix} u(t) \\ \eta(t) \end{pmatrix} \Rvert3_{00} \lesssim \sqrt{1+T} \Lvert3 \begin{pmatrix} u_0 \\ \eta_0 \end{pmatrix} \Rvert3_{00}
\end{equation}
for all $t \in [0,T]$.
\end{theorem}
\begin{proof}
 The result without surface tension follows from Theorem 2.1 of \cite{WTK}.  In the case with surface tension, we proved a  lower-regularity version of this result in Theorem 2.5 of \cite{WTK}.  The higher-regularity version we state here can be proved using a straightforward modification of our method in \cite{WTK}, based on our new a priori energy estimates stated in Theorem \ref{energywith}.
\end{proof}

 \subsection{Unified data analysis}\label{data}

In order to prove our nonlinear instability result, we want to use the linear growing mode solutions constructed in Theorem \ref{growingmode} (cf. \eqref{growing}) to construct small initial data for the nonlinear problem, written in the unified equivalent perturbed operator form \eqref{surfaceLP3}. Due to the nonstandard growth rate estimates stated in Theorem \ref{lineargrownth} (cf. \eqref{grownth}) which force us to derive the nonlinear estimates in the higher-order regularity context (cf. \eqref{energyes}), we cannot simply set the initial data for the nonlinear problem to be a small constant times the linear growing modes. The reason for this is that the initial data for the nonlinear problem must satisfy certain nonlinear compatibility conditions in order for us to guarantee local existence in the  space corresponding to norm $\Lvert3\cdot\Rvert3_{00}$, which the linear growing mode  solutions do not satisfy.

To get around this obstacle, it is noticed that the nonlinear problem   is  slightly perturbed from the linearized problem and so  their compatibility conditions for the small initial data should be  close to each other. In the context of the nonlinear compressible Navier-Stokes-Poisson  problem, Jang and Tice \cite{JT} used the implicit function theorem to produce  a curve of small initial data satisfying the compatibility conditions for the nonlinear problem which are close to the linear growing modes. Such argument is general and is also suitable for our problem,  so we may state this result for our setting without the proof.

\begin{Proposition}\label{intialle}
  Let $u_\star,\eta_\star$ be the growing mode  solution stated in Proposition \ref{prop1} and assume  the normalization $\Lvert3\begin{pmatrix} u_\star\\\eta_\star \end{pmatrix}\Rvert3_0=1$. Then there exists  a number $\iota_0>0$ and a family of initial data
  \begin{equation}\label{initial0}
\begin{pmatrix}
u_0^\iota\\\eta_0^\iota
\end{pmatrix}
=\iota \begin{pmatrix}
u_\star\\\eta_\star
\end{pmatrix}+\iota^2\begin{pmatrix}
\tilde{u}(\iota)\\ \tilde{\eta}(\iota)
\end{pmatrix}
\end{equation}
for $\iota\in [0,\iota_0)$ so that the followings  hold.

1. $u_0^\iota,\eta_0^\iota$ satisfy the nonlinear compatibility conditions required by Theorem \ref{LWP} for a solution to the nonlinear problem \eqref{surfaceLP3} to exist in the norm $\Lvert3\cdot\Rvert3_{00}$.

2. There exists a constant $C>0$ independent of $\iota$ so that
\begin{equation}\label{initial1}
\Lvert3\begin{pmatrix}  \tilde{u}(\iota)\\\tilde{\eta}(\iota)
\end{pmatrix}\Rvert3_0\le\left\|\begin{pmatrix}  \tilde{u}(\iota)\\\tilde{\eta}(\iota)
\end{pmatrix}\right\|\le
C
\end{equation}
and
\begin{equation}\label{initial2}
\Lvert3\begin{pmatrix}
u_0^\iota\\\eta_0^\iota
\end{pmatrix}\Rvert3_{00}^2\le C\iota^2.
\end{equation}
 \end{Proposition}
 \begin{proof}
 See the abstract argument before  Lemma 5.3 of \cite{JT}.
\end{proof}

 \subsection{Proof of main theorems}

With Propositions \ref{prop1}--\ref{lineargrownth2}, Theorem \ref{LWP} and Proposition \ref{intialle} in hand, we can now present together the

\begin{proof}[Proof of Theorems \ref{maintheorem} and \ref{maintheorem2}]
We shall prove the two main theorems together with the understanding that when we write $\lambda$ we take $\lambda=\Lambda$  defined by \eqref{Lambda} for Theorem \ref{maintheorem}, while for  Theorem \ref{maintheorem2} we take $\lambda=\Lambda_\ast$ defined by \eqref{littlelambda}.  In each case, we have that $\frac{\Lambda}{2}<\lambda\le \Lambda $ as stated in Proposition \ref{prop1}.  It is this property that all we need in our unified proof.

First, we restrict that $0<\iota<\iota_0\le \theta_0$, where $\iota_0$ is as small as in Proposition \ref{intialle}  and  the value of $\theta_0$ will be fixed later to be sufficiently small. For $0<\iota\le \iota_0$, we let $u_0^\iota$ and $\eta_0^\iota$ be the initial data given in Proposition \ref{intialle}.  By further restricting $\iota$ we may use \eqref{initial2} to verify that \eqref{lwp_01} holds, which then allows us to use Theorem \ref{LWP} to find  $\begin{pmatrix}  u^\iota \\ \eta^\iota\end{pmatrix}$,  solutions  to the system \eqref{surfaceLP3} with
 \begin{equation}
 \left.\begin{pmatrix}  u^\iota\\\eta^\iota \end{pmatrix}\right|_{t=0}=
\begin{pmatrix}
u_0^\iota\\\eta_0^\iota
\end{pmatrix}
=\iota \begin{pmatrix}
u_\star\\\eta_\star
\end{pmatrix}+\iota^2\begin{pmatrix}
\tilde{u}(\iota)\\ \tilde{\eta}(\iota)
\end{pmatrix}.
\end{equation}
 By \eqref{grownth} and \eqref{initial1}, we can estimate
\begin{equation}\label{ins0}
\begin{split}
 \Lvert3 e^{t\mathcal{L}}\begin{pmatrix} \tilde{u}(\iota)\\\tilde{\eta}(\iota)
 \end{pmatrix}\Rvert3_0  \le Ce^{\Lambda t}\left\|\begin{pmatrix} \tilde{u}(\iota)\\\tilde{\eta}(\iota)\end{pmatrix}\right\| \le C_1
e^{\Lambda t}
\end{split}
\end{equation}
for some constant $C_1>0$ independent of $\iota$.

Let us now fix $\delta>0$ as small as in both Propositions \ref{prop1} and \ref{lineargrownth2}, and let $C_\delta>0$ be the constant appearing in Proposition \ref{prop1} for this fixed choice of $\delta$.   We then define $\tilde{\delta}=\min\{\delta,\frac{\lambda}{2C_\delta}\}$.  Denote
 \begin{equation}
 \begin{split}
 T^\ast=\sup_t\left\{ \Lvert3\begin{pmatrix}  u^\iota(t) \\  \eta^\iota(t)
 \end{pmatrix}\Rvert3_{00}   \le\tilde{\delta}\right\}\hbox{ and }
 T^{\ast\ast}=\sup_t\left\{ \Lvert3\begin{pmatrix}  u^\iota(t) \\  \eta^\iota(t)
 \end{pmatrix}\Rvert3_{0}\le 2 \iota e^{\lambda t} \right\}.
 \end{split}
 \end{equation}
 With $\iota_0$ small enough, \eqref{initial0}--\eqref{initial2} and \eqref{lwp_02} guarantee that $T^\ast$, $T^{\ast\ast}>0$. Then for all $t\le \min\{T^\iota,T^\ast,T^{\ast\ast}\}$, we deduce from the estimate \eqref{energyes} of Proposition \ref{prop1}, the definitions of $T^\ast$ and $T^{\ast\ast}$, and \eqref{initial2} that
 \begin{equation}\label{ins1}
 \begin{split}\Lvert3\begin{pmatrix}    u^\iota(t)\\ \eta^\iota(t)
 \end{pmatrix}\Rvert3_{00}^2
&\le  C_\delta\Lvert3\begin{pmatrix}     u^\iota_0\\\eta^\iota_0
 \end{pmatrix}  \Rvert3_{00}^2
   +  \frac{\lambda}{2}\int_0^t\Lvert3\begin{pmatrix}    u^\iota(s)\\ \eta^\iota(s)
 \end{pmatrix} \Rvert3_{00}^2\,ds \\&\quad+   C_\delta\int_0^t\Lvert3\begin{pmatrix}    u^\iota(s)\\ \eta^\iota(s)
 \end{pmatrix}\Rvert3_{00}^3\,ds
   + C_\delta\int_0^t \Lvert3\begin{pmatrix}    u^\iota(s)\\ \eta^\iota(s)
 \end{pmatrix}\Rvert3_{0}^2\,ds
 \\&\le \left(\frac{\lambda}{2} + \tilde{\delta} C_\delta \right) \int_0^t\Lvert3\begin{pmatrix}    u^\iota(s)\\ \eta^\iota(s)
 \end{pmatrix} \Rvert3_{00}^2\,ds+C_\delta C\iota^2+\frac{C_\delta(2\iota)^2}{2\lambda}e^{2\lambda t}
  \\&\le \lambda\int_0^t\Lvert3\begin{pmatrix}    u^\iota(s)\\ \eta^\iota(s)
 \end{pmatrix} \Rvert3_{00}^2\,ds+C_2 \iota^2e^{2\lambda t}.
\end{split}
\end{equation}
for some constant $C_2>0$ independent of $\iota$. We may view \eqref{ins1} as a differential inequality.  Then Gronwall's lemma implies that
\begin{equation}\label{ins11}
 \begin{split}\Lvert3\begin{pmatrix}    u^\iota(t)\\ \eta^\iota(t)
 \end{pmatrix}\Rvert3_{00}^2
&\le   C_2 \iota^2e^{2\lambda t}+C_2\iota^2e^{\lambda t}\int_0^t \lambda e^{\lambda s}\,ds
 \\&\le   C_2 \iota^2e^{2\lambda t}+ C_2\iota^2 e^{2\lambda t}  =  2C_2\iota^2e^{2\lambda t}.
\end{split}
\end{equation}
We then deduce from Proposition  \ref{lineargrownth2} and \eqref{ins11} that
\begin{equation}\label{ins2}
\begin{split}
 &\Lvert3\begin{pmatrix}     u^\iota(t)\\\eta^\iota(t)
 \end{pmatrix}-\iota e^{\lambda t }\begin{pmatrix}  u_\star \\\eta_\star\end{pmatrix}
  -\iota^2e^{t\mathcal{L}} \begin{pmatrix}  \tilde{u}(\iota)\\\tilde{\eta}(\iota)
 \end{pmatrix}\Rvert3_{0}
   \\&\quad\le C \Lvert3\begin{pmatrix}    u^\iota(s)\\ \eta^\iota(s)
 \end{pmatrix}\Rvert3_{00}^2 +C \int_0^te^{\Lambda(t-s) }\Lvert3\begin{pmatrix}    u^\iota(s)\\ \eta^\iota(s)
 \end{pmatrix} \Rvert3_{00}^2 ds
\\&\quad\le 2C C_2\iota^2 e^{ 2\lambda t} +2CC_2\iota^2e^{\Lambda t }
\int_0^t e^{ (2\lambda-\Lambda) s} ds
 \\&\quad\le 2C C_2\iota^2 e^{ 2\lambda t} +\frac{2CC_2\iota^2}{2\lambda-\Lambda}e^{2\lambda t } :=
C_3\iota^2 e^{ 2\lambda t}.
\end{split}
\end{equation}
Here  we have used the fact that $2\lambda-\Lambda>0$.

Now we claim that
\begin{equation}\label{Tmin}
T^\iota= \min\{T^\iota,T^\ast,T^{\ast\ast}\}
\end{equation}
by fixing $\theta_0$ small enough, namely, setting
 \begin{equation}
 \theta_0=\min\left\{\frac{\tilde{\delta}}{2\sqrt{2C_2}},\frac{1}{2(C_1+C_3)}\right\}.
 \end{equation}
 Indeed, if $T^\ast= \min\{T^\iota,T^\ast,T^{\ast\ast}\}$, then by \eqref{ins11}, we have
  \begin{equation}
  \begin{split}
  \Lvert3\begin{pmatrix}     u^\iota(T^\ast)\\\eta^\iota(T^\ast)
 \end{pmatrix}\Rvert3_{00}
\le    \sqrt{2C_2}\iota e^{ \lambda T^\ast} \le   \sqrt{2C_2}\iota e^{ \lambda
T^\iota}=\sqrt{2C_2}\theta_0\le \frac{\tilde{\delta}}{2} < \tilde{\delta},
\end{split}
\end{equation}
which is a contradiction to the definition of $T^\ast$.  If $T^{\ast\ast} =\min\{T^\iota,T^\ast,T^{\ast\ast}\}$, then by \eqref{ins0}, \eqref{ins2}, and the fact that $\Lambda/2 < \lambda \le\Lambda$, we have that
\begin{equation}
\begin{split}
\Lvert3\begin{pmatrix}     u^\iota(T^{\ast\ast})\\\eta^\iota(T^{\ast\ast})
 \end{pmatrix}\Rvert3_{0}   &\le \iota e^{\lambda T^{\ast\ast} }\Lvert3 \begin{pmatrix}  u_\star\\\eta_\star \end{pmatrix}
  \Rvert3_{0}+\iota^2\Lvert3 e^{T^{\ast\ast}\mathcal{L}} \begin{pmatrix}  \tilde{u}(\iota)\\\tilde{\eta}(\iota)
 \end{pmatrix}\Rvert3_{0}
+ C_3\iota^2 e^{2\lambda T^{\ast\ast}}
\\
&\le \iota e^{\lambda T^{\ast\ast}} +C_1\iota^2 e^{\Lambda T^{\ast\ast}} + C_3\iota^2 e^{2\lambda T^{\ast\ast}}
\\
&\le \iota e^{\lambda T^{\ast\ast}} +C_1\iota^2 e^{2\lambda T^{\ast\ast}} + C_3\iota^2 e^{2\lambda T^{\ast\ast}}
\\
& \le  \iota e^{\lambda T^{\ast\ast}} (1 +C_1\iota e^{ \lambda T^{\iota}} +C_3\iota  e^{ \lambda T^\iota})
 \\&\le \iota e^{\lambda T^{\ast\ast}}(1 +C_1\theta_0
+C_3\theta_0)<2\iota e^{\lambda T^{\ast\ast}},
\end{split}
\end{equation}
which is a contradiction to the definition of $T^{\ast\ast}$.  Hence \eqref{Tmin} must hold, proving the claim.

Now we again use \eqref{ins2} and \eqref{ins0}  to find that
\begin{equation}
\begin{split}
\Lvert3\begin{pmatrix}
 u^\iota(T^\iota) \\ \eta^\iota(T^\iota)
 \end{pmatrix}\Rvert3_{0}
&\ge \Lvert3 \iota e^{\lambda T^\iota }\begin{pmatrix}
u_\star\\\eta_\star
   \end{pmatrix}\Rvert3_{0}-\Lvert3 \iota^2e^{T^\iota\mathcal{L}} \begin{pmatrix}  \tilde{u}(\iota)\\\tilde{\eta}(\iota)
 \end{pmatrix}\Rvert3_{0}
  \\&\quad-\Lvert3\begin{pmatrix}  u^\iota(T^\iota)\\ \eta^\iota(T^\iota)
 \end{pmatrix}-\iota e^{\lambda T^\iota }\begin{pmatrix} u_\star\\  \eta_\star\end{pmatrix}
  -\iota^2e^{T^\iota\mathcal{L}} \begin{pmatrix} \tilde{u}(\iota)\\ \tilde{\eta}(\iota)
 \end{pmatrix}\Rvert3_{0}
  \\&\ge\iota e^{\lambda T^\iota }-\iota^2 C_1 e^{\Lambda T^\iota}-C_3\iota^2 e^{ 2\lambda T^\iota}
  \\&\ge\iota e^{\lambda T^\iota }-\iota^2 C_1 e^{2\lambda T^\iota}-C_3\iota^2 e^{ 2\lambda T^\iota}
  \\&\ge\theta_0- C_1\theta_0^2- C_3\theta_0^2 \ge
 \frac{\theta_0}{2}.
\end{split}
\end{equation}
This completes the proof of Theorem \ref{maintheorem} and Theorem \ref{maintheorem2} in the unified setting.
\end{proof}

\appendix

\section{Analytic tools}\label{section_appendix}

\subsection{Poisson extension}

We will now define the appropriate Poisson integrals that allow us to extend $\eta_\pm$, defined on the surfaces $\Sigma_\pm$, to functions defined on $\Omega$, with ``good'' boundedness.

Suppose that $\Sigma_+ = \mathrm{T}^2\times \{1\}$, where $\mathrm{T}^2:=(2\pi L_1 \mathbb{T}) \times (2\pi L_2 \mathbb{T})$. We define the Poisson integral in $\mathrm{T}^2 \times (-\infty,1)$ by
\begin{equation}\label{P-1def}
\mathcal{P}_{-,1}f(x) = \sum_{\xi \in    (L_1^{-1} \mathbb{Z}) \times
(L_2^{-1} \mathbb{Z}) }  \frac{e^{i \xi \cdot x' }}{2\pi \sqrt{L_1 L_2}} e^{|\xi|(x_3-1)} \hat{f}(\xi),
\end{equation}
where for $\xi \in  (L_1^{-1} \mathbb{Z}) \times (L_2^{-1} \mathbb{Z})$ we have written
\begin{equation}
 \hat{f}(\xi) = \int_{\mathrm{T}^2} f(x')  \frac{e^{- i \xi \cdot x' }}{2\pi \sqrt{L_1 L_2}} dx'.
\end{equation}
Here ``$-$'' stands for extending downward and ``$1$'' stands for extending at $x_3=1$, etc. It is well-known that $\mathcal{P}_{-,1}:H^{s}(\Sigma_+) \rightarrow H^{s+1/2}(\mathrm{T}^2 \times (-\infty,1))$ is a bounded linear operator for $s>0$. However, if restricted to the domain $\Omega$, we can have the following improvements.

\begin{lemma}\label{Poi}
Let $\mathcal{P}_{-,1}f$ be the Poisson integral of a function $f$ that is either in $\dot{H}^{q}(\Sigma_+)$ or
$\dot{H}^{q-1/2}(\Sigma_+)$ for $q \in \mathbb{N}=\{0,1,2,\dots\}$, where we have written $\dot{H}^s(\Sigma_+)$ for the homogeneous Sobolev space of order $s$.  Then
\begin{equation}
 \|\nabla^q \mathcal{P}_{-,1}f \|_{0} \lesssim \|f\|_{\dot{H}^{q-1/2}(\mathrm{T}^2)}^2 \text{ and }
 \|\nabla^q \mathcal{P}_{-,1}f \|_{0} \lesssim \|f\|_{\dot{H}^{q}(\mathrm{T}^2)}^2.
\end{equation}
\end{lemma}
\begin{proof}
 See Lemma A.3 of \cite{GT_per}.
\end{proof}

We extend $\eta_+$ to be defined on $\Omega$ by
\begin{equation}\label{P+def}
\bar{\eta}_+(x',x_3)=\mathcal{P}_+\eta_+(x',x_3):=\mathcal{P}_{-,1}\eta_+(x',x_3),\text{ for } x_3\le 1.
\end{equation}
Then Lemma \ref{Poi} implies in particular that if $\eta_+\in H^{s-1/2}(\Sigma_+)$ for $s\ge 0$, then $\bar{\eta}_+\in H^{s}(\Omega)$.

Similarly, for $\Sigma_- = \mathrm{T}^2\times \{0\}$ we define the Poisson integral in $\mathrm{T}^2 \times (-\infty,0)$ by
\begin{equation}\label{P-0def}
\mathcal{P}_{-,0}f(x) = \sum_{\xi \in    (L_1^{-1} \mathbb{Z}) \times (L_2^{-1} \mathbb{Z}) }  \frac{e^{ i \xi \cdot x' }}{2\pi \sqrt{L_1 L_2}} e^{ |\xi|x_3} \hat{f}(\xi).
\end{equation}
It is clear that $\mathcal{P}_{-,0}$ has the  same regularity properties as $\mathcal{P}_{-,1}$. This allows us to extend $\eta_-$ to be defined on $\Omega_-$. However, we do not extend $\eta_-$ to the upper domain $\Omega_+$ by the reflection  since this will result in the discontinuity of the partial derivatives in $x_3$ of the extension. For our purposes,  we instead to do the extension through the following. Let $0<\lambda_0<\lambda_1<\cdots<\lambda_m<\infty$ for $m\in \mathbb{N}$ and define the $(m+1) \times (m+1)$ Vandermonde matrix $V(\lambda_0,\lambda_1,\dots,\lambda_m)$ by $V(\lambda_0,\lambda_1,\dots,\lambda_m)_{ij} = (-\lambda_j)^i$ for $i,j=0,\dotsc,m$.  It is well-known that the Vandermonde matrices are invertible, so we are free to let $\alpha=(\alpha_0,\alpha_1,\dots,\alpha_m)^T$ be the solution to
\begin{equation}\label{Veq}
V(\lambda_0,\lambda_1,\dots,\lambda_m)\,\alpha=q_m,
\end{equation}
$q_m=(1,1,\dots,1)^T$.  Now we define the specialized Poisson integral in $\mathrm{T}^2 \times (0,\infty)$
by
\begin{equation}\label{P+0def}
\mathcal{P}_{+,0}f(x) = \sum_{\xi \in    (L_1^{-1} \mathbb{Z}) \times
(L_2^{-1} \mathbb{Z}) }  \frac{e^{ i \xi \cdot x' }}{2\pi \sqrt{L_1 L_2}}  \sum_{j=0}^m\alpha_j
e^{- |\xi|\lambda_jx_3} \hat{f}(\xi).
\end{equation}
It is easy to check that, due to \eqref{Veq}, $\partial_3^l\mathcal{P}_{+,0}f(x',0)=
\partial_3^l\mathcal{P}_{-,0}f(x',0)$  for all $0\le l\le m$ and hence
\begin{equation}
\partial^\alpha\mathcal{P}_{+,0}f(x',0)=
\partial^\alpha\mathcal{P}_{-,0}f(x',0) \, \forall\, \alpha\in \mathbb{N}^3 \text{ with }0\le |\alpha|\le m.\end{equation}
These facts allow us to  extend $\eta_-$ to be defined on $\Omega$
by
\begin{equation}\bar{\eta}_-(x',x_3)=
\mathcal{P}_-\eta_-(x',x_3):=\left\{\begin{array}{lll}\mathcal{P}_{+,0}\eta_-(x',x_3),\quad
x_3> 0 \\
\mathcal{P}_{-,0}\eta_-(x',x_3),\quad x_3\le
0.\end{array}\right.\label{P-def}\end{equation} It is clear now that if $\eta_-\in H^{s-1/2}(\Sigma_-)$ for $ 0\le s\le m$, then $\bar{\eta}_-\in H^{s}(\Omega)$.  Since we will only work with $s$ lying in a finite interval, we may assume that $m$ is sufficiently large in \eqref{Veq} for $\bar{\eta}_- \in H^s(\Omega)$ for all $s$ in the interval.

\subsection{Some inequalities}

We will need some estimates of the product of functions in Sobolev spaces.

\begin{lemma}\label{sobolev}
Let $U$ denote a domain either of the form $\Omega_\pm$ or of the form $\Sigma_\pm$.
\begin{enumerate}
 \item Let $0\le r \le s_1 \le s_2$ be such that  $s_1 > n/2$.  Let $f\in H^{s_1}(U)$, $g\in H^{s_2}(U)$.  Then $fg \in H^r(U)$ and
\begin{equation}\label{i_s_p_01}
 \norm{fg}_{H^r} \lesssim \norm{f}_{H^{s_1}} \norm{g}_{H^{s_2}}.
\end{equation}

\item Let $0\le r \le s_1 \le s_2$ be such that  $s_2 >r+ n/2$.  Let $f\in H^{s_1}(U)$, $g\in H^{s_2}(U)$.  Then $fg \in H^r(U)$ and
\begin{equation}\label{i_s_p_02}
 \norm{fg}_{H^r} \lesssim \norm{f}_{H^{s_1}} \norm{g}_{H^{s_2}}.
\end{equation}

\item Let $0\le r \le s_1 \le s_2$ be such that  $s_2 >r+ n/2$. Let $f \in H^{-r}(\Sigma),$ $g \in H^{s_2}(\Sigma)$.  Then $fg \in H^{-s_1}(\Sigma)$ and
\begin{equation}\label{i_s_p_03}
 \norm{fg}_{-s_1} \ls \norm{f}_{-r} \norm{g}_{s_2}.
\end{equation}
\end{enumerate}
\end{lemma}
\begin{proof}
These results are standard and may be derived, for example, by use of the Fourier characterization of the $H^s$ spaces.
\end{proof}

The estimates in $(2)$ and $(3)$ in Lemma \ref{sobolev} above are not the best possible estimates of that form.  We now record one specific improvement that we will use for products in $\H(\Omega)$.

\begin{lemma}\label{products}
Suppose that $f \in H^{s}(\Omega)$ with $s > 3/2$ and $g \in \H(\Omega)$.  Then $f g\in \H(\Omega)$ with $\|fg\|_{1} \lesssim \|f\|_{s} \|g\|_{1}$.  Moreover, if $g \in (\H(\Omega))^\ast$ then $fg \in (\H(\Omega))^\ast$ via $\langle fg,\varphi\rangle_{\ast} := \langle g,f\varphi\rangle_{\ast}$.  In this case we have the estimate $\|fg\|_{ \Hd } \lesssim \|f\|_{s} \|g\|_{\Hd}$.
\end{lemma}
\begin{proof}
See Lemma A.3 of \cite{WTK}.
\end{proof}

We will also use the following lemma.

\begin{lemma}\label{-1norm}
The following hold.
\begin{enumerate}
\item Let $f\in L^2(\Omega),\ g\in H^1(\Omega)$, then  $fg\in (\H(\Omega))^\ast$ and
\begin{equation}\label{i_s_p_06}
\|fg\|_{(\H(\Omega))^\ast}\lesssim \|f\|_0\|g\|_{1}.
\end{equation}

\item Let $f\in L^2(\Sigma),\ g\in H^{1/2}(\Sigma)$, then  $fg\in H^{-1/2}(\Sigma)$ and
\begin{equation}\label{i_s_p_07}
\|fg\|_{-1/2}\lesssim \|f\|_0\|g\|_{0}.
\end{equation}

\end{enumerate}
\end{lemma}
\begin{proof}
See Lemma A.4 of \cite{WTK}.
\end{proof}

In the following lemma we let $U$ denote a periodic domain of the form $\Omega_\pm$ with flat upper boundary $\Gamma_{u}$ and lower boundary $\Gamma_{l}$, which may not be flat.  We now record some Poincar\'e-type inequalities for such domains.

\begin{lemma}\label{poincare}
 The following hold.
\begin{enumerate}
 \item $\|f\|_{L^2(U)}^2 \lesssim \|f\|_{L^2(\Gamma_u)}^2 +  \|\partial_3f\|_{L^2(U)}^2$ for all $f \in H^1(U)$.

 \item $\|f\|_{L^2(\Gamma_u)} \lesssim \|\partial_3f\|_{L^2(U)}$ for $f \in H^1(U)$ so that $f =0$ on $\Gamma_l$.

 \item $\|f\|_{0} \lesssim \|f\|_{1}\lesssim \|\nabla f\|_{0}$ for all $f \in H^1(U)$ so that $f=0$ on $\Gamma_l$.
\end{enumerate}
\end{lemma}
\begin{proof}
See Appendix A.4 of \cite{GT_per}.
\end{proof}

We will need the following version of Korn's inequality.

\begin{lemma}\label{korn}
It holds that $\|u\|_{1} \lesssim \|\mathbb{D}u \|_{0}$ for all $u \in \H(\Omega)$.
\end{lemma}
\begin{proof}
 See Lemma 2.7 of \cite{B1}.
\end{proof}

 \subsection{Elliptic estimates}

Considering the  two-phase stationary Stokes problem
\begin{equation}\label{cS}\left\{\begin{array}{ll}
-\mu \Delta u +\nabla {p} =F^1  &\hbox{ in }\Omega
\\ \diverge{u} =F^2    &\hbox{ in}\ \Omega\\
(p_+I-\mu_+\mathbb{D}(u_+))e_3=F^3_+   &\hbox{ on }\Sigma_+
\\
 \Lbrack u\Rbrack=0,\quad \Lbrack(pI-\mu\mathbb{D}(u))e_3\Rbrack=-F^3_-  &\hbox{ on }\Sigma_-
\\  u_-=0 &\hbox{ on }\Sigma_b,\end{array}\right.
\end{equation}
we have the following elliptic regularity theory.
\begin{lemma}\label{cStheorem}
Let $r\ge 2$.  If $F^1\in \ddot{H}^{r-2}(G),\ F^2\in
\ddot{H}^{r-1}(G),\ F^3\in  {H}^{r-3/2}(\Gamma)$, then the problem
\eqref{cS} admits a unique strong solution $(u,p)\in ({}_0H^1(G)\cap
\ddot{H}^r(G))\times \ddot{H}^{r-1}(G)$.
 Moreover,\begin{equation}
 \label{cSresult}\|u\|_{r}+\| p\|_{r-1} \lesssim \|F^1\|_{r-2}+\|F^2\|_{r-1}+\|F^3\|_{r-3/2}.
 \end{equation}
\end{lemma}
\begin{proof}
 See \cite[Theorem 3.1]{WTK}. The proof follows by  using the flatness of the interface $\Sigma_-$ and applying Lemma
\ref{cS1phaselemma2} below.
\end{proof}

We let $G$ denote a horizontal periodic slab and write $\Gamma$ for
its boundary (not necessarily flat),  consisting of two smooth
pieces $\Gamma_1,\ \Gamma_2$.  We shall recall the classical
regularity theory for the
  Stokes problem
with Dirichlet boundary conditions on both $\Gamma_1$ and
$\Gamma_2$.
\begin{eqnarray}\label{cS1phase2}
  \left\{\begin{array}{lll}-\mu \Delta u +\nabla p =f  \quad &\hbox{in }
  G
  \\   \diverge{u} =h  \quad  &\hbox{in } G
\\ u=\varphi_1\quad &\hbox{on }\Gamma_1
\\ u=\varphi_2\quad &\hbox{on }\Gamma_2.\end{array}\right.
\end{eqnarray}
The following records the regularity theory for this problem.

\begin{lemma}\label{cS1phaselemma2}Let $r\ge 2$.  Let $f\in H^{r-2}(G),\ h\in H^{r-1}(G),\ \varphi_1\in H^{r-1/2}(\Gamma_1), \ \varphi_2\in H^{r-1/2}(\Gamma_2)$ be given such that
\begin{eqnarray}\label{cS1phasecomp}\int_G h =\int_{\Gamma_1} \varphi_1\cdot\nu +\int_{\Gamma_2} \varphi_2\cdot\nu ,\end{eqnarray}
then there exists unique $u\in H^r(G),\  p\in H^{r-1}(G)$(up to
constants) solving \eqref{cS1phase2}. Moreover,
\begin{equation}\|u\|_{H^r(G)}+\|\nabla p\|_{H^{r-2}(G)}
\lesssim\|f\|_{H^{r-2}(G)}+\|h\|_{H^{r-1}(G)}+\|\varphi_1\|_{H^{r-1/2}(\Gamma_1)}+\|\varphi_2\|_{H^{r-1/2}(\Gamma_2)}.\end{equation}
\end{lemma}
\begin{proof}
 See \cite{L,T}.
\end{proof}

\subsection{Pressure estimate}

 By the following lemma, the pressure can be viewed as a Lagrange multiplier, which in turn gives an $L^2$ estimate for the pressure.

 \begin{lemma}\label{Pressure}
If $\Lambda \in ( {}_0H^1(\Omega))^\ast$ is such that $\Lambda (v) =
0$ for all $v \in   {}_0H_\sigma^1(\Omega)$, then there exists a
unique $p  \in L^2(\Omega)$ so that
\begin{equation}
(p , \diverge  v) = \Lambda (v) \text{ for all } v\in
{}_0H^1(\Omega)
\end{equation}
and we have the estimate
\begin{equation}\|p\|_{0} \lesssim
\|\Lambda\|_{( {}_0H^1(\Omega))^\ast}.\end{equation}
\end{lemma}
\begin{proof}
See \cite{SS}.
\end{proof}

\subsection{Commutator estimate}

Let $\mathcal{J}= {(1-\Delta)}^{1/2}$ with $\Delta$  the Laplace operator on $\mathrm{T}^n$, $n\ge 1$.  We define the commutator
\begin{equation}
[\mathcal{J}^{s},f]g=\mathcal{J}^s(fg)-f\mathcal{J}^sg.
\end{equation}

We recall the following commutator estimate:
 \begin{lemma}\label{commutator}
\begin{equation}
\|[\mathcal{J}^{s},f]g\|_{L^2}\le \|\nabla f\|_{L^\infty}\|\mathcal{J}^{s-1}g\|_{L^2}
+ \|\mathcal{J}^s f\|_{L^2}\| g\|_{L^\infty}.
\end{equation}
\end{lemma}
\begin{proof}
See \cite[Lemma X1]{KP} for the case $\mathbb{R}^n$.  The case of $\mathrm{T}^n$ can be handled similarly.
\end{proof}

\section*{Acknowledgement}

We would like to thank Yan Guo for his attention on this subject and for his constant encouragement.  Part of this work was completed while I. Tice was visiting Peking University.  We would like to thank them for their hospitality.  Y. J. Wang would like to express his gratitude for the hospitality of the Division of Applied Mathematics at Brown University during his visit, where this work was initiated.

\end{document}